\documentclass[12pt]{article}
\usepackage{latexsym}
\usepackage{amsmath, amssymb, amsthm,amsfonts}
\usepackage{graphicx}
\usepackage{pdflscape}
\usepackage{rotating}
\usepackage{breqn}
\usepackage{geometry}
\usepackage{enumitem}
\usepackage{booktabs}
\usepackage{array}
\usepackage{placeins}
\usepackage{hyperref}
\hypersetup{colorlinks=true,linkcolor=blue,citecolor=blue,urlcolor=blue}

\usepackage{color}

\newif\ifshowchanges
\showchangesfalse
\newcommand{\chg}[1]{\ifshowchanges{\color{blue}#1}\else#1\fi}

\newtheorem{theorem}{Theorem}[section]
\newtheorem{lemma}{Lemma}[section]
\newtheorem{proposition}{Proposition}[section]
\newtheorem{remark}{Remark}[section]
\newtheorem{corollary}{Corollary}[section]

\title{Direct Sample-Average Learning for Queueing Control:\\
Gradient-Free Top-$K$ Coarse-to-Fine Optimization}
\author{%
\ifshowchanges\color{blue}\fi
Mehrdad Nozohour Yazdi \qquad Yiqiang Q. Zhao\\[0.4em]
\small School of Mathematics and Statistics, Carleton University\\
\small \texttt{mehrdadnozohouryazdi@cmail.carleton.ca}\qquad
\texttt{zhao@math.carleton.ca}%
}
\date{\ifshowchanges\color{blue}\fi October 18, 2026}

\begin{document}
\maketitle

\begin{abstract}
We study steady-state queueing control when decisions must be learned from a
single persistent operating trajectory: only the implemented action produces
data, and the queue is not reset when the action changes. Fixed-event-count
cycles provide a direct ratio-of-totals sample-average approximation (SAA) of
each action's steady-state cost rate. A Top-$K$ coarse-to-fine (CTF) search then
ranks these estimates, retains several promising refinement paths, and
concentrates later operation on finer candidate sets. The method requires no
gradient, convexity, or differentiability.

We decompose direct-SAA error into accumulated conditional transient bias and
martingale fluctuation. For regenerative queues, a bias-function
representation reduces transient effects within contiguous action blocks to
boundary terms. A finite-sample triangular-array bound with
Lyapunov-weighted innovation variance then links stagewise estimation accuracy
to master-grid identification through guardian retention. The queueing and
sampling hypotheses are verified completely for a specified persistent-workload
$M/M/1$ implementation. At a common active set, Top-$K$ has a weakly larger
guardian margin than Top-1, and the implemented hierarchy supplies an
exact-objective example in which Top-1 irreversibly discards the optimal
refinement branch. This structural comparison does not assert uniform
superiority of complete algorithms at equal budgets.

The numerical study delineates scope rather than universal dominance. With
exact workload observations, Top-$K$ gives smaller terminal gaps than the
corrected LiQUAR comparator in the reported samples, whereas LiQUAR has lower
median cumulative learning cost at the large budget. Top-$K$ also improves on
Top-1 and flat SAA in every budget of a new canonical, approximately
clock-matched mixed discrete--continuous staffing experiment. A nonpreemptive
lognormal-service experiment requires much longer horizons, while GOLiQ is
markedly stronger on a smooth pricing--staffing benchmark designed for
queueing sensitivities. Direct ranking is therefore best viewed as a
complement when reliable model-specific gradients are unavailable,
inconvenient, or mismatched to the operational decision space.
\end{abstract}


\section{Introduction}
\label{sec:introduction}

Many queueing-control problems require decisions to be made while the system
is operating and while important model primitives are only partially known. A
service provider may choose a price, staffing level, service capacity, routing
rule, or admission threshold while learning from the same queueing trajectory
generated by those decisions. The long-run target is a steady-state cost
rate, but the observations immediately after a policy change are generally
neither independent nor stationary: they inherit the congestion created by
previous controls.

Two difficulties are coupled in this setting. First, objective values must be
learned from dependent, inherited-state observations rather than from fresh
steady-state samples. Second, a fine action grid may be too large to evaluate
uniformly, so the learner must decide which regions deserve further operating
time. The question addressed here is whether direct cost observations can
support both statistically valid steady-state estimation and an economical
multiresolution search without resetting the queue.

We study a route based on direct objective learning.
For an action $a$, let $f(a)$ denote its steady-state expected cost rate. We
organize operation into cycles containing a fixed number $D$ of designated
events and estimate $f(a)$ from the cumulative realized cost and cumulative
clock time observed whenever $a$ is actually implemented:
\[
\widehat f_n(a)
=
\frac{\text{cumulative realized cost under }a}
     {\text{cumulative operating time under }a}.
\]
No gradient or counterfactual trajectory is constructed. When realized cost
is observable, the learner need not be supplied with a steady-state formula,
a demand curve, or a service-time distribution.

A prominent alternative in online queueing optimization is to estimate or
exploit gradients. Chen, Liu, and Hong \cite{ChenLiuHong2023} develop GOLiQ
for joint pricing and capacity sizing, using queueing sensitivity information
to form a stochastic gradient. Chen, Hong, and Liu
\cite{ChenHongLiu2026} develop LiQUAR, which uses two perturbed controls to
construct a finite-difference gradient when the demand function and service
distribution are unknown. Queueing bandit methods provide another important
route; see, among others,
\cite{KrishnasamySenJohariShakkottai2020,JiaShiShen2022}. These methods can
be highly effective when their gradient, confidence, or regret structure fits
the model. Our aim is not to replace them, but to develop the complementary
case in which direct objective values are observable while reliable
model-specific sensitivity information is unavailable, inconvenient, or
poorly matched to discrete decisions.

Direct estimation alone does not solve the search problem when a fine action
grid is large. We therefore place it inside a Top-$K$ coarse-to-fine search:
\[
\boxed{
\text{evaluate candidates}
\;\longrightarrow\;
\text{rank}
\;\longrightarrow\;
\text{retain Top-}K
\;\longrightarrow\;
\text{refine}.}
\]
The method keeps cumulative cost/time statistics and operates each active
action in contiguous blocks. Retaining several parents is important because
a coarse action need not be attractive for its own current-grid objective
value in order for its descendants to contain the eventual fine-grid
optimizer. We formalize this idea through a \emph{guardian}: a parent whose
descendants contain a master-grid optimizer.

The resulting contribution has two interacting components: a statistical
foundation for direct cost-rate estimation on a persistent trajectory, and a
search mechanism that protects promising refinement regions from irreversible
early deletion. The main statistical difficulty is the persistence of the
queueing state.
For cycle $t$, let $a_t$ be the selected action, $T_t$ its random duration, and
$G_t$ its realized cost. Our analysis starts from
\[
G_t-f(a_t)T_t
=
\text{conditional transient contribution}
+
\text{martingale innovation}.
\]
The transient term need not vanish cycle by cycle. Instead, we ask whether
its accumulated contribution is negligible relative to the total time spent
under an action. For regenerative queues, a bias-function representation
makes these transient terms telescope over contiguous action blocks, leaving
only block-boundary contributions. Thus the learning cycles themselves need
not be regeneration cycles and the queue need not be reset.

This transient issue is not absent from the closest gradient-based literature.
GOLiQ explicitly separates regret caused by nonstationarity from regret caused
by suboptimal decisions and uses increasing cycle lengths to control
transient bias; LiQUAR likewise uses warm-up/overtime treatment in its
performance estimation. Our distinction is therefore not that transient
queueing data are unique to the present paper. Rather, we analyze a different
statistical object: the direct cost-rate SAA itself. The resulting
transient-bias/martingale decomposition supports objective consistency and,
through ranking, hierarchical optimization rather than a gradient update.

The paper makes four main contributions, ordered from estimation to search and
then to empirical scope.
\begin{enumerate}[label=(\roman*),leftmargin=2em]
\item We formulate a direct ratio-of-totals SAA estimator for fixed-$D$
operating cycles on a single persistent queue trajectory and derive an exact
error decomposition into accumulated conditional transient bias and martingale
fluctuation. This yields actionwise and finite-action uniform consistency.

\item We give queueing-specific sufficient conditions for that consistency.
For regenerative queues, bias-function boundary terms telescope within
contiguous blocks of a fixed action. Martingale averaging, selected-cycle
moment conditions, and a common Lyapunov route then provide fixed-$D$
verification without resetting the queue.

\item We develop a fixed-depth Top-$K$ CTF method and its retention theory.
A finite-sample blocked-SAA lemma handles horizon-dependent stage budgets
with Lyapunov-weighted innovation variance. A guardian-separation condition
then bounds irreversible pruning across stages, while a common-active-set
comparison quantifies the margin advantage over Top-1. Convexity,
differentiability, and gradient information are not
assumed; mixed discrete--continuous controls are therefore handled without a
continuous relaxation. A worked $M/M/1$ workload model verifies the
common drift, Poisson representation, cycle-reward moments, and completed-stage
sampling conditions.
For its four-level benchmark hierarchy, an exhaustive audit also verifies
positive guardian margins over every required active set.
The same hierarchy supplies a deterministic example in which exact-objective
Top-1 discards the branch containing the master-grid optimizer while Top-$K$
retains it.

\item We evaluate where this architecture helps and where it does not. On
the LiQUAR benchmark with exact workload observations and corrected
evaluation-domain projection, Top-$K$ has a terminal-policy identification
advantage in the reported samples, while the corrected LiQUAR comparator
has lower median cumulative learning cost at the large budget. In an
integer-staffing $M/M/s$ problem, Top-$K$ has lower mean terminal gaps and
higher optimizer-selection frequencies than Top-1 refinement and flat direct
SAA at all three approximately matched operating-time budgets in a canonical
200-replication experiment. A nonpreemptive
lognormal-service $M/G/s$ extension, with different objective geometry and
retention settings, shows slower finite-horizon stabilization but sharp
terminal identification once the operating horizon is sufficiently long.
Finally, on the smooth GOLiQ benchmark, the specialized
gradient method is substantially more efficient, identifying a regime in which
available sensitivity structure should be exploited.
\end{enumerate}

\subsection{Related literature and positioning}
\label{subsec:related-literature}

Two important recent advances in online queueing optimization, and the closest papers to our setting, are GOLiQ \cite{ChenLiuHong2023} and LiQUAR
\cite{ChenHongLiu2026}. Both operate queues in successive cycles and
explicitly address transient data. Their optimization mechanism is gradient
based: GOLiQ exploits queueing sensitivity formulas, whereas LiQUAR compares
two perturbations around a continuously updated center. The present method
stops one layer earlier: it estimates steady-state objective values directly
and uses ranking and structured refinement. This distinction matters when
controls are discrete or mixed, when the objective is nonsmooth or nonconvex,
or when constructing a reliable model-specific gradient is undesirable.

Queueing bandits address partial feedback and exploration with unknown service
rates, prices, or rewards; examples include
\cite{KrishnasamySenJohariShakkottai2020,JiaShiShen2022}. Our finite-grid
problem also has partial feedback, but its basic observation is a
steady-state cost-rate SAA contaminated by inherited-state transients rather
than an i.i.d. arm reward. Top-$K$ pruning is used to preserve refinement
paths and concentrate later operating time, not to construct a UCB index.

The work is also related to simulation optimization and steady-state
estimation. Single-run steady-state optimization for queues has a long
history; see, for example, \cite{SuriZazanis1988,LEcuyerGlynn1994}.
Regenerative simulation \cite{CraneIglehart1975,GlynnIglehart1993}
motivates our bias-function construction. \chg{Classical sample-average approximation for stochastic discrete optimization is developed in
\cite{KleywegtShapiroHomemdeMello2002}, while extensions to dependent data include
\cite{HomemdeMello2008,WangEtAl2022}.} The distinguishing feature here is that dependence is generated by
an adaptive control sequence and the same action may be revisited in multiple
persistent operating blocks.

The positioning is therefore deliberately integrative. Direct SAA,
regenerative analysis, and multiresolution search are not individually claimed
to be new. The contribution is to connect them on one persistent controlled
trajectory: queueing structure validates the direct estimator, the estimator
drives a Top-$K$ hierarchy, and guardian separation translates ranking accuracy
into protection against irreversible loss of an optimal refinement path. This
connection also identifies the method's natural boundary: when accurate local
sensitivity information is available and the decision space is smooth, a
specialized gradient method may be preferable.

The remainder is organized around the operational method. Section
\ref{sec:model-saa} defines persistent fixed-event-count sampling and the
direct estimator. Section \ref{sec:topk-ctf} gives the finite-budget Top-$K$
procedure before the technical analysis. Section \ref{sec:theory} develops
direct-SAA consistency, queueing verification, and guardian-retention theory.
Section \ref{sec:numerical} reports implementation validation and four
benchmark families. Section \ref{sec:discussion} discusses scope,
limitations, and the distinction between terminal identification and
cumulative regret.

\section{Persistent Queueing Control and Direct SAA}
\label{sec:model-saa}


We begin with a nonempty finite feasible action set, supplied to the learner,
\[
\mathcal A
=
\{a_1,\ldots,a_M\},\qquad M\ge1.
\]
The action set may arise, for example, from a finite discretization of an
underlying continuous decision space. For each $a\in\mathcal A$, assume
that its specified fixed-action dynamics admit a well-defined steady-state
regime with finite expected cost rate $f(a)\in\mathbb R$.
Identifying a feasible set from unknown queueing primitives is a separate
problem; the present search operates within the supplied set. Conditions
controlling the trajectory under adaptive switching are stated separately
in the theoretical results. The finite-action optimization problem is
therefore
\begin{equation}
a_{\mathcal A}^*
\in
\operatorname*{argmin}_{a\in\mathcal A} f(a).
\label{eq:grid-optimum}
\end{equation}

We organize the operating trajectory into successive learning cycles.
Each cycle contains a fixed positive integer number \(D\) of designated
queueing events, with the counting convention fixed before learning.
Depending on the application, these events may be external arrivals,
admitted customers, service entries, service completions, or departures
from the system.  In a queueing network, they may also be completions or
departures at a designated node.  Thus, throughout the present development,
\begin{equation}
D_t\equiv D.
\label{eq:fixed-D}
\end{equation}
Although the number of designated events is fixed, the corresponding
clock-time duration is random and may depend on the action being implemented.
We use a simple, nonexplosive designated-event stream and require that the
$D$th subsequent counted event occurs in positive finite time almost surely
under every admissible action and inherited state used in the analysis.
An action that prevents the counted event from occurring is not admissible
for that cycle convention, even if its stationary cost is finite.
Simultaneous physical events require a specified counting and ordering rule
consistent with these conditions.

The queue is observed along a single persistent trajectory.  By persistent
operation, we mean that the physical system is not reset between learning
cycles: the terminal state of one cycle becomes the initial state of the
next, even when the selected action changes.  Consequently, observations
collected under a given action may begin from states generated partly under
previously implemented actions.

Let $X_t$ denote a complete physical state at the end of cycle $t$, with
$X_0$ the initial state. Thus $X_{t-1}$ is inherited by cycle $t$.
Depending on the model, this state may include residual service requirements,
arrival ages or residual arrival information, and legacy staffing. Queue
length alone need not suffice. The learner is not required to observe every
state component.

Let $\mathcal O_t$ be the learner's observed history through cycle $t$,
including policy randomization available before the next action is selected.
For analysis, use a filtration $\{\mathcal H_t\}_{t\ge0}$ containing
$\mathcal O_t$ and $X_t$. Conditional evolution, moment bounds, and
martingale identities are imposed with respect to this specified analysis
filtration. An enlargement of the observed history must preserve the
conditional laws used in the relevant result; future primitive randomness
is not included merely because a simulator can generate it in advance.
At the beginning of cycle \(t\), an action
\[
a_t\in\mathcal A
\]
is selected using $\mathcal O_{t-1}$. Thus $a_t$ is both
$\mathcal O_{t-1}$- and $\mathcal H_{t-1}$-measurable. The selected action is held fixed
throughout cycle \(t\).

Each cycle excludes its initial boundary event and includes its terminating
designated event. The terminal state is recorded after that event, before
the next action is selected. The model must specify how a new action affects
unfinished work, admissions, residual arrival information, and any immediate
service starts. All fixed-action continuation laws below use this same
inherited-state and switching convention. Persistence alone does not specify
these dynamics. The examples have no separate lump switching fee; introducing
one would require including it exactly once in the cycle cost and verifying
the corresponding bias decomposition or negligible accumulated contribution.

Let
\[
T_t>0
\]
denote the realized clock duration of cycle \(t\), and let
\[
G_t
\]
denote the total realized cost accumulated over cycle \(t\). Both \(T_t\) and
\(G_t\) generally depend on the selected action \(a_t\), on the state
inherited at the beginning of the cycle, and on the stochastic evolution
of the system during the cycle.  We suppress this dependence in the
notation and write simply \(T_t\) and \(G_t\).
For every deterministic cycle index $t$ of the policy under analysis, assume
\[
0<T_t<\infty\quad\hbox{almost surely},\qquad
\mathbb E[T_t+|G_t|]<\infty.
\]
The first-moment requirement is imposed separately for every experiment
index $N$ and finite cycle index $t$ in the stagewise analysis. It concerns
the actual controlled trajectory; the common bounds needed for finite-sample
conclusions are stated later. Since $\max_{a\in\mathcal A}|f(a)|<\infty$,
$G_t-f(a_t)T_t$ is integrable, so its conditional mean and centered innovation
in Section~\ref{subsec:bias-decomp} are well defined.

Thus the information chronology is
\[
\mathcal H_{t-1}
\longrightarrow
a_t
\longrightarrow
(T_t,G_t)
\longrightarrow
\mathcal H_t.
\]

For a fixed action \(a\in\mathcal A\), define
\begin{equation}
I_t(a)
:=
\mathbf 1\{a_t=a\},
\label{eq:action-indicator}
\end{equation}
and let
\begin{equation}
S_n(a)
:=
\sum_{t=1}^n I_t(a)T_t
\label{eq:operating-time}
\end{equation}
denote the cumulative clock time for which action \(a\) has been operated
through cycle \(n\), including cycle \(n\). Similarly, define the cumulative observed cost
\begin{equation}
C_n(a)
:=
\sum_{t=1}^n I_t(a)G_t.
\label{eq:cumulative-cost}
\end{equation}

Whenever \(S_n(a)>0\), the direct SAA estimator of the steady-state cost
rate is
\begin{equation}
\widehat f_n(a)
=
\frac{C_n(a)}{S_n(a)}
=
\frac{
\sum_{t=1}^n I_t(a)G_t
}{
\sum_{t=1}^n I_t(a)T_t
}.
\label{eq:direct-saa-estimator}
\end{equation}

The ratio form in \eqref{eq:direct-saa-estimator} is natural because the
objective \(f(a)\) is a long-run cost per unit of clock time, while the
clock durations \(T_t\) of fixed-\(D\) cycles are random.

\begin{remark}[Unknown primitives and imperfect cost observation]
\label{rem:unknown-primitives}
The direct estimator requires realized operating times and realized costs, not
an analytical formula for $f(a)$.  Thus arrival-rate functions, service-time
distributions, and other queueing primitives may be unknown provided their
realized consequences are reflected in $G_t$.  
More generally, suppose the learner observes
\[
\widetilde G_t = G_t + E_t,
\]
instead of \(G_t\), where \(E_t\) denotes the observation or approximation
error in the realized cost for cycle \(t\).
For the infinite-trajectory results, suppose that, for every action
under consideration,
\[
\frac{\sum_{t=1}^n I_t(a)E_t}{S_n(a)}
\longrightarrow0
\qquad\text{almost surely}.
\]
Whenever the hypotheses of the corresponding almost-sure consistency
result hold for the actual selection policy, its conclusion then also
holds for the estimator formed from $\widetilde G_t$. All sums and
hypotheses refer to the trajectory generated by that policy, including
its use of imperfect observations. The increasing-budget stagewise
results require the separate condition in
Remark~\ref{rem:stagewise-observation-error}.
\end{remark}


\paragraph{From direct estimation to search.}
For a small finite set, one can simply sample every action deeply and return
an empirical leader. On a fine grid this becomes expensive because each
action requires actual queue operation. The next section therefore uses the
same estimator as a ranking primitive inside a fixed-depth refinement
hierarchy. The practical algorithm requires only cumulative costs, operating
times, and comparisons among active actions.

\section{Top-$K$ Coarse-to-Fine Learning}
\label{sec:topk-ctf}

A direct objective estimator is useful only if it can be embedded in an
efficient search. The proposed CTF architecture is intentionally simple:
active controls are operated in blocks, ranked by direct SAA, pruned, and
refined. Earlier observations of a surviving control are retained rather
than discarded.

\paragraph{Why Top-$K$?}
The Top-$K$ rule is intended as a middle ground between exhaustive refinement
and greedy Top-1 refinement. Refining every active action avoids irreversible
pruning but can require operating a rapidly growing number of fine-grid
controls on the persistent queue trajectory. At the other extreme, retaining
only the current empirical best action is inexpensive but can discard a
coarse action whose descendants contain the eventual fine-grid optimizer.
Top-$K$ retention reduces the candidate set while preserving several promising
refinement paths. It also fits naturally with direct SAA: only objective
estimates and rankings are required, without confidence bounds, gradients, or
additional stochastic-approximation machinery. The guardian analysis in
Section~\ref{sec:theory} formalizes when the retained set is large enough to
preserve a path to a master-grid optimizer.

\subsection{Finite-budget Top-K CTF procedure}
\label{subsec:practical-topk}

Let $\mathcal U_0,\ldots,\mathcal U_L$ be finite nested candidate grids
contained in $\mathcal A$, and put
$\mathcal U:=\bigcup_{\ell=0}^L\mathcal U_\ell\subseteq\mathcal A$.
The procedure samples only actions in $\mathcal U$ and targets the best
action on the master grid $\mathcal U_L$. Let
$\operatorname{Ch}_\ell(a)\subseteq\mathcal U_{\ell+1}$ denote the children
of $a\in\mathcal U_\ell$.  We assume $\mathcal U_0$ and $\mathcal U_L$ are nonempty and that
$\operatorname{Ch}_\ell(a)\neq\varnothing$ for every
$a\in\mathcal U_\ell$ and $\ell<L$.  Thus a nonempty retained set generates a nonempty next
active set and the final empirical argmin is well defined.
Here $N\in\mathbb N$ indexes a prescribed family of deterministic sampling
schedules; it is not itself a realized cycle, event, or clock-time count.
Choose a fixed event count $D$, positive integer cumulative per-action stage targets
$m_{0,N}\leq\cdots\leq m_{L,N}$, and retention counts
$K_0,\ldots,K_{L-1}\ge1$. At a nonempty active set $A$, retain
$k_\ell(A):=\min\{K_\ell,|A|\}$ parents, so the rule also covers stages
with fewer than $K_\ell$ candidates.  The following blocked implementation is used in the
numerical study and is the finite-budget counterpart of the asymptotic theory
below. In the increasing-budget limit, each fixed stage target tends to
infinity as $N\to\infty$. For example,
$m_{\ell,N}=\lceil Nw_\ell\rceil$ with fixed
$0<w_0\le\cdots\le w_L$ is an admissible scaling; finite experimental
target vectors can be specified directly.

\begin{enumerate}[label=\textbf{Step \arabic*:},leftmargin=2.2cm]
\item Set the initial active set $A_{0,N}=\mathcal U_0$ and initialize, for
every action, cumulative cycle count, operating time, and observed cost to zero.

\item At stage $\ell$, operate each $a\in A_{\ell,N}$ in one contiguous block
until its \emph{cumulative} number of fixed-$D$ cycles reaches
$m_{\ell,N}$.  If $a$ was sampled at an earlier stage, retain and reuse all of
its previous statistics.  The order of the blocks may be fixed or randomized
using only information available before the corresponding block is started.

\item Compute
\[
\widehat f_{\ell,N}(a)
=
\frac{C_{\ell,N}(a)}{S_{\ell,N}(a)},
\qquad a\in A_{\ell,N},
\]
and rank the active actions, using a fixed tie-breaking rule.

\item If $\ell<L$, retain the empirical top $k_\ell(A_{\ell,N})$ parents,
$\widehat{\mathcal R}_{\ell,N}$, and set
\[
A_{\ell+1,N}
=
\bigcup_{a\in\widehat{\mathcal R}_{\ell,N}}
\operatorname{Ch}_\ell(a).
\]
At the final level return
\[
\widehat a_N
\in
\operatorname*{argmin}_{a\in A_{L,N}}
\widehat f_{L,N}(a).
\]
\end{enumerate}

The child maps are part of the algorithm specification; grid sizes and
nestedness alone do not determine them. Shared children are included only
once in each active set, and an action that reappears keeps its cumulative
statistics. A claim relative to the whole master grid additionally requires
reachability of a master-grid optimizer, as specified in
Section~\ref{subsec:topk-guardian}.

\paragraph{Cycle, event, and elapsed-time budgets.}
Let $\sigma_{\ell,N}$ be the number of completed primitive cycles at the
stage-$\ell$ ranking. Nondecreasing targets and reuse of observations imply
\[
\sigma_{\ell,N}\le|\mathcal U|m_{\ell,N}.
\]
Together with almost-sure completion of each primitive cycle, this ensures
completion of every prescribed finite schedule. Its designated-event count
is $D\sigma_{L,N}$ and its random elapsed operating time is
\[
H_N:=\sum_{t=1}^{\sigma_{L,N}}T_t^{(N)}.
\]
Here $T_t^{(N)}$ is the duration of cycle $t$ in schedule $N$.
Thus a cycle cap $b_N^{\rm cyc}$ or designated-event cap $b_N^{\rm evt}$
can be respected by choosing targets with
$|\mathcal U|m_{L,N}\le b_N^{\rm cyc}$ or
$D|\mathcal U|m_{L,N}\le b_N^{\rm evt}$, respectively.
These bounds are conservative. Elapsed-time comparisons use the resulting
random $H_N$; a hard clock-time cutoff is a different stopping rule, discussed
after Theorem~\ref{thm:topk-guardian-retention}.

This procedure is gradient free.  Convexity and differentiability of $f$ are
not algorithmic requirements.  The price of this generality is that successful
refinement depends on a ranking-separation condition: an action whose
descendants contain a final optimizer must not be outranked by $K_\ell$ or more
competing non-guardian actions at a pruning stage.

\paragraph{Choosing the retention counts without knowing $f$.}
The algorithm can be run with retention counts specified before observing the
unknown objective: they express the breadth of refinement that the sampling
budget can afford. A budgeted pilot or sensitivity comparison can inform this
choice, provided its operating cost is counted. Such choices do not by
themselves certify the guardian-margin condition. That condition depends on
true parent costs and on which descendants are optimal. Without a relation
between parent values and descendant quality, a gap between observed coarse
estimates cannot certify that a pruned region contains no better descendant.
We therefore distinguish a usable fixed-$K$ algorithm from a guarantee that a
particular $K$ preserves an optimizer. The latter requires the structural
margin hypothesis, an appropriate audit, or additional coarse-to-fine bounds.
Retaining all parents avoids this pruning risk but increases sampling cost.

\paragraph{Operational interpretation of a guardian.}
At a coarse level, a parent represents a region rather than merely its own
current-grid point. We call a parent a guardian when that region contains a
refinement path to a master-grid optimizer. A guardian need not itself be the
best coarse action. Top-$K$ is useful precisely when several attractive
non-guardian parents can outrank the best guardian, either because coarse
discretization distorts descendant quality or because finite-sample SAA
rankings are noisy.

\section{Theoretical Guarantees}
\label{sec:theory}

The analysis has three layers. First, direct-SAA error is separated into
accumulated transient bias and martingale fluctuation. Second, queueing
structure is used to verify that these errors are negligible under persistent
fixed-$D$ operation. Third, stagewise SAA accuracy is converted into
hierarchical optimization through guardian retention. The first two layers
provide almost-sure foundations for direct SAA under persistent sampling;
they are not an almost-sure convergence claim for a finite run that permanently
discards actions. Section~\ref{subsec:stagewise-blocked-saa} instead analyzes
the implemented pruning scheme through finite-sample bounds and a family of
experiments whose stage targets grow across the prescribed schedules.

The principal results can be read in that order. Theorem~
\ref{thm:direct-saa-consistency} establishes consistency of the direct
ratio-of-totals estimator. Theorem~\ref{thm:transient-bias-blocks} explains how
contiguous action blocks reduce regenerative transient effects to boundary
terms. Lemma~\ref{lem:stagewise-blocked-saa} supplies the finite-sample
stagewise ranking bound used by Theorem~\ref{thm:topk-guardian-retention}.
Proposition~\ref{prop:topk-margin-comparison} then isolates the structural
Top-$K$ versus Top-1 margin comparison, and Proposition~
\ref{prop:mm1-complete-verification} closes the assumptions for the specified
$M/M/1$ workload implementation. Thus the section moves from estimator
validity, to queueing verification, to hierarchical decision quality.

\subsection{Direct-SAA consistency on a persistent trajectory}
\subsubsection{Conditional transient-bias decomposition}
\label{subsec:bias-decomp}

For cycle \(t\), define the conditional transient term
\begin{equation}
\beta_t
:=
\mathbb E
\left[
G_t-f(a_t)T_t
\mid
\mathcal H_{t-1}
\right].
\label{eq:conditional-transient}
\end{equation}
Since \(a_t\) is \(\mathcal H_{t-1}\)-measurable,
\begin{equation}
\beta_t
=
\mathbb E[G_t\mid\mathcal H_{t-1}]
-
f(a_t)\mathbb E[T_t\mid\mathcal H_{t-1}].
\label{eq:conditional-transient-interpretation}
\end{equation}
Thus \(\beta_t\) measures the predictable discrepancy between the expected
cost accumulated during the actual finite cycle and the cost predicted by
the steady-state rate \(f(a_t)\) over the corresponding operating time.
Because the initial state of the cycle is inherited from the preceding
system evolution, this discrepancy represents the finite-cycle transient,
or boundary, effect.

Define
\begin{equation}
M_t
:=
G_t-f(a_t)T_t-\beta_t.
\label{eq:martingale-innovation}
\end{equation}
By construction,
\begin{equation}
\mathbb E[M_t\mid\mathcal H_{t-1}]=0.
\label{eq:martingale-difference}
\end{equation}
Hence
\begin{equation}
G_t
=
f(a_t)T_t+\beta_t+M_t.
\label{eq:cycle-decomposition}
\end{equation}

For any fixed \(a\in\mathcal A\), multiplying
\eqref{eq:cycle-decomposition} by \(I_t(a)\) and summing over
\(t=1,\ldots,n\) gives
\[
C_n(a)
=
f(a)S_n(a)
+
\sum_{t=1}^n I_t(a)\beta_t
+
\sum_{t=1}^n I_t(a)M_t.
\]
Therefore, whenever \(S_n(a)>0\),
\begin{equation}
\widehat f_n(a)-f(a)
=
\frac{\sum_{t=1}^n I_t(a)\beta_t}{S_n(a)}
+
\frac{\sum_{t=1}^n I_t(a)M_t}{S_n(a)}.
\label{eq:saa-error-decomposition}
\end{equation}

The two terms in \eqref{eq:saa-error-decomposition} have distinct roles:
the first represents accumulated transient bias, while the second
represents zero-mean stochastic fluctuation.

\subsubsection{Consistency of the direct SAA estimator}
\label{subsec:saa-consistency}

We first state a basic consistency result under abstract conditions on the
two terms in \eqref{eq:saa-error-decomposition}.  Sufficient conditions
that are more directly verifiable in queueing models will be developed
subsequently.

\begin{theorem}[Consistency of the direct SAA estimator]
\label{thm:direct-saa-consistency}
Fix an action \(a\in\mathcal A\).  Suppose that the following conditions
hold almost surely.

\medskip
\noindent\textnormal{(A1)}
The cumulative operating time diverges:
\begin{equation}
S_n(a)\longrightarrow\infty.
\label{ass:sufficient-operating-time}
\end{equation}

\medskip
\noindent\textnormal{(A2)}
The accumulated conditional transient term is negligible relative to the
cumulative operating time:
\begin{equation}
\frac{\sum_{t=1}^n I_t(a)\beta_t}{S_n(a)}
\longrightarrow0.
\label{ass:vanishing-average-bias}
\end{equation}

\medskip
\noindent\textnormal{(A3)}
The accumulated martingale innovation is negligible relative to the
cumulative operating time:
\begin{equation}
\frac{\sum_{t=1}^n I_t(a)M_t}{S_n(a)}
\longrightarrow0.
\label{ass:martingale-ratio}
\end{equation}

Then
\begin{equation}
\widehat f_n(a)
\longrightarrow
f(a)
\qquad\text{almost surely}.
\label{eq:saa-consistency}
\end{equation}
\end{theorem}

\begin{proof}
By \textnormal{(A1)}, $S_n(a)>0$ for all sufficiently large $n$, so the
estimator and the normalized error decomposition are eventually well defined.
Then \eqref{eq:saa-error-decomposition} together with
\textnormal{(A2)}--\textnormal{(A3)} gives
$\widehat f_n(a)-f(a)\to0$ almost surely.
\end{proof}

\begin{remark}[Discussion of the assumptions]
\label{rem:consistency-assumptions}
Condition (A1) is an information-accumulation requirement: action \(a\)
must be operated for an unbounded amount of total clock time.

Condition (A2) controls the systematic transient effect associated with
persistent operation.  Importantly, it does not require the individual
terms \(\beta_t\) to converge to zero.  Individual fixed-\(D\) cycles may
remain transiently biased; what is required is only that their accumulated
contribution become negligible per unit of total operating time under
action \(a\).

Condition (A3) concerns stochastic fluctuation rather than systematic
bias.  Since
\[
\mathbb E[M_t\mid\mathcal H_{t-1}]=0,
\]
it is a martingale law-of-large-numbers condition.  Under suitable moment
conditions and sufficient repeated sampling of action \(a\), (A3) can be
verified using standard martingale strong-law arguments.

An important point is that Theorem~\ref{thm:direct-saa-consistency}
does not require the number of designated events per cycle to increase
with \(n\).  Thus the fixed-cycle construction \(D_t\equiv D\) is
admissible.  Fixing \(D\), however, does not by itself eliminate
transient bias.  Because the queue is not reset between cycles, each
cycle begins from an inherited state, and the corresponding transient
term \(\beta_t\) need not vanish as \(t\to\infty\).  Condition (A2)
requires only that, for each action \(a\), these transient contributions
become negligible after accumulation and normalization by the total
operating time \(S_n(a)\).

Subsection~\ref{sec:queue-verification} develops queueing-specific sufficient conditions for
(A2) and (A3).  Martingale arguments are used to control the stochastic
fluctuation in (A3), while a regenerative bias representation is used
to show how transient contributions in (A2) telescope over contiguous
operating blocks of a fixed action.
\end{remark}

\subsubsection{Uniform consistency and empirical-leader identification}
\label{subsec:uniform-leader}

Because the action set \(\mathcal A\) is finite, actionwise consistency
immediately implies uniform consistency over \(\mathcal A\).

\begin{corollary}[Uniform consistency over a finite action set]
\label{cor:uniform-consistency}
Suppose that the conditions of
Theorem~\ref{thm:direct-saa-consistency} hold for every
\(a\in\mathcal A\).  Then
\begin{equation}
\max_{a\in\mathcal A}
\left|
\widehat f_n(a)-f(a)
\right|
\longrightarrow 0
\qquad
\text{almost surely}.
\label{eq:uniform-consistency}
\end{equation}
\end{corollary}

\begin{proof}
For every $a\in\mathcal A$, Theorem~\ref{thm:direct-saa-consistency} gives
\[
\widehat f_n(a)-f(a)\longrightarrow0
\qquad\text{almost surely}.
\]
Since $\mathcal A$ is finite, these finitely many almost-sure convergence
events have probability one jointly, and on their intersection the maximum
of the finitely many absolute errors also converges to zero.
\end{proof}

Define the optimal value over the finite action set by
\begin{equation}
f_{\mathcal A}^*
:=
\min_{a\in\mathcal A} f(a),
\label{eq:finite-optimal-value}
\end{equation}
and let
\begin{equation}
\mathcal A^*
:=
\left\{
a\in\mathcal A:
f(a)=f_{\mathcal A}^*
\right\}
\label{eq:finite-optimal-set}
\end{equation}
denote the set of optimal actions.

Whenever \(S_n(a)>0\) for every \(a\in\mathcal A\), define an empirical
leader by
\[
\widehat a_n
\in
\operatorname*{argmin}_{a\in\mathcal A}\widehat f_n(a).
\]
Under the assumptions of Corollary~\ref{cor:uniform-consistency},
the empirical leader is therefore well defined for all sufficiently large
\(n\), almost surely.

\begin{corollary}[Consistency of the empirical leader]
\label{cor:empirical-leader}
Suppose that the conditions of
Theorem~\ref{thm:direct-saa-consistency} hold for every
\(a\in\mathcal A\).  Then, almost surely,
\begin{equation}
\widehat a_n\in\mathcal A^*
\qquad
\text{for all sufficiently large } n.
\label{eq:eventual-optimal-set}
\end{equation}

In particular, if the finite-action problem has a unique optimizer
\(a_{\mathcal A}^*\), then
\begin{equation}
\widehat a_n
=
a_{\mathcal A}^*
\qquad
\text{for all sufficiently large } n,
\quad\text{almost surely}.
\label{eq:eventual-unique-leader}
\end{equation}
\end{corollary}

\begin{proof}
If every action is optimal, then
\(\mathcal A^*=\mathcal A\), and the conclusion is immediate.
Otherwise, since \(\mathcal A\) is finite, the optimality gap
\begin{equation}
\Delta_{\mathcal A}
:=
\min_{a\notin\mathcal A^*}
\left\{
f(a)-f_{\mathcal A}^*
\right\}
\label{eq:finite-action-gap}
\end{equation}
is strictly positive.

By Corollary~\ref{cor:uniform-consistency}, almost surely there exists a
finite random index \(N\) such that, for all \(n\geq N\),
\[
\max_{a\in\mathcal A}
\left|
\widehat f_n(a)-f(a)
\right|
<
\frac{\Delta_{\mathcal A}}{2}.
\]
Fix such an \(n\).  For any \(a\notin\mathcal A^*\),
\[
\widehat f_n(a)
>
f(a)-\frac{\Delta_{\mathcal A}}{2}
\geq
f_{\mathcal A}^*
+
\frac{\Delta_{\mathcal A}}{2}.
\]
On the other hand, for any \(a^*\in\mathcal A^*\),
\[
\widehat f_n(a^*)
<
f_{\mathcal A}^*
+
\frac{\Delta_{\mathcal A}}{2}.
\]
Hence no action outside \(\mathcal A^*\) can minimize
\(\widehat f_n(\cdot)\).  Therefore
\[
\widehat a_n\in\mathcal A^*
\]
for all sufficiently large \(n\), almost surely.  If
\(\mathcal A^*=\{a_{\mathcal A}^*\}\), the unique-optimum conclusion
follows immediately.
\end{proof}

\begin{remark}[Finite-action versus continuous optimality]
\label{rem:grid-versus-continuous}
Corollary~\ref{cor:empirical-leader} establishes identification of an
optimizer over the finite action set \(\mathcal A\).  When \(\mathcal A\)
discretizes an underlying continuous decision space, an additional
discretization error remains.  This error is treated separately in the
continuous-space approximation analysis of the coarse-to-fine procedure.
\end{remark}


\subsection{Queueing verification of the direct-SAA conditions}
\label{sec:queue-verification}

The preceding results isolate the two quantities that must be controlled:
accumulated transient bias and martingale fluctuation.  We now give
queueing-specific sufficient conditions.  The main device is a regenerative
bias representation that is compatible with fixed-$D$ learning cycles even
though those learning cycles are not regeneration cycles.  The resulting
boundary terms telescope within contiguous operating blocks, while standard
martingale averaging controls centered fluctuations.

\subsubsection{Martingale averaging over adaptively selected cycles}
\label{subsec:martingale-averaging}

We next give a convenient sufficient condition for the martingale
requirement (A3). For \(a\in\mathcal A\), let
\begin{equation}
N_n(a)
:=
\sum_{t=1}^n I_t(a)
\label{eq:action-count}
\end{equation}
denote the number of times action \(a\) has been selected up to cycle \(n\).

For \(k\geq1\), define
\begin{equation}
\tau_k(a)
:=
\inf\{t\geq1:N_t(a)=k\},
\label{eq:kth-visit-time}
\end{equation}
with the convention that \(\inf\varnothing=\infty\). Thus
\(\tau_k(a)\) is the cycle index of the \(k\)-th selection of action \(a\),
whenever that selection occurs.

The following lemma is stated for a generic martingale-difference
sequence, so that it can be used for more than one stochastic-error term
in the subsequent analysis.

\begin{lemma}[Martingale averaging along selected cycles]
\label{lem:martingale-averaging}
Let \(\{Z_t\}_{t\geq1}\) be adapted to
\(\{\mathcal H_t\}_{t\geq0}\) and satisfy
\begin{equation}
\mathbb E[Z_t\mid\mathcal H_{t-1}]=0,
\qquad t\geq1.
\label{eq:generic-martingale-difference}
\end{equation}
Fix \(a\in\mathcal A\).  Suppose that, almost surely,

\medskip
\noindent\textnormal{(M1)}
action \(a\) is selected infinitely often, in the sense that
\begin{equation}
N_n(a)\longrightarrow\infty
\qquad\text{as } n\to\infty.
\label{ass:infinite-visits}
\end{equation}
Under \textnormal{(M1)}, \(\tau_k(a)<\infty\) for every \(k\geq1\)
almost surely.

\medskip
\noindent\textnormal{(M2)}
the average clock duration of cycles operated under \(a\) is bounded away
from zero asymptotically:
\begin{equation}
\liminf_{m\to\infty}
\frac{1}{m}
\sum_{k=1}^m T_{\tau_k(a)}
>0;
\label{ass:positive-average-duration}
\end{equation}

\medskip

\noindent\textnormal{(M3)}
the conditional second moments are uniformly bounded along cycles in which
action \(a\) is selected: there exists a constant \(C_a<\infty\) such that
\begin{equation}
\mathbb E
\left[
Z_t^2
\mid
\mathcal H_{t-1}
\right]
\leq C_a
\qquad
\text{a.s. on the event }\{a_t=a\},
\quad t\geq1.
\label{ass:uniform-selected-second-moment}
\end{equation}

Then
\begin{equation}
\frac{
\sum_{t=1}^n I_t(a)Z_t
}{
S_n(a)
}
\longrightarrow0
\qquad\text{almost surely}.
\label{eq:martingale-selected-ratio}
\end{equation}
\end{lemma}

\begin{proof}
Because $a_t$ is $\mathcal H_{t-1}$-measurable, the selection times
$\tau_k(a)$ are predictable: for every $t$,
$\{\tau_k(a)=t\}\in\mathcal H_{t-1}$.  Consequently the martingale-
difference property can be evaluated at the random selected index.  More
explicitly, if $H$ is bounded and
$\mathcal H_{\tau_k(a)-1}$-measurable, then
\begin{align*}
\mathbb E[H Z_{\tau_k(a)}]
&=
\sum_{t\geq1}
\mathbb E\!\left[H\mathbf 1\{\tau_k(a)=t\}Z_t\right]\\
&=
\sum_{t\geq1}
\mathbb E\!\left[
H\mathbf 1\{\tau_k(a)=t\}
\mathbb E[Z_t\mid\mathcal H_{t-1}]
\right]
=0.
\end{align*}
Thus
\[
\mathbb E[Z_{\tau_k(a)}\mid\mathcal H_{\tau_k(a)-1}]=0.
\]
The same predictable-index argument and \textnormal{(M3)} give
\[
\mathbb E[Z_{\tau_k(a)}^2\mid\mathcal H_{\tau_k(a)-1}]
\leq C_a.
\]

Now let
\[
\mathcal G_k(a)
:=
\mathcal H_{\tau_k(a)},
\qquad k\geq1,
\]
with $\mathcal G_0(a):=\mathcal H_0$, and define
\[
Y_k(a):=Z_{\tau_k(a)}.
\]
Since
$\mathcal G_{k-1}(a)\subseteq\mathcal H_{\tau_k(a)-1}$, the tower property
shows that $\{Y_k(a)\}_{k\geq1}$ is a martingale-difference sequence with
respect to $\{\mathcal G_k(a)\}_{k\geq0}$ and that
\[
\mathbb E
\left[
Y_k(a)^2\mid\mathcal G_{k-1}(a)
\right]
\leq C_a.
\]
Consequently,
\[
\sum_{k=1}^{\infty}
\frac{
\mathbb E[
Y_k(a)^2\mid\mathcal G_{k-1}(a)]
}{
k^2
}
\leq
C_a\sum_{k=1}^{\infty}\frac{1}{k^2}
<\infty.
\]
It follows that the martingale series
\[
\sum_{k=1}^{\infty}\frac{Y_k(a)}{k}
\]
converges almost surely.  Kronecker's lemma therefore gives
\begin{equation}
\frac{1}{m}
\sum_{k=1}^m Y_k(a)
=
\frac{1}{m}
\sum_{k=1}^m Z_{\tau_k(a)}
\longrightarrow0
\qquad\text{almost surely}.
\label{eq:selected-martingale-slln}
\end{equation}

For \(m=N_n(a)\),
\[
\sum_{t=1}^n I_t(a)Z_t
=
\sum_{k=1}^{m}Z_{\tau_k(a)}
\]
and
\[
S_n(a)
=
\sum_{k=1}^{m}T_{\tau_k(a)}.
\]
Therefore
\[
\frac{
\sum_{t=1}^n I_t(a)Z_t
}{
S_n(a)
}
=
\frac{
m^{-1}\sum_{k=1}^{m}Z_{\tau_k(a)}
}{
m^{-1}\sum_{k=1}^{m}T_{\tau_k(a)}
}.
\]
The numerator converges to zero by
\eqref{eq:selected-martingale-slln}, while the denominator is bounded
away from zero asymptotically by (M2).  Since \(N_n(a)\to\infty\) by
(M1), the result follows.
\end{proof}

\begin{corollary}[A sufficient condition for (A3)]
\label{cor:martingale-condition-A3}
Fix \(a\in\mathcal A\).  Suppose that
\[
N_n(a)\longrightarrow\infty
\qquad\text{almost surely},
\]
and
\begin{equation}
\liminf_{n\to\infty}
\frac{S_n(a)}{N_n(a)}
>0
\qquad\text{almost surely}.
\label{ass:positive-average-operating-time}
\end{equation}
Suppose further that there exists \(C_a<\infty\) such that
\begin{equation}
\mathbb E
\left[
M_t^2\mid\mathcal H_{t-1}
\right]
\leq C_a
\qquad
\text{on }\{a_t=a\}.
\label{ass:martingale-second-moment}
\end{equation}
Then
\begin{equation}
\frac{
\sum_{t=1}^n I_t(a)M_t
}{
S_n(a)
}
\longrightarrow0
\qquad\text{almost surely}.
\end{equation}
Thus condition \textnormal{(A3)} of
Theorem~\ref{thm:direct-saa-consistency} holds.
\end{corollary}

\begin{proof}
Apply Lemma~\ref{lem:martingale-averaging} with \(Z_t=M_t\).
Condition \eqref{eq:martingale-difference} gives the required
martingale-difference property, while
\eqref{ass:positive-average-operating-time} and
\eqref{ass:martingale-second-moment} give (M2) and (M3), respectively.
\end{proof}

\begin{remark}
The uniform conditional second-moment bound in
\eqref{ass:uniform-selected-second-moment} is a convenient sufficient
condition rather than a minimal one.  It can be replaced by more general
martingale summability conditions when the cycle-level fluctuations are
not uniformly square integrable.  The present form is useful because it
makes the distinction between information accumulation and stochastic
averaging transparent.
\end{remark}

\subsubsection{Regenerative bias representation of a fixed-\(D\) cycle}
\label{subsec:regenerative-cycle}

The main difficulty in verifying the transient-bias condition
\textnormal{(A2)} is that, under fixed-\(D\) operation, individual learning
cycles need not become asymptotically unbiased: each cycle begins from an
inherited state.  The key is therefore to represent the centered cost of
each cycle as a difference of state-dependent boundary terms plus a
martingale innovation.  When the same action is operated over consecutive
cycles, these boundary terms telescope, leaving only the states at the
beginning and end of the corresponding operating block.

We now justify this cycle-level boundary representation under explicit
regenerative and filtration-compatibility assumptions, without imposing a
continuous-time Markov-chain or parametric distributional model.
Fix an action \(a\in\mathcal A\), and let
\(\{\mathcal F_s^a\}_{s\geq0}\) be the filtration used to describe its
continuous operation.  We require the regeneration construction to be
compatible with this filtration, including the information observed by the
learning procedure.
Let
\[
0<R_1^a<R_2^a<\cdots
\]
be almost surely finite stopping times increasing to infinity.  These are
the regeneration epochs strictly after time zero, even when the initial
state is itself regenerative.  Conditional on \(\mathcal F_{R_j^a}^a\),
the next complete regenerative-cycle evolution, including its duration,
cost increment, and designated-event count, is assumed to have a common
law independent of \(\mathcal F_{R_j^a}^a\).  The same law is assumed for
the first complete cycle when the system is initialized in a regeneration
state.  An accessible regeneration atom with a compatible observation
filtration is a sufficient setting for these requirements.

We use \(X^a(s)\) as a complete state descriptor of the system; it need
not be a CTMC.  In non-Markovian queues it may include residual or elapsed
interarrival and service information.  Its state convention and the
allocation of any costs or designated events at a regeneration epoch must
agree with the filtration and cycle convention above.  In particular,
unconditional independence of successive cycle rewards is not a substitute
for the stated conditional restart property.

Let \(\mathcal C_a(s)\) denote cumulative direct cost under continuous
operation of action \(a\), normalized by \(\mathcal C_a(0)=0\).
Any initial-time charge is assigned consistently with this origin convention;
costs at subsequent boundaries follow the increment convention below.
Recall that the
designated events are the queueing events chosen to define the fixed-\(D\)
learning cycles, such as arrivals, service completions, or departures.
Let \(N_a^{\circ}(s)\) denote the counting process of these designated
events under action \(a\); that is, \(N_a^{\circ}(s)\) is the number of
designated events occurring over \((0,s]\), with \(N_a^{\circ}(0)=0\).
The superscript \(\circ\) distinguishes this event-counting process from
the action-selection count \(N_n(a)\) introduced earlier.
For
an ordinary regenerative cycle, define
\begin{align}
L_j^a
&:=R_{j+1}^a-R_j^a,
\label{eq:regen-cycle-length}
\\
Y_j^a
&:=\mathcal C_a(R_{j+1}^a)-\mathcal C_a(R_j^a),
\label{eq:regen-cycle-cost}
\\
D_j^a
&:=N_a^{\circ}(R_{j+1}^a)-N_a^{\circ}(R_j^a).
\label{eq:regen-cycle-events}
\end{align}
By the conditional restart property,
\[
\{(L_j^a,Y_j^a,D_j^a)\}_{j\geq1}
\]
are i.i.d. For this regenerative verification, assume that $\mathcal C_a$
has locally finite-variation paths and define the absolute accumulated
cost of a complete cycle by
\[
Y_{j,\mathrm{abs}}^a
:=\int_{(R_j^a,R_{j+1}^a]}|d\mathcal C_a|(s),
\qquad |Y_j^a|\le Y_{j,\mathrm{abs}}^a,
\]
where $|d\mathcal C_a|$ is the total variation of the signed cost measure.
For running costs and event charges, this accumulates the absolute running
rate and the absolute cost jump at each event. Assume
\begin{equation}
0<\mathbb E[L_1^a]<\infty,
\qquad
\mathbb E[Y_{1,\mathrm{abs}}^a]<\infty.
\label{ass:regen-first-moments}
\end{equation}
This controls positive and negative within-cycle contributions before
they cancel; integrability of the absolute net reward alone would not.
The renewal-reward theorem then identifies the steady-state expected cost rate
\begin{equation}
f(a)
=
\frac{\mathbb E[Y_1^a]}{\mathbb E[L_1^a]}.
\label{eq:regen-cost-rate}
\end{equation}
Consequently, the centered regenerative-cycle rewards
\begin{equation}
Z_j^a
:=
Y_j^a-f(a)L_j^a
\label{eq:centered-regen-reward}
\end{equation}
are i.i.d. and satisfy \(\mathbb E[Z_j^a]=0\).
Absolute cost is used for integrability only: the rewards, residual costs,
and bias function in the decomposition continue to use signed net costs.

For a state \(x\) that can occur at a learning-cycle boundary, let
\(\sigma_a(x)\) denote the time from state \(x\) to the next strictly
positive regeneration epoch.  Thus, when \(x\) is a regeneration state,
this time covers one complete regenerative cycle.
Let
\begin{equation}
B_a(x)
:=
\mathcal C_a(\sigma_a(x))-f(a)\sigma_a(x)
\label{eq:regen-residual-centered-cost}
\end{equation}
denote the centered cost accumulated over this residual portion of the
regenerative cycle when the system is started from \(x\).  We assume
\begin{equation}
\mathbb E_x[|B_a(x)|]<\infty
\label{ass:regen-residual-integrability}
\end{equation}
for every possible learning-cycle boundary state.  Define the regenerative
bias function
\begin{equation}
g_a(x)
:=
\mathbb E_x[B_a(x)].
\label{eq:regen-bias-function}
\end{equation}
At a regeneration state, the assumed conditional restart law gives
\(g_a=0\), because the ensuing full cycle has mean centered reward zero
conditional on that state.  This is a consequence of the compatibility
assumption, not of an unconditional mean-zero cycle reward alone.

Let
\begin{equation}
\tau_D^a
:=
\inf\{s\geq0:N_a^{\circ}(s)\geq D\}
\label{eq:fixed-D-stopping-time}
\end{equation}
be the random clock time required to observe \(D\) designated events under
action \(a\).  The learning cycle itself is not assumed to begin or end at
a regeneration epoch.

We also impose the following stopping-time residual-state property.
For the almost surely finite stopping times $\tau$ used below, whenever
the centered residual cost to the next regeneration is integrable, assume
\begin{equation}
\mathbb E_x^a\!\left[
\mathcal C_a(R^+(\tau))-\mathcal C_a(\tau)
-f(a)\{R^+(\tau)-\tau\}
\,\middle|\,
\mathcal F_\tau^a
\right]
=
g_a(X^a(\tau)),
\label{ass:stopping-time-residual-state}
\end{equation}
where $R^+(\tau)$ is the first regeneration epoch at or after $\tau$.
When $\tau$ is itself a regeneration epoch, the residual on the left is
zero and agrees with $g_a(X^a(\tau))=0$.  Thus the at-or-after convention
in this identity is consistent with the strictly positive regeneration
time used to define $B_a$ and $g_a$.  Away from regeneration epochs, both
conventions refer to the same next regeneration.  A strong Markov model
with an accessible regeneration atom and compatible state, cost, and event
conventions can establish this property.  Strong Markovness alone, however,
does not imply it for an arbitrary random regeneration state or observation
filtration; \eqref{ass:stopping-time-residual-state} remains an explicit
assumption here.

\begin{proposition}[Regenerative representation of a fixed-\(D\) cycle]
\label{prop:regenerative-fixed-D}
Fix $a\in\mathcal A$, and suppose that, under continuous operation of
action $a$, the regeneration epochs and conditional restart law are
compatible with $\{\mathcal F_s^a\}$ as specified above, and that
\eqref{ass:regen-first-moments},
\eqref{ass:regen-residual-integrability}, and the stopping-time residual-state
property \eqref{ass:stopping-time-residual-state} hold.

Assume in addition that each complete regenerative cycle contains at least
one designated event,
\begin{equation}
D_j^a\geq1
\qquad\text{almost surely},
\label{ass:regen-event-positive}
\end{equation}
and that
\[
\tau_D^a<\infty
\qquad\text{almost surely}.
\]

Suppose also that, for every possible initial state
\(x=X^a(0)\),
\begin{equation}
\mathbb E_x\!\left[
\left|
\mathcal C_a(\tau_D^a)-f(a)\tau_D^a
\right|
\right]
<\infty,
\label{ass:fixed-D-centered-cost-integrability}
\end{equation}
and that the actual centered residual cost from the terminal time to
$R^+(\tau_D^a)$ is integrable:
\begin{equation}
\mathbb E_x\!\left[
\left|
\mathcal C_a(R^+(\tau_D^a))-\mathcal C_a(\tau_D^a)
-f(a)\{R^+(\tau_D^a)-\tau_D^a\}
\right|
\right]
<\infty.
\label{ass:terminal-residual-integrability}
\end{equation}
Then
\begin{equation}
\mathcal C_a(\tau_D^a)-f(a)\tau_D^a
=
g_a(X^a(0))-g_a(X^a(\tau_D^a))
+U_a(\tau_D^a),
\label{eq:stopped-regenerative-representation}
\end{equation}
where \(U_a(\tau_D^a)\) is integrable and, for every possible initial
state \(x\),
\begin{equation}
\mathbb E_x[U_a(\tau_D^a)]=0.
\label{eq:stopped-regenerative-zero-mean}
\end{equation}
\end{proposition}

The integrability conditions require finite expected absolute magnitudes
for the finite-cycle and residual boundary contributions.  They permit
heavy-tailed models when the stated moments and the separate regeneration
and residual-state assumptions hold.  The learning cycles themselves need
not be regenerative or independent.

\begin{proof}
Let $R^+(\tau_D^a)$ be the first regeneration epoch at or after
$\tau_D^a$.  Define
\[
K_D^a
:=\inf\{k\geq0:N_a^\circ(R_{k+1}^a)\geq D\},
\qquad
\mathcal R_k^a:=\mathcal F_{R_{k+1}^a}^a,
\quad k\geq0.
\]
Then $R^+(\tau_D^a)=R_{K_D^a+1}^a$ and $K_D^a$ counts the complete
cycles after the initial segment ending at $R_1^a$.  The conditional
restart assumption gives
$\mathbb E_x[Z_j^a\mid\mathcal R_{j-1}^a]=0$, because
$\mathcal R_{j-1}^a=\mathcal F_{R_j^a}^a$.  Hence the partial sums
\[
\mathcal M_k^a
:=
\sum_{j=1}^k Z_j^a
\]
therefore form a martingale with respect to $\{\mathcal R_k^a\}$.

For every $k\geq0$,
\[
\{K_D^a\leq k\}
=\{N_a^\circ(R_{k+1}^a)\geq D\}\in\mathcal R_k^a,
\]
so $K_D^a$ is a stopping time for $\{\mathcal R_k^a\}$.  Moreover,
\[
N_a^\circ(R_{D+1}^a)
=N_a^\circ(R_1^a)+\sum_{j=1}^{D}D_j^a
\geq D,
\]
and therefore $K_D^a\leq D$ almost surely.  Bounded optional sampling gives
\begin{equation}
\mathbb E_x\left[
\sum_{j=1}^{K_D^a} Z_j^a
\right]=0.
\label{eq:regen-stopped-sum-zero}
\end{equation}

Starting from \(x=X^a(0)\), the centered cost accumulated up to
\(R^+(\tau_D^a)\) consists of the initial residual centered cost \(B_a(x)\)
plus the stopped sum of complete centered regenerative-cycle rewards.
If $x$ is itself a regeneration state, this initial segment is one full
cycle ending at the strictly positive epoch $R_1^a$, and its mean is
$g_a(x)=0$.  If $\tau_D^a\leq R_1^a$, then $K_D^a=0$ and the sum is
empty.  Thus the same decomposition covers both boundary cases.
Hence, by \eqref{eq:regen-bias-function} and
\eqref{eq:regen-stopped-sum-zero},
\begin{equation}
\mathbb E_x\left[
\mathcal C_a(R^+(\tau_D^a))-f(a)R^+(\tau_D^a)
\right]
=
g_a(x).
\label{eq:regen-total-to-next-regeneration}
\end{equation}

By the stopping-time residual-state property
\eqref{ass:stopping-time-residual-state}, applied at $\tau_D^a$, the expected
centered cost from $\tau_D^a$ to $R^+(\tau_D^a)$ conditional on
$\mathcal F_{\tau_D^a}^a$ is
\begin{equation}
\mathbb E_x^a\left[
\mathcal C_a(R^+(\tau_D^a))-\mathcal C_a(\tau_D^a)
-f(a)\{R^+(\tau_D^a)-\tau_D^a\}
\,\middle|\,
\mathcal F_{\tau_D^a}^a
\right]
=
g_a(X^a(\tau_D^a)).
\label{eq:regen-terminal-residual}
\end{equation}
Subtracting \eqref{eq:regen-terminal-residual} from
\eqref{eq:regen-total-to-next-regeneration} and taking expectations yields
\begin{equation}
\mathbb E_x\left[
\mathcal C_a(\tau_D^a)-f(a)\tau_D^a
\right]
=
g_a(x)-\mathbb E_x[g_a(X^a(\tau_D^a))].
\label{eq:regen-mean-boundary-identity}
\end{equation}
Conditional Jensen's inequality and
\eqref{ass:terminal-residual-integrability} give
$\mathbb E_x|g_a(X^a(\tau_D^a))|<\infty$. Together with
\eqref{ass:fixed-D-centered-cost-integrability}, this ensures that the
following innovation is integrable. Defining
\[
U_a(\tau_D^a)
:=
\mathcal C_a(\tau_D^a)-f(a)\tau_D^a
-g_a(X^a(0))+g_a(X^a(\tau_D^a))
\]
gives \eqref{eq:stopped-regenerative-representation}, and
\eqref{eq:regen-mean-boundary-identity} gives
\eqref{eq:stopped-regenerative-zero-mean}.
\end{proof}

\begin{remark}[Learning cycles need not be regenerative]
\label{rem:learning-cycles-not-regenerative}
The regeneration epochs are used only as an analytical device.  The online
learning cycles remain the fixed-\(D\) cycles defined by the algorithm: the
queue is not reset at their boundaries, and successive observations may be
serially dependent through the inherited state.  Thus the regenerative
assumptions do not require the learning data to be i.i.d.  Their role is to
provide a queueing-based construction of the bias function \(g_a\) and to
verify the boundary representation needed for the transient-bias analysis.
Condition \eqref{ass:regen-event-positive} is a convenient sufficient
condition ensuring that only finitely, indeed at most \(D\), complete
regenerative cycles can be crossed during one fixed-\(D\) learning cycle.
It can be weakened whenever the corresponding stopped sum of centered
regenerative-cycle rewards is integrable and has mean zero.
\end{remark}

\begin{remark}[GI/G/1 and other regenerative queues]
\label{rem:gig1-regeneration}
The regenerative formulation does not require exponential interarrival or
service times.  For example, consider a stable GI/G/1 queue under a fixed
action \(a\).  Let \(A_1^a,A_2^a,\ldots\) be i.i.d. interarrival times and
\(B_1^a,B_2^a,\ldots\) be i.i.d. service requirements, with the two
sequences independent, finite positive means, and
\begin{equation}
\rho_a
:=
\frac{\mathbb E[B_1^a]}{\mathbb E[A_1^a]}
<1.
\label{eq:gig1-stability}
\end{equation}
An arrival that finds the system empty is a standard candidate for a
regeneration construction.  Its compatibility with the observed state must
be specified.  One convention places the restart after recording the
empty-system arrival but before revealing the fresh service requirement
and next-interarrival mark.  Those marks then belong to the new cycle.
The pre-/post-arrival state and any arrival-time costs or designated events
must be assigned consistently with that convention and with the information
actually available to the learner.  If the state at the proposed restart
already records a fresh random mark, the conditional restart law and the
zero-bias normalization cannot be inferred merely from i.i.d. cycles; a
compatible construction or a separate bias argument is required.

Proposition~\ref{prop:regenerative-fixed-D} applies only after its
filtration, stopping-time residual identity, designated-event, and
integrability assumptions have been checked for the chosen construction.
For queue-length holding costs, the moment verification concerns the
busy-cycle duration, the area under the queue-length process, and the
initial and terminal residual contributions.  Stronger moments of the
primitive interarrival and service times can provide sufficient conditions
in a particular model, but stability alone does not supply this complete
verification.

The same approach is available for M/M/1, M/G/1, GI/M/1, finite-capacity,
and network queues when a compatible regeneration construction and the
other stated assumptions are available.  These assumptions concern the
analysis; the direct online SAA estimator itself requires neither observed
regeneration epochs nor knowledge of the parametric forms of the arrival
or service distributions.
\end{remark}

\begin{remark}[Connection with Poisson equations]
\label{rem:regenerative-poisson-connection}
When the complete state process under a fixed action is Markov, the
regenerative bias function in \eqref{eq:regen-bias-function} is, up to an
additive constant, the usual relative-value or Poisson-equation solution
under standard regularity conditions.  Thus the continuous-time Markov-chain
Poisson formulation is a useful special case of the regenerative argument,
rather than an assumption required by the learning method.
\end{remark}

\begin{remark}[Connection with the learning cycles]
\label{rem:regenerative-learning-cycle}
Conditional on \(\mathcal H_{t-1}\), both the selected action \(a_t\) and
the inherited initial state \(X_{t-1}\) are known.  During cycle \(t\), the
action is held fixed at \(a_t\), and the cycle ends after \(D\) designated
events.  Hence Proposition~\ref{prop:regenerative-fixed-D}, applied
conditionally with \(a=a_t\), yields the following identity when its
regeneration and residual-state assumptions hold for every admissible
action and inherited state, and the frozen-action cycle law agrees with
the conditional law given the actual learning history:
\begin{equation}
G_t-f(a_t)T_t
=
g_{a_t}(X_{t-1})-g_{a_t}(X_t)+U_t(a_t),
\label{eq:learning-cycle-regenerative-decomposition}
\end{equation}
where
\begin{equation}
\mathbb E\left[
U_t(a_t)\mid\mathcal H_{t-1}
\right]
=0.
\label{eq:learning-cycle-regenerative-zero-mean}
\end{equation}
Thus the cycle-level representation follows from the specified compatible
regenerative construction and moment conditions, without requiring a CTMC
model for the primitive queueing dynamics.  A valid Poisson-equation
construction of the same representation is an alternative when this
particular regeneration convention is unavailable.
\end{remark}

\subsubsection{Control of transient bias through action blocks}
\label{subsec:regenerative-transient}

We next develop sufficient conditions for the transient-bias requirement
(A2), using the cycle-boundary state process \(\{X_t\}_{t\geq0}\)
introduced in Section~\ref{sec:model-saa}.

Fix \(a\in\mathcal A\).  Suppose that there exists a measurable function
\(g_a\) such that, whenever \(a_t=a\),
\begin{equation}
G_t-f(a)T_t
=
g_a(X_{t-1})-g_a(X_t)+U_t(a),
\label{eq:cycle-bias-decomposition}
\end{equation}
where
\begin{equation}
\mathbb E
\left[
U_t(a)\mid\mathcal H_{t-1}
\right]
=
0.
\label{eq:bias-martingale}
\end{equation}
By Proposition~\ref{prop:regenerative-fixed-D} and
Remark~\ref{rem:regenerative-learning-cycle}, this representation follows
from the fixed-action regenerative construction under the conditions stated
there.

Define
\begin{equation}
\Delta_t(a)
:=
g_a(X_{t-1})-g_a(X_t).
\label{eq:boundary-increment}
\end{equation}
Then, on the event \(\{a_t=a\}\),
\[
\beta_t
=
\mathbb E
\left[
\Delta_t(a)\mid\mathcal H_{t-1}
\right].
\]
Define further
\begin{equation}
R_t(a)
:=
I_t(a)
\left[
\Delta_t(a)
-
\mathbb E\!\left[
\Delta_t(a)\mid\mathcal H_{t-1}
\right]
\right].
\label{eq:boundary-martingale}
\end{equation}
Since \(I_t(a)\) is \(\mathcal H_{t-1}\)-measurable,
\begin{equation}
\mathbb E
\left[
R_t(a)\mid\mathcal H_{t-1}
\right]
=
0.
\label{eq:boundary-martingale-difference}
\end{equation}

Among cycles \(1,\ldots,n\), an \emph{action block} for action \(a\)
is a maximal contiguous sequence of cycles during which \(a\) is
operated. Let \(J_n(a)\) denote the number of such blocks, and write them as
\[
[\ell_{a,1},r_{a,1}],\ldots,
[\ell_{a,J_n(a)},r_{a,J_n(a)}].
\]
Then
\begin{equation}
\sum_{t=1}^n I_t(a)\Delta_t(a)
=
\sum_{j=1}^{J_n(a)}
\left\{
g_a(X_{\ell_{a,j}-1})
-
g_a(X_{r_{a,j}})
\right\}.
\label{eq:block-telescoping}
\end{equation}

\begin{theorem}[Vanishing transient bias under negligible action-block boundaries]
\label{thm:transient-bias-blocks}
Fix \(a\in\mathcal A\), and suppose that
\eqref{eq:cycle-bias-decomposition}--\eqref{eq:bias-martingale}
hold.  Assume that, almost surely,

\medskip
\noindent\textnormal{(B1)}
action \(a\) is selected infinitely often and the average clock duration
of its selected cycles is bounded away from zero asymptotically:
\begin{equation}
N_n(a)\longrightarrow\infty,
\qquad
\liminf_{n\to\infty}
\frac{S_n(a)}{N_n(a)}
>0.
\label{ass:block-positive-operating-time}
\end{equation}

\medskip
\noindent\textnormal{(B2)}
the total magnitude of the regenerative bias function at the action-block
boundaries is negligible relative to the cumulative operating time:
\begin{equation}
\frac{1}{S_n(a)}
\sum_{j=1}^{J_n(a)}
\left(
|g_a(X_{\ell_{a,j}-1})|
+
|g_a(X_{r_{a,j}})|
\right)
\longrightarrow0;
\label{ass:normalized-block-boundary}
\end{equation}

\medskip
\noindent\textnormal{(B3)}
the centered boundary fluctuations satisfy the martingale-variance
summability condition
\begin{equation}
\sum_{t=1}^{\infty}
\frac{
I_t(a)\,
\mathbb E\!\left[R_t(a)^2\mid\mathcal H_{t-1}\right]
}{
\{N_{t-1}(a)+1\}^2
}
<\infty.
\label{ass:boundary-variance-summability}
\end{equation}

Then the accumulated conditional transient term satisfies
\begin{equation}
\frac{
\sum_{t=1}^n I_t(a)\beta_t
}{
S_n(a)
}
\longrightarrow0
\qquad\text{almost surely}.
\label{eq:transient-bias-conclusion}
\end{equation}
Hence condition \textnormal{(A2)} of
Theorem~\ref{thm:direct-saa-consistency} holds.
\end{theorem}

\begin{proof}
By \eqref{eq:boundary-martingale-difference}, \(\{R_t(a)\}\) is a
martingale-difference sequence and \(R_t(a)=0\) whenever \(a_t\neq a\).
Since \(N_{t-1}(a)+1\) is \(\mathcal H_{t-1}\)-measurable, define the
auxiliary weighted sequence
\begin{equation}
V_t(a)
:=
\frac{R_t(a)}{N_{t-1}(a)+1}.
\label{eq:weighted-boundary-martingale}
\end{equation}
Then \(\{V_t(a)\}\) is again a martingale-difference sequence.
Condition \textnormal{(B3)}
implies
\[
\sum_{t=1}^{\infty}
\mathbb E\!\left[V_t(a)^2\mid\mathcal H_{t-1}\right]
<\infty
\qquad\text{almost surely},
\]
so the martingale series \(\sum_{t\geq1}V_t(a)\) converges almost surely.
At the \(k\)-th selection of action \(a\),
\(N_{t-1}(a)+1=k\).  Hence the convergent series is precisely
\[
\sum_{k=1}^{\infty}
\frac{R_{\tau_k(a)}(a)}{k}.
\]
Kronecker's lemma therefore yields
\begin{equation}
\frac{1}{N_n(a)}
\sum_{t=1}^n R_t(a)
\longrightarrow0
\qquad\text{almost surely}.
\label{eq:boundary-martingale-selected-average}
\end{equation}
Combining this with \textnormal{(B1)} gives
\begin{equation}
\frac{
\sum_{t=1}^n R_t(a)
}{
S_n(a)
}
\longrightarrow0
\qquad\text{almost surely}.
\label{eq:boundary-martingale-ratio}
\end{equation}

From the definition of \(R_t(a)\),
\[
I_t(a)\beta_t
=
I_t(a)\Delta_t(a)-R_t(a).
\]
Therefore, using the telescoping identity \eqref{eq:block-telescoping},
\begin{align*}
\frac{
\left|\sum_{t=1}^n I_t(a)\beta_t\right|
}{S_n(a)}
&\leq
\frac{1}{S_n(a)}
\sum_{j=1}^{J_n(a)}
\left(
|g_a(X_{\ell_{a,j}-1})|
+
|g_a(X_{r_{a,j}})|
\right)
\\
&\quad+
\frac{
\left|\sum_{t=1}^nR_t(a)\right|
}{S_n(a)}.
\end{align*}
The first term converges to zero by \textnormal{(B2)}, while the second
converges to zero by \eqref{eq:boundary-martingale-ratio}.  This proves
\eqref{eq:transient-bias-conclusion}.
\end{proof}

Theorem~\ref{thm:transient-bias-blocks} reduces the transient-bias
problem to conditions \textnormal{(B2)} and \textnormal{(B3)}.
We next give readily verifiable sufficient conditions for these two
requirements.  The first route uses moment bounds only along cycles in
which action \(a\) is actually selected, together with sufficiently
sparse creation of new action blocks.  A subsequent result gives a
queueing-oriented verification through a common Lyapunov condition.

\begin{corollary}[Selected-cycle moments and sparse block creation]
\label{cor:block-moment-growth}
Fix \(a\in\mathcal A\), and assume \textnormal{(B1)}.  Recall that
\(\tau_k(a)\) is the cycle index of the \(k\)-th selection of action \(a\).
Suppose that, for some \(q>2\),
\begin{equation}
\sup_{k\geq1}
\mathbb E\!\left[
\left|g_a(X_{\tau_k(a)-1})\right|^q
+
\left|g_a(X_{\tau_k(a)})\right|^q
\right]
<\infty.
\label{ass:selected-endpoint-q-moment}
\end{equation}
Suppose also that, for some
\(\varepsilon\in(0,1-1/q)\),
\begin{equation}
\frac{
J_n(a)\,N_n(a)^{1/q+\varepsilon}
}{S_n(a)}
\longrightarrow0
\qquad\text{almost surely}.
\label{ass:selected-block-growth}
\end{equation}
Then conditions \textnormal{(B2)}--\textnormal{(B3)} of
Theorem~\ref{thm:transient-bias-blocks} hold, and hence
\eqref{eq:transient-bias-conclusion} follows.
\end{corollary}

\begin{proof}
Set
\[
Y_k(a)
:=
\left|g_a(X_{\tau_k(a)-1})\right|
+
\left|g_a(X_{\tau_k(a)})\right|.
\]
Condition \eqref{ass:selected-endpoint-q-moment} implies
\(\sup_k\mathbb E[Y_k(a)^q]<\infty\).  By Markov's inequality and the
Borel--Cantelli lemma,
\begin{equation}
Y_k(a)
=
o\!\left(k^{1/q+\varepsilon}\right)
\qquad\text{almost surely}.
\label{eq:selected-endpoint-growth}
\end{equation}
Consequently,
\[
\max_{1\leq i\leq m}Y_i(a)
=
o\!\left(m^{1/q+\varepsilon}\right)
\qquad\text{almost surely}.
\]
For \(m=N_n(a)\), every action-block boundary among cycles
\(1,\ldots,n\) is either the state immediately before some selected cycle
or the state immediately after some selected cycle.  Therefore
\begin{align*}
&\sum_{j=1}^{J_n(a)}
\left(
|g_a(X_{\ell_{a,j}-1})|
+
|g_a(X_{r_{a,j}})|
\right)
\\
&\qquad\leq
2J_n(a)
\max_{1\leq i\leq N_n(a)}Y_i(a)
=
o\!\left(
J_n(a)N_n(a)^{1/q+\varepsilon}
\right)
\qquad\text{almost surely}.
\end{align*}
Together with \eqref{ass:selected-block-growth}, this proves
\textnormal{(B2)}.

It remains to verify \textnormal{(B3)}.  On the \(k\)-th selected cycle,
write
\begin{equation}
A_k(a)
:=
\mathbb E\!\left[
R_{\tau_k(a)}(a)^2
\,\middle|\,
\mathcal H_{\tau_k(a)-1}
\right].
\label{eq:selected-boundary-variance}
\end{equation}
Since \(X_{\tau_k(a)-1}\) is known at the beginning of the selected cycle,
\begin{equation}
A_k(a)
=
\operatorname{Var}\!\left(
 g_a(X_{\tau_k(a)})
 \mid
 \mathcal H_{\tau_k(a)-1}
\right).
\label{eq:selected-terminal-variance}
\end{equation}
Hence, using
\(\operatorname{Var}(Z\mid\mathcal F)\leq
\mathbb E[Z^2\mid\mathcal F]\) and conditional Jensen's inequality,
\[
A_k(a)^{q/2}
\leq
\mathbb E\!\left[
|g_a(X_{\tau_k(a)})|^q
\,\middle|\,
\mathcal H_{\tau_k(a)-1}
\right].
\]
Taking expectations and using
\eqref{ass:selected-endpoint-q-moment} gives
\[
\sup_{k\geq1}\mathbb E[A_k(a)^{q/2}]<\infty.
\]
Choose \(\delta\in(0,1-2/q)\).  Another application of Markov's
inequality and Borel--Cantelli yields
\[
A_k(a)
=
O\!\left(k^{2/q+\delta}\right)
\qquad\text{almost surely}.
\]
Since \(2/q+\delta<1\),
\[
\sum_{k=1}^{\infty}\frac{A_k(a)}{k^2}<\infty
\qquad\text{almost surely}.
\]
At the \(k\)-th selected cycle,
\(N_{t-1}(a)+1=k\), and therefore this last display is exactly
\textnormal{(B3)}.
\end{proof}

\begin{remark}[Interpretation of the selected-cycle conditions]
\label{rem:block-moment-interpretation}
Condition \eqref{ass:selected-endpoint-q-moment} is imposed only on the
queue states immediately before and after cycles in which action \(a\) is
actually selected; it does not require independence of these states or of
successive learning observations.  The restriction \(q>2\) is useful
because the same moment assumption controls both parts of the proof: the
block-boundary magnitudes in \textnormal{(B2)} and the boundary-martingale
variance in \textnormal{(B3)}.  Thus no separate variance-growth assumption
is needed in Corollary~\ref{cor:block-moment-growth}.

Condition \eqref{ass:selected-block-growth} requires new action blocks to
be created slowly enough that their accumulated boundary contribution is
negligible relative to operating time.  By \textnormal{(B1)}, a simpler
sufficient version is
\begin{equation}
\frac{J_n(a)}{
N_n(a)^{1-1/q-\varepsilon}}
\longrightarrow0.
\label{ass:block-count-vs-selections}
\end{equation}
These are sufficient rather than necessary conditions; the direct boundary
condition \textnormal{(B2)} can hold under substantially weaker action
schedules.
\end{remark}

\medskip
We next give a queueing-oriented route for verifying the preceding moment
requirements along the actual adaptively controlled trajectory.  The
argument uses a common Lyapunov function at fixed-\(D\) learning-cycle
boundaries together with moments of the regenerative residual cost.  It
does not require the primitive queueing process to be a CTMC.

Let \(\mathsf X\) denote the set of possible complete states at learning-cycle
boundaries.  For a nonnegative function \(\mathcal V\) on \(\mathsf X\), define
\begin{equation}
P_a^{(D)}\mathcal V(x)
:=
\mathbb E_x^a\!\left[
\mathcal V\!\left(X^a(\tau_D^a)\right)
\right],
\label{eq:fixed-D-lyapunov-operator}
\end{equation}
where the expectation is taken when one fixed-\(D\) cycle is run from state
\(x\) with action \(a\) held fixed throughout the cycle.

\begin{corollary}[Common Lyapunov and regenerative-moment verification]
\label{cor:common-lyapunov-regenerative}
Suppose that there exist a function
\(\mathcal V:\mathsf X\to[1,\infty)\), constants \(\rho\in(0,1)\),
\(b<\infty\), and \(C_B<\infty\), and some \(q>2\), such that, uniformly
for \(a\in\mathcal A\),
\begin{equation}
P_a^{(D)}\mathcal V(x)
\leq
\rho\mathcal V(x)+b,
\qquad x\in\mathsf X,
\label{ass:common-fixed-D-drift}
\end{equation}
and
\begin{equation}
\mathbb E_x^a\!\left[|B_a(x)|^q\right]
\leq
C_B\mathcal V(x),
\qquad x\in\mathsf X,
\label{ass:regenerative-residual-q-moment}
\end{equation}
where \(B_a(x)\) is the centered residual regenerative cost in
\eqref{eq:regen-residual-centered-cost}.  Assume also that
\begin{equation}
\mathbb E[\mathcal V(X_0)]<\infty.
\label{ass:initial-lyapunov-moment}
\end{equation}

Fix \(a\in\mathcal A\) and assume \textnormal{(B1)}.  Suppose that, for
some \(\varepsilon>0\),
\begin{equation}
\frac{J_n(a)n^{1/q+\varepsilon}}{S_n(a)}
\longrightarrow0
\qquad\text{almost surely},
\label{ass:global-time-block-growth}
\end{equation}
and that there exist constants \(c_a>0\) and
\(\kappa_a>2/q\) such that
\begin{equation}
N_n(a)
\geq
c_a n^{\kappa_a}
\qquad\text{for all sufficiently large }n,
\quad\text{almost surely}.
\label{ass:polynomial-selection-frequency}
\end{equation}
Then conditions \textnormal{(B2)}--\textnormal{(B3)} of
Theorem~\ref{thm:transient-bias-blocks} hold, and therefore
\eqref{eq:transient-bias-conclusion} follows.
\end{corollary}

\begin{proof}
Because \(a_t\) is \(\mathcal H_{t-1}\)-measurable, is held fixed during
cycle \(t\), and \(X_{t-1}\) is a complete state descriptor,
\[
\mathbb E\!\left[
\mathcal V(X_t)\mid\mathcal H_{t-1}
\right]
=
P_{a_t}^{(D)}\mathcal V(X_{t-1}).
\]
Hence the common drift condition
\eqref{ass:common-fixed-D-drift} gives
\begin{equation}
\mathbb E\!\left[
\mathcal V(X_t)
\mid
\mathcal H_{t-1}
\right]
\leq
\rho\mathcal V(X_{t-1})+b.
\label{eq:adaptive-lyapunov-drift}
\end{equation}
Iterating and using \eqref{ass:initial-lyapunov-moment},
\begin{equation}
\sup_{t\geq0}\mathbb E[\mathcal V(X_t)]<\infty.
\label{eq:uniform-adaptive-lyapunov-moment}
\end{equation}
Thus the Lyapunov control holds along the actual persistent trajectory,
even though the selected action may change from cycle to cycle.

By the regenerative definition \(g_a(x)=\mathbb E_x^a[B_a(x)]\) and
Jensen's inequality,
\begin{equation}
|g_a(x)|^q
\leq
\mathbb E_x^a[|B_a(x)|^q]
\leq
C_B\mathcal V(x).
\label{eq:bias-growth-from-regeneration}
\end{equation}
Combining \eqref{eq:uniform-adaptive-lyapunov-moment} and
\eqref{eq:bias-growth-from-regeneration} yields
\begin{equation}
\sup_{t\geq0}\mathbb E[|g_a(X_t)|^q]<\infty.
\label{eq:uniform-deterministic-bias-moment}
\end{equation}
For every \(\eta>0\), Markov's inequality and Borel--Cantelli therefore
imply
\begin{equation}
\max_{0\leq s\leq n}|g_a(X_s)|
=
o\!\left(n^{1/q+\eta}\right)
\qquad\text{almost surely}.
\label{eq:global-bias-path-growth}
\end{equation}
Taking \(\eta=\varepsilon\) and using
\eqref{ass:global-time-block-growth},
\begin{align*}
&\frac{1}{S_n(a)}
\sum_{j=1}^{J_n(a)}
\left(
|g_a(X_{\ell_{a,j}-1})|
+
|g_a(X_{r_{a,j}})|
\right)
\\
&\qquad\leq
\frac{2J_n(a)}{S_n(a)}
\max_{0\leq s\leq n}|g_a(X_s)|
\longrightarrow0
\qquad\text{almost surely}.
\end{align*}
Hence \textnormal{(B2)} holds.

To verify \textnormal{(B3)}, first note from
\eqref{eq:uniform-adaptive-lyapunov-moment} that, for every \(\delta>0\),
\begin{equation}
\mathcal V(X_t)
=
o\!\left(t^{1+\delta}\right)
\qquad\text{almost surely},
\label{eq:lyapunov-path-growth}
\end{equation}
again by Markov's inequality and Borel--Cantelli.  On \(\{a_t=a\}\), the definition of \(R_t(a)\) gives
\[
\mathbb E\!\left[R_t(a)^2\mid\mathcal H_{t-1}\right]
=
\operatorname{Var}\!\left(g_a(X_t)\mid\mathcal H_{t-1}\right).
\]
Therefore \eqref{eq:bias-growth-from-regeneration} and
\eqref{eq:adaptive-lyapunov-drift} give
\begin{align}
\mathbb E\!\left[R_t(a)^2\mid\mathcal H_{t-1}\right]
&\leq
\left\{
\mathbb E\!\left[|g_a(X_t)|^q\mid\mathcal H_{t-1}\right]
\right\}^{2/q}
\notag\\
&\leq
C\{1+\mathcal V(X_{t-1})\}^{2/q}
\label{eq:boundary-variance-lyapunov-bound}
\end{align}
for a finite constant \(C\) independent of \(t\).

At the \(k\)-th selected cycle,
\(N_{\tau_k(a)}(a)=k\).  Hence condition
\eqref{ass:polynomial-selection-frequency} implies
\begin{equation}
\tau_k(a)
\leq
C_a' k^{1/\kappa_a}
\qquad\text{eventually, almost surely}
\label{eq:selected-time-growth-bound}
\end{equation}
for some finite \(C_a'\).  Choose \(\delta>0\) sufficiently small that
\[
\frac{2(1+\delta)}{q\kappa_a}<1,
\]
which is possible because \(\kappa_a>2/q\).  Using
\eqref{eq:lyapunov-path-growth}--\eqref{eq:selected-time-growth-bound} in
\eqref{eq:boundary-variance-lyapunov-bound},
\[
\mathbb E\!\left[
R_{\tau_k(a)}(a)^2
\,\middle|\,
\mathcal H_{\tau_k(a)-1}
\right]
=
o\!\left(
 k^{2(1+\delta)/(q\kappa_a)}
\right)
\qquad\text{almost surely}.
\]
The exponent is strictly smaller than one, and hence
\[
\sum_{k=1}^{\infty}
\frac{
\mathbb E[R_{\tau_k(a)}(a)^2\mid
\mathcal H_{\tau_k(a)-1}]
}{k^2}
<\infty
\qquad\text{almost surely}.
\]
This is exactly \textnormal{(B3)}.
\end{proof}

\begin{remark}[Queueing interpretation of the common conditions]
\label{rem:common-lyapunov-queueing}
Condition \eqref{ass:common-fixed-D-drift} is a uniform stability condition
for the family of admissible actions at learning-cycle boundaries.  Its
role is stronger than merely requiring each frozen-action queue to be
stable: it controls moments along the persistent trajectory generated by
an adaptive sequence of actions.  Conditions of this type are standard in
online queueing optimization, where the admissible control set is typically
kept uniformly away from the instability boundary.

The drift condition \eqref{ass:common-fixed-D-drift} is a geometric drift
condition at fixed-$D$ cycle boundaries, not a consequence of ordinary
frozen-action stability alone.  In suitable light-tailed GI/G/1 settings it
may be verified by a common workload Lyapunov function or by comparison with
a uniformly stable dominating queue, together with the required uniform
moment control.  Condition \eqref{ass:regenerative-residual-q-moment} is a
separate $q$th-moment requirement on the centered residual busy-cycle cost;
for queue-length holding costs this requires suitable moments of the residual
busy-cycle duration and the area under the queue-length process.

For genuinely heavy-tailed primitives, such as lognormal service times, a
geometric drift condition of the form \eqref{ass:common-fixed-D-drift} need
not hold even though the queue is stable and many polynomial moments are
finite.  The corollary should therefore be viewed as one convenient
verification route, especially for sufficiently light-tailed models, rather
than as a general heavy-tail result.  Heavy-tailed models may instead be
handled through the direct conditions \textnormal{(B2)}--\textnormal{(B3)},
the selected-cycle moment route of Corollary~\ref{cor:block-moment-growth}, or
future subgeometric-drift extensions.  In particular, sufficiently high
finite regenerative moments alone do not imply the geometric drift condition
used in this corollary.
\end{remark}

\begin{remark}[Why action blocks matter]
\label{rem:why-action-blocks}
Theorem~\ref{thm:transient-bias-blocks} shows that the relevant measure of
persistent transient contamination is not the total number of fixed-\(D\)
cycles, but the number of separate operating blocks of action \(a\).
Within a contiguous block, the bias-function boundary increments telescope, so
only the state at the beginning and the state at the end of the block
remain.  Thus repeated fixed-\(D\) cycles under the same action can
accumulate arbitrarily large operating time while generating only one pair
of boundary terms.
\end{remark}


\subsection{Stagewise Top-$K$ consistency and guardian retention}
A coarse-to-fine procedure introduces a second asymptotic index: as the total
schedule index $N$ grows, every fixed early pruning stage must itself receive more
information. Otherwise a nonvanishing probability of irreversible early
mis-ranking can remain. The triangular-array result below supplies a
finite-sample error bound under Lyapunov-weighted variance control; the
guardian theorem turns that bound into a probability of preserving a path to
a master-grid optimizer. Almost-sure infinite-visit assumptions are not imposed
on actions discarded in a finite-budget run.

\subsubsection{Stagewise blocked-SAA consistency}
\label{subsec:stagewise-blocked-saa}

Let the fixed-depth refinement hierarchy have finite candidate grids
\[
\mathcal U_0,\ldots,\mathcal U_L,
\qquad
\mathcal U:=\bigcup_{\ell=0}^L\mathcal U_\ell\subseteq\mathcal A.
\]
The hierarchy is held fixed in the present asymptotic result. For each
schedule index $N$, let
\(\{\mathcal H_t^{(N)}\}_{t\geq0}\) denote the history of the corresponding
persistent learning experiment, let $a_t^{(N)}\in\mathcal U$ be
$\mathcal H_{t-1}^{(N)}$-measurable, and define
\[
I_t^{(N)}(a):=\mathbf 1\{a_t^{(N)}=a\}.
\]
Suppose that for every $a\in\mathcal U$, whenever $a_t^{(N)}=a$,
\begin{equation}
G_t^{(N)}-f(a)T_t^{(N)}
=
g_a(X_{t-1}^{(N)})-g_a(X_t^{(N)})+U_t^{(N)}(a),
\label{eq:triangular-cycle-decomposition}
\end{equation}
where
\begin{equation}
\mathbb E\!\left[
U_t^{(N)}(a)\mid\mathcal H_{t-1}^{(N)}
\right]=0.
\label{eq:triangular-innovation-zero-mean}
\end{equation}
As before, the innovation may be extended by zero outside
$\{a_t^{(N)}=a\}$.  The representation
\eqref{eq:triangular-cycle-decomposition} is the row-wise counterpart of
\eqref{eq:learning-cycle-regenerative-decomposition}.

Fix a stage $\ell\in\{0,\ldots,L\}$.  Let $\sigma_{\ell,N}$ be the cycle at
which the stage-$\ell$ ranking is made, and let $A_{\ell,N}\subseteq\mathcal
U_\ell$ be the active set at that stage.  For $a\in\mathcal U$, define the
cumulative number of sampled cycles, operating time, and observed cost up to
that decision time by
\begin{align}
n_{\ell,N}(a)
&:=\sum_{t=1}^{\sigma_{\ell,N}}I_t^{(N)}(a),
\label{eq:stage-count}
\\
S_{\ell,N}(a)
&:=\sum_{t=1}^{\sigma_{\ell,N}}I_t^{(N)}(a)T_t^{(N)},
\label{eq:stage-operating-time}
\\
C_{\ell,N}(a)
&:=\sum_{t=1}^{\sigma_{\ell,N}}I_t^{(N)}(a)G_t^{(N)}.
\label{eq:stage-cost}
\end{align}
Thus observations collected under an action at earlier stages are carried
forward automatically.  For $a\in A_{\ell,N}$, set
\begin{equation}
\widehat f_{\ell,N}(a)
:=
\frac{C_{\ell,N}(a)}{S_{\ell,N}(a)}.
\label{eq:stage-saa-estimator}
\end{equation}

\begin{lemma}[Finite-sample stagewise blocked-SAA error]
\label{lem:stagewise-blocked-saa}
Fix a stage $\ell$ with nonempty active set $A_{\ell,N}$. In addition to
\eqref{eq:triangular-cycle-decomposition}--\eqref{eq:triangular-innovation-zero-mean},
assume the following conditions. Constants are uniform in the schedule index $N$
and in the actions unless otherwise indicated.

\medskip
\noindent\textnormal{(S1) Bounded ranking time.}
The ranking time $\sigma_{\ell,N}$ is a stopping time and
$\sigma_{\ell,N}\le r_{\ell,N}$ almost surely, for a deterministic integer
$r_{\ell,N}$.
\label{ass:stage-horizon}

\medskip
\noindent\textnormal{(S2) Blocked sampling and possible boundaries.}
There are deterministic integers $m=m_{\ell,N}\ge1$, $j=j_{\ell,N}\ge1$,
and $\Lambda=\Lambda_{\ell,N}\ge1$ such that every active action has
$n_{\ell,N}(a)\ge m$. Up to $\sigma_{\ell,N}$ each action is operated in
at most $j$ maximal contiguous blocks. A deterministic set
$\mathcal B_{\ell,N}\subseteq\{0,\ldots,r_{\ell,N}\}$, of cardinality at
most $\Lambda$, contains both $b-1$ and $e$ for every possible action block
$[b,e]$ up to the ranking time. One can always use all indices
$\{0,\ldots,r_{\ell,N}\}$; a smaller deterministic covering set gives a
sharper bound.

\medskip
\noindent\textnormal{(S3) Common Lyapunov control.}
There are $\mathcal V\ge1$, $0<\rho<1$, $b\ge0$, and $v_0<\infty$ such that
\begin{equation}
\mathbb E[\mathcal V(X_t^{(N)})\mid\mathcal H_{t-1}^{(N)}]
\le\rho\mathcal V(X_{t-1}^{(N)})+b,
\qquad
\sup_N\mathbb E\mathcal V(X_0^{(N)})\le v_0.
\label{ass:stage-lyapunov}
\end{equation}
Consequently $\sup_{N,t}\mathbb E\mathcal V(X_t^{(N)})\le\bar v$, where
$\bar v:=\max\{v_0,b/(1-\rho)\}$.

\medskip
\noindent\textnormal{(S4) Bias-function growth.}
For some $q\ge1$ and $\kappa_g<\infty$,
\begin{equation}
|g_a(x)|^q\le\kappa_g\mathcal V(x),
\qquad a\in\mathcal U.
\label{ass:triangular-bias-moment}
\end{equation}

\medskip
\noindent\textnormal{(S5) Lyapunov-weighted innovation variance.}
For some $\kappa_U<\infty$,
\begin{equation}
\mathbb E[\{U_t^{(N)}(a)\}^2\mid\mathcal H_{t-1}^{(N)}]
\le\kappa_U\mathcal V(X_{t-1}^{(N)})
\quad\text{on }\{a_t^{(N)}=a\}.
\label{ass:triangular-innovation-second-moment}
\end{equation}

For $c>0$ define the duration-failure probability
\begin{equation}
\eta_{\ell,N}(c):=
\mathbb P\left(\min_{a\in A_{\ell,N}}
\frac{S_{\ell,N}(a)}{n_{\ell,N}(a)}<c\right).
\label{eq:duration-failure}
\end{equation}
Then, for every $\varepsilon>0$, the uniform estimation error satisfies
\begin{align}
e_{\ell,N}&:=\max_{a\in A_{\ell,N}}
|\widehat f_{\ell,N}(a)-f(a)|,
\label{eq:stage-uniform-consistency}\\
\mathbb P(e_{\ell,N}\ge\varepsilon)
&\le B_{\ell,N}(\varepsilon;c),
\label{eq:stage-finite-bound}\\
B_{\ell,N}(\varepsilon;c)
&:=\eta_{\ell,N}(c)
+\frac{4\kappa_U\bar v\,r_{\ell,N}}{(\varepsilon c m_{\ell,N})^2}
\notag\\[-2pt]
&\hspace{1cm}
+|\mathcal U|\Lambda_{\ell,N}\kappa_g\bar v
\left(\frac{4j_{\ell,N}}{\varepsilon c m_{\ell,N}}\right)^q.
\label{eq:stage-bound-definition}
\end{align}
The bound may of course be truncated at one.
\end{lemma}

\begin{proof}
Suppress $N$ and $\ell$ where unambiguous. Summing the cycle representation
gives, for each fixed $a\in\mathcal U$,
\begin{equation}
C(a)-f(a)S(a)=H(a)+Z(a),
\label{eq:stage-error-decomposition}
\end{equation}
where
\begin{align}
H(a)&:=\sum_{t=1}^{\sigma}I_t(a)
\{g_a(X_{t-1})-g_a(X_t)\},
\label{eq:stage-boundary-term}\\
Z(a)&:=\sum_{t=1}^{\sigma}I_t(a)U_t(a).
\label{eq:stage-innovation-term}
\end{align}
For the action-$a$ blocks $[b_i,e_i]$, $i=1,\ldots,J\le j$, telescoping yields
\[
H(a)=\sum_{i=1}^{J}\{g_a(X_{b_i-1})-g_a(X_{e_i})\},
\qquad |H(a)|\le2jM,
\]
where $M:=\max_{a\in\mathcal U,\,s\in\mathcal B_{\ell,N}}|g_a(X_s)|$.
Because the covering set $\mathcal B_{\ell,N}$ is deterministic, (S3)--(S4)
give
\[
\mathbb E M^q\le
\sum_{a\in\mathcal U}\sum_{s\in\mathcal B_{\ell,N}}
\mathbb E|g_a(X_s)|^q
\le |\mathcal U|\Lambda\kappa_g\bar v.
\]
This argument does not apply a deterministic-time moment bound directly at
an arbitrary data-dependent boundary time.

For the innovation term define, for $1\le t\le r$,
\[
D_t(a):=\mathbf1\{t\le\sigma\}I_t(a)U_t(a).
\]
Since $\{t\le\sigma\}=\{\sigma>t-1\}\in\mathcal H_{t-1}$ and $I_t(a)$
is predictable, $\mathbb E[D_t(a)\mid\mathcal H_{t-1}]=0$.
The terms are square integrable by (S3) and (S5), so the stopped sum
$Z(a)=\sum_{t=1}^rD_t(a)$ is centered and martingale orthogonality applies.
Summing over actions before discarding selection indicators gives
\begin{align}
\sum_{a\in\mathcal U}\mathbb E Z(a)^2
&=\sum_{t=1}^r\sum_{a\in\mathcal U}\mathbb E D_t(a)^2
\notag\\
&\le\kappa_U\sum_{t=1}^r
\mathbb E\left[\mathbf1\{t\le\sigma\}
\mathcal V(X_{t-1})\sum_{a\in\mathcal U}I_t(a)\right]
\notag\\
&\le\kappa_U\bar v r.
\label{eq:stage-innovation-variance}
\end{align}
Here exactly one action is operated in each cycle. Thus a union bound and
Chebyshev's inequality give
$\mathbb P(\max_a|Z(a)|\ge x)\le\kappa_U\bar v r/x^2$.

Outside the duration-failure event, $S(a)\ge cm$ for every active action, and
\[
e_{\ell,N}\le\frac{2jM+\max_a|Z(a)|}{cm}.
\]
If this quantity is at least $\varepsilon$, then either
$M\ge\varepsilon cm/(4j)$ or $\max_a|Z(a)|\ge\varepsilon cm/2$.
Applying the two tail bounds and adding $\eta_{\ell,N}(c)$ proves the claim.
\end{proof}

The three terms in \eqref{eq:stage-bound-definition} have distinct operational
meanings. The duration term prevents too little clock time from being collected
despite enough completed cycles; the innovation term controls centered
cycle-to-cycle noise at the data-dependent ranking time; and the boundary term
charges switching between separated action blocks. The lemma is therefore the
bridge between the persistent-trajectory queueing analysis and a finite-stage
ranking decision. It does not yet guarantee that a useful refinement branch is
retained; that step requires the guardian margin introduced below.

\begin{remark}[Imperfect observations at stagewise rankings]
\label{rem:stagewise-observation-error}
Suppose the implemented procedure uses
$\widetilde G_t^{(N)}=G_t^{(N)}+E_t^{(N)}$ in its cumulative cost
statistics and rankings. Let the histories, selected actions, active
sets, and ranking times describe this actual procedure, and assume
that the hypotheses of Lemma~\ref{lem:stagewise-blocked-saa} hold for
the true costs $G_t^{(N)}$ along its trajectory. In particular, actions
remain predictable with respect to the histories used in those
hypotheses.

For $a\in A_{\ell,N}$ define
\[
\widetilde f_{\ell,N}(a)
:=
\frac{\sum_{t=1}^{\sigma_{\ell,N}}
 I_t^{(N)}(a)\widetilde G_t^{(N)}}{S_{\ell,N}(a)},
\]
and let
\begin{equation}
d_{\ell,N}
:=
\max_{a\in A_{\ell,N}}
\frac{\left|\sum_{t=1}^{\sigma_{\ell,N}}
 I_t^{(N)}(a)E_t^{(N)}\right|}{S_{\ell,N}(a)}.
\label{eq:stage-observation-error}
\end{equation}
Here $\widehat f_{\ell,N}$ is computed from the true costs on the
same realized trajectory; it need not be observable. Writing
$\widetilde e_{\ell,N}:=
\max_{a\in A_{\ell,N}}|\widetilde f_{\ell,N}(a)-f(a)|$,
the triangle inequality gives
$\widetilde e_{\ell,N}\le e_{\ell,N}+d_{\ell,N}$. Hence, for every
$\varepsilon>0$,
\begin{equation}
\mathbb P(\widetilde e_{\ell,N}\ge\varepsilon)
\le
\widetilde B_{\ell,N}(\varepsilon;c)
:=
B_{\ell,N}(\varepsilon/2;c)
+\mathbb P(d_{\ell,N}\ge\varepsilon/2).
\label{eq:stage-observation-error-bound}
\end{equation}
Thus stagewise consistency transfers if $d_{\ell,N}\to0$ in probability
at every fixed stage, in addition to the conditions ensuring
$e_{\ell,N}\to0$.

The guardian arguments in Theorem~\ref{thm:topk-guardian-retention}
apply to rankings formed from $\widetilde f_{\ell,N}$ with the same
structural hypotheses and with $B_{\ell,N}$ replaced by
$\widetilde B_{\ell,N}$. The contribution of a nonautomatic stage
then becomes
\[
\widetilde b_{\ell,N}
=
B_{\ell,N}(\underline\Delta_\ell/4;c_\ell)
+\mathbb P(d_{\ell,N}\ge\underline\Delta_\ell/4),
\]
where $\underline\Delta_\ell$ is the margin defined in that theorem;
automatic stages still contribute zero. The final recommendation bound
uses $\widetilde B_{L,N}(\varepsilon;c_L)$.
\end{remark}

\begin{remark}[Connecting the variance bound to the Poisson representation]
\label{rem:weighted-variance-verification}
A state-dependent bound is more suitable for an unbounded queue than a
constant conditional-variance bound, but it remains a separate hypothesis.
Here is one sufficient verification. Under a fixed action and inherited state
$x$, put $Y=G-f(a)T$. Suppose
\[
\mathbb E_x^a Y^2\le\kappa_Y\mathcal V(x),\qquad
|g_a(x)|^2\le\kappa_{g,2}\mathcal V(x).
\]
The Poisson representation implies
$\mathbb E_x^a[Y+g_a(X_1)]=g_a(x)$, so
\begin{align*}
\mathbb E_x^a U^2
&=\operatorname{Var}_x^a(Y+g_a(X_1))\\
&\le2\kappa_Y\mathcal V(x)
+2\kappa_{g,2}\{\rho\mathcal V(x)+b\}\\
&\le\{2\kappa_Y+2\kappa_{g,2}(\rho+b)\}\mathcal V(x).
\end{align*}
The regenerative construction and Jensen's inequality can supply the growth
bound on $g_a$ from the corresponding boundary-reward moments. A bound on
the cycle reward $Y^2$ must still be checked. Neither frozen-action stability
alone nor the Lyapunov drift alone establishes (S5). The assumption $q\ge1$
in Lemma~\ref{lem:stagewise-blocked-saa} concerns its boundary estimate;
it does not relax the $q>2$ requirements of the earlier almost-sure results.
\end{remark}

The duration condition needed for consistency is
\begin{equation}
\textnormal{(S6)}\qquad \eta_{\ell,N}(c_\ell)\longrightarrow0
\quad\text{for some }c_\ell>0.
\label{ass:stage-positive-duration}
\end{equation}
It can be verified without independence of the adaptively selected durations.

\begin{proposition}[Duration control at a data-dependent ranking time]
\label{prop:stage-duration}
Assume (S1)--(S2). Suppose, uniformly over rows and histories, that the
conditional law of $T_t^{(N)}$ stochastically dominates a nonnegative random
variable $\xi$ with $0<\mu:=\mathbb E\xi<\infty$. Fix $0<c<\mu$ and write
$\psi(\theta):=\mathbb E e^{-\theta\xi}$. For sufficiently small $\theta>0$,
\begin{equation}
\kappa(\theta,c):=e^{\theta c}\psi(\theta)<1,
\qquad
\eta_{\ell,N}(c)\le |\mathcal U|\kappa(\theta,c)^{m_{\ell,N}}.
\label{eq:duration-exponential-bound}
\end{equation}
In particular, $m_{\ell,N}\to\infty$ implies (S6).
\end{proposition}

\begin{proof}
The right derivative at zero of $e^{\theta c}\psi(\theta)$ is $c-\mu<0$,
which proves the existence of $\theta$. Suppress $N$, fix $a$, and define
\[
N_t(a):=\sum_{s=1}^tI_s(a),\qquad
L_t(a):=\frac{\exp\{-\theta\sum_{s=1}^tI_s(a)T_s\}}
{\psi(\theta)^{N_t(a)}}.
\]
Conditional stochastic domination implies
$\mathbb E[e^{-\theta T_t}\mid\mathcal H_{t-1}]\le\psi(\theta)$.
Predictability of $I_t(a)$ therefore makes $L_t(a)$ a nonnegative
supermartingale with $L_0(a)=1$. At the bounded stopping time $\sigma$,
$\mathbb E L_\sigma(a)\le1$. On
$\{a\in A_{\ell,N},\ S(a)<cn(a)\}$, (S2) gives
\[
L_\sigma(a)>\kappa(\theta,c)^{-n(a)}
\ge\kappa(\theta,c)^{-m}.
\]
Markov's inequality bounds this event by $\kappa(\theta,c)^m$; a union bound
over the fixed set $\mathcal U$ proves the result. This is the standard
nonnegative-supermartingale method; see \cite{HowardEtAl2020} for its broader
use. No independence after adaptive selection, and no conditioning on a
random selected sample size, is required.
\end{proof}

\begin{remark}[Poisson designated events]
\label{rem:poisson-duration}
If each cycle ends at its $D$th external Poisson arrival, the action is held
fixed during the cycle, and $\lambda(a)\le\bar\lambda$, then conditional on
the past $T_t$ has law $\operatorname{Gamma}(D,\lambda(a_t))$ (shape and
rate). It dominates $\xi\sim\operatorname{Gamma}(D,\bar\lambda)$.
Taking $c=D/(2\bar\lambda)$ and $\theta=\bar\lambda$ gives
\begin{equation}
\eta_{\ell,N}(c)
\le |\mathcal U|\exp\{-D(\log2-1/2)m_{\ell,N}\}.
\label{eq:poisson-duration-explicit}
\end{equation}
A simple counting process with uniformly bounded conditional event intensity
admits the same argument by a Poisson envelope. For example, exponential
service completions have intensity at most $\bar s\bar\nu$ when at most
$\bar s$ servers operate at rates at most $\bar\nu$; use
$\operatorname{Gamma}(D,\bar s\bar\nu)$ as the lower bound. Batch events
counted as several individual events require a different envelope.
For nonexponential completion-based cycles, a suitable duration lower bound
requires separate verification; ordinary service-time moments do not suffice.
External-Poisson-arrival cycles avoid this residual-service issue, even with
nonexponential service. Their duration verification does not, by itself,
verify the other hypotheses of Lemma~\ref{lem:stagewise-blocked-saa}.
\end{remark}

\begin{corollary}[Stagewise consistency for increasing budgets]
\label{cor:stagewise-asymptotic}
Under (S1)--(S6), if
\begin{equation}
\frac{r_{\ell,N}}{m_{\ell,N}^2}\longrightarrow0,
\qquad
\frac{j_{\ell,N}\Lambda_{\ell,N}^{1/q}}{m_{\ell,N}}\longrightarrow0,
\label{ass:triangular-stage-growth}
\end{equation}
then $e_{\ell,N}\to0$ in probability. The same conclusion holds under
(S1)--(S5), the duration domination in Proposition~\ref{prop:stage-duration},
and \eqref{ass:triangular-stage-growth}.
\end{corollary}
\begin{proof}
Each term of \eqref{eq:stage-bound-definition} vanishes. For the second
formulation, nonemptiness of the active set gives $r_{\ell,N}\ge m_{\ell,N}$,
so the first growth condition forces $m_{\ell,N}\to\infty$; then
Proposition~\ref{prop:stage-duration} supplies (S6).
\end{proof}

\begin{remark}[Verification for the blocked CTF schedule]
\label{rem:triangular-block-simplification}
For the algorithm in Section~\ref{subsec:practical-topk}, the hierarchy and
depth are fixed and the cumulative targets are deterministic and
nondecreasing. By stage $\ell$, every action has at most $m_{\ell,N}$ cycles,
so one may take $r_{\ell,N}=|\mathcal U|m_{\ell,N}$ and
$j_{\ell,N}\le\ell+1$. There are finitely many possible sequences of active
sets and within-stage block orders, independently of $N$. For each such
sequence, the cycle counts at block boundaries are determined by the stage
targets and the previous sampling of each action. Taking the union over all
these sequences produces a deterministic covering set $\mathcal B_{\ell,N}$
with $\Lambda_{\ell,N}\le\Lambda_\ell<\infty$, independently of $N$.
Its indices change with $N$, but their number is bounded. Thus
\eqref{ass:triangular-stage-growth} follows from $m_{\ell,N}\to\infty$
even when $q=1$. A small number of realized blocks alone would not justify
this conclusion for an arbitrary adaptive within-block stopping rule.

These guarantees assume that the prescribed sampling stages are completed.
A hard elapsed-time cutoff that interrupts a stage requires an additional
bound on the probability of failing to reach its target; it is not covered
merely by the lower-duration estimate (S6).
\end{remark}

\begin{remark}[Role of the almost-sure theory]
\label{rem:triangular-connection}
Theorem~\ref{thm:direct-saa-consistency} describes an infinite operating
trajectory with persistent sampling of the actions under consideration.
An irreversibly pruned action in a single finite-budget CTF run does not
satisfy that premise. We use the almost-sure results to explain and verify
direct SAA under persistent sampling, not to assert an almost-sure guarantee
for the pruning algorithm. The applicable algorithmic limit is instead
Corollary~\ref{cor:stagewise-asymptotic}: across experiments with increasing
budget, every fixed stage receives increasing information before pruning.
The shared Poisson/boundary decomposition connects the two analyses.
\end{remark}

\subsubsection{Top-$K$ guardian retention}
\label{subsec:topk-guardian}

For $a\in\mathcal U_\ell$, $\ell<L$, let
$\operatorname{Ch}_\ell(a)\subseteq\mathcal U_{\ell+1}$ be the set of
children generated when $a$ is refined.  Throughout this subsection, assume that $\mathcal U_L$ is nonempty and
$\operatorname{Ch}_\ell(a)\neq\varnothing$ for every
$a\in\mathcal U_\ell$ and $\ell<L$.  Hence every active set generated from
a nonempty retained set is nonempty.  Define the master-grid descendant sets
recursively by
\begin{equation}
\operatorname{Desc}_L(a):=\{a\},
\qquad a\in\mathcal U_L,
\label{eq:master-descendant-terminal}
\end{equation}
and, for $\ell<L$,
\begin{equation}
\operatorname{Desc}_\ell(a)
:=
\bigcup_{b\in\operatorname{Ch}_\ell(a)}
\operatorname{Desc}_{\ell+1}(b).
\label{eq:master-descendant-recursion}
\end{equation}
Nonempty child sets guarantee continuation, but do not guarantee coverage
of the master grid. A convenient objective-independent coverage check is
\[
\bigcup_{a\in\mathcal U_\ell}\operatorname{Ch}_\ell(a)
=\mathcal U_{\ell+1},\qquad 0\le\ell<L.
\]
It implies $\bigcup_{a\in\mathcal U_0}\operatorname{Desc}_0(a)=\mathcal U_L$.
The theorem below needs only the weaker condition that at least one
master-grid optimizer is reachable from $\mathcal U_0$; this is not implied
by nested grids alone.

Let
\begin{equation}
f_L^*
:=
\min_{a\in\mathcal U_L}f(a),
\qquad
\mathcal U_L^*
:=
\operatorname*{argmin}_{a\in\mathcal U_L}f(a).
\label{eq:master-grid-optimal-set}
\end{equation}
Thus multiple master-grid optimizers are allowed.

For a nonempty active set $A\subseteq\mathcal U_\ell$, define its guardian set
by
\begin{equation}
\operatorname{Gd}_\ell(A)
:=
\left\{
a\in A:
\operatorname{Desc}_\ell(a)\cap\mathcal U_L^*\neq\varnothing
\right\}.
\label{eq:guardian-set}
\end{equation}
An action is therefore a guardian precisely when retaining it leaves at least
one master-grid optimizer reachable through subsequent refinement.

Suppose $G:=\operatorname{Gd}_\ell(A)\neq\varnothing$, and let
\begin{equation}
f_{\ell,G}(A)
:=
\min_{g\in G}f(g).
\label{eq:best-guardian-value}
\end{equation}
For an integer $K\geq1$, let $\nu_{\ell,K}(A)$ denote the $K$th order
statistic of the action-indexed list
\[
\bigl(f(a)\bigr)_{a\in A\setminus G}.
\]
Each distinct action contributes one entry, so equal objective values are
counted with multiplicity. For example, nonguardian costs $1,1,3$ have
second order statistic $1$. The same convention applies to empirical
order statistics; the fixed tie-breaking rule orders the actions.
If $|A\setminus G|<K$, set $\nu_{\ell,K}(A):=+\infty$.  Define the guardian
separation margin
\begin{equation}
\Delta_{\ell,K}(A)
:=
\nu_{\ell,K}(A)-f_{\ell,G}(A).
\label{eq:guardian-separation-margin}
\end{equation}

At stage $\ell<L$, retain $k_\ell(A_{\ell,N})$ parents, where
\[
k_\ell(A):=\min\{K_\ell,|A|\},\qquad K_\ell\ge1.
\]
Let $\widehat{\mathcal R}_{\ell,N}$ be the corresponding empirical top set
under the fixed tie-breaking rule, and put
\begin{equation}
A_{\ell+1,N}:=\bigcup_{a\in\widehat{\mathcal R}_{\ell,N}}
\operatorname{Ch}_\ell(a).
\label{eq:topk-child-update}
\end{equation}
The relevant family of active sets includes every possible retention that
keeps a guardian, not just the exact-objective refinement path. Define
\begin{align}
\mathfrak A_0^+&:=\{\mathcal U_0\},\notag\\
\mathfrak A_{\ell+1}^+
&:=\left\{\bigcup_{a\in R}\operatorname{Ch}_\ell(a):
\begin{array}{l}
A\in\mathfrak A_\ell^+,\ R\subseteq A,\\
|R|=k_\ell(A),\ R\cap\operatorname{Gd}_\ell(A)\ne\varnothing
\end{array}\right\}.
\label{eq:guardian-active-family}
\end{align}
This family is finite. Its enumeration can nevertheless be large; a small
retention count alone does not guarantee a cheap exhaustive audit.

\begin{theorem}[Finite-sample Top-$K$ guardian retention]
\label{thm:topk-guardian-retention}
Assume $A_{0,N}=\mathcal U_0$ and
$\operatorname{Gd}_0(\mathcal U_0)\ne\varnothing$. At each stage $\ell<L$
and for every $A\in\mathfrak A_\ell^+$ with at least $k_\ell(A)$
nonguardians, assume
\begin{equation}
\Delta_{\ell,k_\ell(A)}(A)>0.
\label{ass:guardian-positive-margin}
\end{equation}
If such sets exist, let
\begin{equation}
\underline\Delta_\ell:=
\min_{\substack{A\in\mathfrak A_\ell^+\\
|A\setminus\operatorname{Gd}_\ell(A)|\ge k_\ell(A)}}
\Delta_{\ell,k_\ell(A)}(A)>0.
\label{eq:uniform-guardian-margin}
\end{equation}
Otherwise call the stage automatic. Define
\begin{equation}
F_{\ell,N}:=\{\operatorname{Gd}_\ell(A_{\ell,N})\ne\varnothing\}.
\label{eq:guardian-survival-event}
\end{equation}
Then $F_{\ell+1,N}\subseteq F_{\ell,N}$, and on $F_{\ell,N}$,
$A_{\ell,N}\in\mathfrak A_\ell^+$. At a nonautomatic stage,
\begin{equation}
F_{\ell,N}\cap\{e_{\ell,N}<\underline\Delta_\ell/2\}
\subseteq F_{\ell+1,N};
\label{eq:one-stage-guardian-retention}
\end{equation}
automatic stages never lose a guardian.

If (S1)--(S5) hold at each stage, choose $c_\ell>0$ and put
$b_{\ell,N}:=B_{\ell,N}(\underline\Delta_\ell/2;c_\ell)$ at nonautomatic
stages and $b_{\ell,N}:=0$ at automatic stages. Then
\begin{equation}
\mathbb P(F_{L,N}^c)\le\sum_{\ell=0}^{L-1}b_{\ell,N}.
\label{eq:guardian-loss-finite-bound}
\end{equation}
For $\widehat a_N\in\operatorname*{argmin}_{a\in A_{L,N}}
\widehat f_{L,N}(a)$ and every $\varepsilon>0$,
\begin{equation}
\mathbb P(f(\widehat a_N)-f_L^*>2\varepsilon)
\le\sum_{\ell=0}^{L-1}b_{\ell,N}+B_{L,N}(\varepsilon;c_L).
\label{eq:recommendation-finite-bound}
\end{equation}
No independence between stages is required.
\end{theorem}

\begin{proof}
If a child at stage $\ell+1$ is a guardian, any retained parent producing
that child has a master-grid optimizer among its descendants. That parent
is a guardian at stage $\ell$. This proves $F_{\ell+1,N}\subseteq F_{\ell,N}$
and that the retained set meets $\operatorname{Gd}_\ell(A_{\ell,N})$
whenever $F_{\ell+1,N}$ occurs. Induction from $A_{0,N}=\mathcal U_0$
using \eqref{eq:guardian-active-family} proves membership in
$\mathfrak A_\ell^+$ on $F_{\ell,N}$.

On $F_{\ell,N}$ write $A=A_{\ell,N}$, $G=\operatorname{Gd}_\ell(A)$,
and $k=k_\ell(A)$. If $|A\setminus G|<k$, every retained $k$-set contains
a guardian. Otherwise choose $g^*\in\operatorname*{argmin}_{g\in G}f(g)$
and put $\Delta=\Delta_{\ell,k}(A)\ge\underline\Delta_\ell$.
Uniform error smaller than $\Delta/2$ gives
\[
\widehat f_{\ell,N}(g^*)<f_{\ell,G}(A)+\Delta/2,
\qquad
\widehat\nu_{\ell,k}(A)>f_{\ell,G}(A)+\Delta/2,
\]
where $\widehat\nu_{\ell,k}$ is the $k$th empirical nonguardian order
statistic. The second inequality follows from
$\widehat\nu_{\ell,k}\ge\nu_{\ell,k}-e_{\ell,N}$.
Thus $k$ nonguardians cannot all rank ahead of $g^*$, and the retained set
contains a guardian. At least one of its children is again a guardian,
proving \eqref{eq:one-stage-guardian-retention}.

Because $F_{0,N}$ holds, loss of every guardian by the final stage has a
first loss stage. Consequently
\begin{equation}
\mathbb P(F_{L,N}^c)\le
\sum_{\ell\ \mathrm{nonautomatic}}
\mathbb P\big(F_{\ell,N}\cap
\{e_{\ell,N}\ge\underline\Delta_\ell/2\}\big).
\label{eq:guardian-loss-union-bound}
\end{equation}
Lemma~\ref{lem:stagewise-blocked-saa} proves
\eqref{eq:guardian-loss-finite-bound}. Finally, on $F_{L,N}$ choose
$a^*\in A_{L,N}\cap\mathcal U_L^*$. Then
\[
f(\widehat a_N)\le\widehat f_{L,N}(\widehat a_N)+e_{L,N}
\le\widehat f_{L,N}(a^*)+e_{L,N}\le f_L^*+2e_{L,N}.
\]
Applying the lemma once more proves \eqref{eq:recommendation-finite-bound}.
\end{proof}

\paragraph{Completed schedules and clock-time cutoffs.}
The bound concerns the completed prescribed schedule. A capped procedure
follows the same sampling rule until completion or a hard elapsed-time
cutoff $t_N^{\max}$. Couple it to the full scheduled run.
Let $\widehat a_N^{\rm cap}$ equal $\widehat a_N$ when
$H_N\le t_N^{\max}$, and use a specified feasible fallback otherwise.
The event of an inaccurate capped recommendation is contained in the union
of an inaccurate full recommendation and noncompletion. Hence
\[
\mathbb P\{f(\widehat a_N^{\rm cap})-f_L^*>2\varepsilon\}
\le \sum_{\ell=0}^{L-1}b_{\ell,N}
+B_{L,N}(\varepsilon;c_L)+\mathbb P(H_N>t_N^{\max}).
\]
The last probability requires an additional completion-time bound; the
lower-duration condition (S6) does not control it.
Lemma~\ref{lem:stagewise-blocked-saa} is applied to the full run, without
conditioning on completion. Clipping its ranking time at a clock cutoff
does not preserve the minimum sample counts in (S2).

\begin{corollary}[Master-grid consistency]
\label{cor:topk-master-grid-consistency}
Suppose the hierarchy and positive guardian margins in
Theorem~\ref{thm:topk-guardian-retention} are fixed. If
$e_{\ell,N}\to0$ in probability for each $\ell=0,\ldots,L$, then
\begin{equation}
\mathbb P(F_{L,N})\to1,\qquad
f(\widehat a_N)-f_L^*\to0\quad\text{in probability}.
\label{eq:master-grid-value-consistency}
\end{equation}
In fact $\mathbb P(\widehat a_N\in\mathcal U_L^*)\to1$.
Corollary~\ref{cor:stagewise-asymptotic} supplies sufficient conditions for
the required stagewise errors to vanish.
\end{corollary}
\begin{proof}
The first-loss bound \eqref{eq:guardian-loss-union-bound} tends to zero.
On $F_{L,N}$ the proof above gives $f(\widehat a_N)-f_L^*\le2e_{L,N}$.
If any master-grid actions are suboptimal, finiteness of the grid implies
that their smallest positive gap is strictly positive, so convergence of
the gap to zero implies optimizer identification. If all actions are optimal,
identification is immediate.
\end{proof}

\begin{proposition}[Top-$K$ versus Top-1 at a common active set]
\label{prop:topk-margin-comparison}
Fix $\ell<L$ and an active set $A$ containing a guardian. For
$1\le K\le K'\le|A|$,
\begin{equation}
\Delta_{\ell,K'}(A)\ge\Delta_{\ell,K}(A),
\label{eq:margin-monotonicity}
\end{equation}
with the convention $\nu_{\ell,K}=+\infty$ if fewer than $K$ nonguardians
exist. When both order statistics are finite, the increase equals
$\nu_{\ell,K'}(A)-\nu_{\ell,K}(A)$ and is strict if these values differ.
For the same vector of estimates and the same tie-breaking rule, empirical
top-$K$ sets are nested as $K$ increases. Thus the event of retaining a
guardian is also nested, for every realization of the estimates.

Suppose $\Delta_{\ell,1}(A)<0<\Delta_{\ell,K}(A)$. For estimates on this
fixed set let $e_N(A)=\max_{a\in A}|\widehat f_N(a)-f(a)|$ and
$\gamma=-\Delta_{\ell,1}(A)>0$. If $e_N(A)<\gamma/2$, Top-1 retains no
guardian. If $e_N(A)<\Delta_{\ell,K}(A)/2$, Top-$K$ retains a guardian
(automatically if its margin is infinite). Therefore, when $e_N(A)\to0$
in probability, Top-1 loses all guardians with probability tending to one,
while Top-$K$ keeps one with probability tending to one.
\end{proposition}
\begin{proof}
Nonguardian order statistics are nondecreasing in $K$, whereas the best
guardian value is unchanged. This proves the margin ordering without
subtracting infinite values. A fixed ordering of the estimated costs proves
the nesting assertion. For the last claim, choose a minimum-cost
nonguardian $a^\circ$. Its true cost is $\gamma$ below the best guardian,
so $e_N(A)<\gamma/2$ implies
$\widehat f_N(a^\circ)<\widehat f_N(g)$ for every guardian $g$.
Hence an empirical minimizer is a nonguardian; uniqueness of $a^\circ$
is unnecessary. The Top-$K$ claim follows from the deterministic retention
argument in Theorem~\ref{thm:topk-guardian-retention}.
\end{proof}

\begin{corollary}[Sufficient per-action sample sizes]
\label{cor:guardian-sample-size}
At a nonautomatic stage under Theorem~\ref{thm:topk-guardian-retention},
suppose $r_{\ell,N}\le R_\ell m_{\ell,N}$,
$j_{\ell,N}\le j_\ell$, $\Lambda_{\ell,N}\le\Lambda_\ell$, and
$\eta_{\ell,N}(c_\ell)\le|\mathcal U|\kappa_\ell^{m_{\ell,N}}$
for some $0<\kappa_\ell<1$. For $0<\delta<1$, it suffices to take an
integer $m_{\ell,N}$ at least as large as each of
\begin{align}
&\frac{\log(3|\mathcal U|/\delta)}{-\log\kappa_\ell},
\qquad
\frac{48\kappa_U\bar v R_\ell}
{\delta c_\ell^2\underline\Delta_\ell^2},
\notag\\
&\frac{8j_\ell}{c_\ell\underline\Delta_\ell}
\left(\frac{3|\mathcal U|\Lambda_\ell\kappa_g\bar v}{\delta}\right)^{1/q}
\label{eq:guardian-sample-size}
\end{align}
to ensure $b_{\ell,N}\le\delta$. Applying this with tolerances
$\delta_\ell$ at the nonautomatic stages bounds total guardian loss by
$\sum_\ell\delta_\ell$.
\end{corollary}
\begin{proof}
Substitute $\varepsilon=\underline\Delta_\ell/2$ into
\eqref{eq:stage-bound-definition} and use $r_{\ell,N}\le R_\ell m_{\ell,N}$.
The three requirements make the duration, innovation, and boundary terms
at most $\delta/3$ each.
\end{proof}

\begin{remark}[What the comparison establishes]
\label{rem:why-topk}
Top-$K$ can tolerate up to $K-1$ nonguardians with costs below the best
guardian. At a common active set, when both margins are positive and finite
and the innovation requirement in \eqref{eq:guardian-sample-size} dominates,
the ratio of that sufficient sample-size term for Top-$K$ to that for Top-1
is $(\Delta_{\ell,1}(A)/\Delta_{\ell,K}(A))^2$, holding the other constants
fixed. This is a statement about one term in a sufficient bound, not an
exact sample-complexity ratio. If $\Delta_{\ell,1}(A)=0$, ties can favor a
guardian, so Top-1 failure cannot be concluded from a nonpositive margin
alone.

Different retention schedules generate different future active-set families;
their minimum margins need not be ordered. More retained parents also require
more later sampling. Proposition~\ref{prop:topk-margin-comparison} therefore
does not imply that a complete Top-$K$ run outperforms Top-1 at every equal
total budget. Moreover, the margins and moment constants in
\eqref{eq:guardian-sample-size} are generally unknown: the formula explains
theoretical dependence on accuracy and separation, rather than providing a
calibrated data-only tuning rule.
\end{remark}

\begin{remark}[Scope of a guardian audit]
\label{rem:guardian-audit-scope}
A positive margin along the exact-objective refinement path verifies only
those active sets visited on that path. A noisy ranking can retain a
different collection of parents while still keeping a guardian. Theorem~
\ref{thm:topk-guardian-retention} requires positive margins throughout
\eqref{eq:guardian-active-family}, or a separately justified smaller family
containing every possible guardian-preserving path. Exhaustive enumeration
or a structural argument can establish this condition; one path alone cannot.
For an exhaustive audit, first specify the child maps and identify the full
master-grid optimal set. Compute guardians backwards through the maps, then
generate \eqref{eq:guardian-active-family} forwards, using every retained
$k_\ell(A)$-subset that intersects the guardian set. Deduplicate active sets,
but not equal costs of different actions. At each pruning level report the
number of active sets checked, any automatic cases, and the minimum of the
remaining margins. Terminal active sets need not be enumerated for margin
verification: once the last retained set contains a guardian, the definition
ensures that its children contain a master-grid optimizer.

When objective values are approximate, identifying the optimal descendants
and certifying the sign of each margin require error control. In particular,
a positive margin computed from simulation-based reference values alone is
not a certificate for the true objective. Section~\ref{subsec:liquar-comparison}
provides a full-family audit for the exact-formula $M/M/1$ benchmark.
The retention proof itself needs neither convexity nor differentiability,
and allows overlapping descendant sets and multiple master-grid optimizers.
\end{remark}

\begin{remark}[Continuous-space approximation]
\label{rem:topk-continuous-approximation}
Suppose $\mathcal U_L$ discretizes an underlying compact decision space $B$
and every $x\in B$ lies within distance $\delta_L$ of some point of
$\mathcal U_L$.  If $f$ is uniformly continuous on $B$, define its modulus of
continuity by
\[
\omega_f(\delta)
:=
\sup\{|f(x)-f(y)|:\ x,y\in B,\ \|x-y\|\leq\delta\}.
\]
For $x^*\in\operatorname*{argmin}_{x\in B}f(x)$,
\[
0\leq f_L^*-f(x^*)\leq\omega_f(\delta_L).
\]
Hence, on $F_{L,N}$,
\[
f(\widehat a_N)-f(x^*)
\leq
\omega_f(\delta_L)+2e_{L,N}.
\]
If $f$ is $L_f$-Lipschitz, then
$\omega_f(\delta_L)\leq L_f\delta_L$, recovering the usual Lipschitz bound.
Thus even the continuous-space approximation does not require convexity.  For
a fixed hierarchy the result gives consistency up to deterministic grid error.
If the hierarchy itself is refined with $N$, both the grid radius and guardian
margins may vary with $N$; continuous-space consistency additionally requires
stagewise estimation errors to vanish relative to the corresponding shrinking
guardian margins.
\end{remark}


\subsection{A complete verification for an M/M/1 workload model}
\label{subsec:mm1-workload-verification}

This subsection closes the theory--implementation loop for one concrete
queueing model. We verify the queueing and sampling hypotheses of
Lemma~\ref{lem:stagewise-blocked-saa} for the specified persistent-workload
version of the single-server model used in Section~
\ref{subsec:liquar-comparison}. The result is deliberately model-specific: it
shows that the preceding conditions are verifiable for an implemented queue,
but it is not invoked as automatic verification for the staffing or
heavy-tailed experiments.
The observation is integrated normalized workload. Normalized workload
and queue length have the same stationary mean in this $M/M/1$ model,
but their sample-path integrals differ.

Let $\mathcal U$ be the finite union of the candidate grids and
$a=(\mu,p)\in\mathcal U$. Assume
\begin{equation}
0<\underline\lambda\le\lambda(p)\le\overline\lambda
<\underline\mu\le\mu\le\overline\mu<\infty,
\qquad P:=\max_{a\in\mathcal U}|p|<\infty.
\label{ass:mm1-rate-bounds}
\end{equation}
External arrivals have rate $\lambda(p)$ while action $a$ is held fixed.
Each arrival brings an independent work requirement
$\zeta\sim\operatorname{Exp}(1)$. Total unfinished work $W(s)$ decreases
at speed $\mu$ when positive and increases by $\zeta$ at an arrival.
Thus under a fixed action the calendar service times are exponential
with rate $\mu$. When capacity changes, the existing unfinished work is
preserved and processed at the new speed; it is neither resampled nor
rescaled.

Each learning cycle ends immediately after its $D$th external arrival,
including that arrival's work in the terminal state $X_t$ and charging
its revenue to the action used during that cycle. The next action is
chosen after this state has been observed. Future arrival randomness and
work requirements are independent of the available history conditional
on the selected action. The learner observes the workload and realized
cost; future marks are not revealed in advance. This gives a complete
state for the workload dynamics and costs, even though individual job
identities are not retained. For one cycle, started from $w$ and run
under $a$, define
\begin{equation}
G=\int_0^T W(s)\,ds+\mu T-pD,
\qquad
f(a)=\frac{\lambda(p)}{\mu-\lambda(p)}+\mu-p\lambda(p).
\label{eq:mm1-workload-cost}
\end{equation}
The first term of $f$ is the stationary mean workload in work units.

\begin{proposition}[Model-level verification for blocked Top-$K$ SAA]
\label{prop:mm1-complete-verification}
Under the model and rate bounds above, fix
\[
0<\theta<1-\overline\lambda/\underline\mu,
\qquad \mathcal V(w):=e^{\theta w},
\]
and suppose
$\sup_N\mathbb E\mathcal V(X_0^{(N)})\le v_0<\infty$.
Empty initialization is one admissible choice. There are common
constants $\rho<1$, $b<\infty$, $C_Y<\infty$, and $C_U<\infty$ such
that, conditional on every inherited workload $w$ and selected action,
\begin{align}
\mathbb E_w^a\mathcal V(W(T))&\le\rho\mathcal V(w)+b,
\label{eq:mm1-cycle-drift}\\
G-f(a)T&=g_a(w)-g_a(W(T))+U,
\quad g_a(w):=\frac{w^2}{2\{\mu-\lambda(p)\}},
\label{eq:mm1-cycle-poisson}\\
\mathbb E_w^a U&=0,\qquad
\mathbb E_w^a\{G-f(a)T\}^2\le C_Y(1+w^2),\qquad
\mathbb E_w^a U^2\le C_U(1+w^2).
\label{eq:mm1-cycle-second-moments}
\end{align}
In particular, (S3)--(S5) hold with $q=2$ along the actual adaptive
trajectory. For the fixed-depth blocked algorithm in
Section~\ref{subsec:practical-topk}, completed deterministic cumulative
targets verify (S1)--(S2), and the external-arrival cycle construction
supplies the duration bound below. If $m_{\ell,N}\to\infty$ for every
fixed stage, then (S6) holds and $e_{\ell,N}\to0$ in probability. This conclusion concerns estimation;
optimizer identification additionally requires the hierarchy conditions
of Theorem~\ref{thm:topk-guardian-retention}.
\end{proposition}

\begin{proof}
Write $\lambda=\lambda(p)$, $d_a=\mu-\lambda$, and
$\delta=\underline\mu-\overline\lambda>0$.
We verify the conditions on the same workload trajectory.

\emph{Common drift.}
The next-arrival kernel $P_a$ has
$W'=(w-\mu A)^++\zeta$, where $A\sim\operatorname{Exp}(\lambda)$
and $\zeta\sim\operatorname{Exp}(1)$ are independent. Direct integration
therefore gives
\begin{equation}
P_a\mathcal V(w)
=\frac{\lambda e^{\theta w}+\theta\mu e^{-\lambda w/\mu}}
{(1-\theta)(\lambda+\theta\mu)}
\le r_\theta\mathcal V(w)+b_\theta,
\label{eq:mm1-arrival-drift}
\end{equation}
where
\[
r_\theta:=\frac{\overline\lambda}
{(1-\theta)(\overline\lambda+\theta\underline\mu)}<1,
\qquad b_\theta:=\frac1{1-\theta}.
\]
This also holds at $w=0$. Iterating for the $D$ arrivals gives
\eqref{eq:mm1-cycle-drift} with
\[
\rho=r_\theta^D,
\qquad b=b_\theta\frac{1-r_\theta^D}{1-r_\theta}.
\]
Since the current action is predictable and these constants are common,
\[
\mathbb E[\mathcal V(X_t^{(N)})\mid\mathcal H_{t-1}^{(N)}]
\le\rho\mathcal V(X_{t-1}^{(N)})+b,
\qquad
\sup_{N,t}\mathbb E\mathcal V(X_t^{(N)})
\le\bar v:=\max\{v_0,b/(1-\rho)\}.
\]
This verifies (S3) under adaptive switching, rather than only under a
frozen action.

\emph{Poisson representation and moment control.}
For the workload generator,
\[
\mathcal L_a g_a(w)
=-\mu\mathbf1\{w>0\}g_a'(w)
+\lambda\mathbb E[g_a(w+\zeta)-g_a(w)]
=-w+\frac{\lambda}{d_a}.
\]
The identity also holds at zero, since $g_a'(0)=0$.
Let $\Pi(ds,dz)$ be the marked arrival measure during this frozen-action
cycle, with compensator $\lambda\,ds\,e^{-z}dz$, and let
$\widetilde\Pi$ be its compensated version. The pathwise change-of-variable
formula, together with the compensated arrival revenue, gives
\begin{equation}
U=\int_{(0,T]}\int_0^\infty
\left\{\frac{W(s^-)z+z^2/2}{d_a}-p\right\}
\widetilde\Pi(ds,dz)
\label{eq:mm1-corrected-innovation}
\end{equation}
in \eqref{eq:mm1-cycle-poisson}. The term $-p$ accounts for the
fluctuation of $D-\lambda T$; it cannot be dropped merely because the
cycle contains exactly $D$ arrivals.

For completeness, the stopping and integrability in this calculation
can be checked directly. Put $H_D=\sum_{i=1}^{D}\zeta_i$. Up to $T$,
\[
0\le W(s)\le w+H_D,\qquad
T\sim\operatorname{Gamma}(D,\lambda),\qquad H_D\perp T.
\]
The stopping factor in the marked integral is
$\mathbf1\{N(s^-)<D\}$, where $N$ counts arrivals from the cycle origin;
it is predictable and includes the $D$th jump. For
$h_a(w,z)=(wz+z^2/2)/d_a-p$, exponential work moments give
\[
\mathbb E h_a(w,\zeta)^2
=\frac{2w^2+6w+6}{d_a^2}
-\frac{2p(w+1)}{d_a}+p^2
\le C_h(1+w^2)
\]
with a common finite $C_h$. Consequently,
\begin{align*}
\lambda\mathbb E_w^a\int_0^T
\mathbb E[h_a(W(s^-),\zeta)^2\mid W(s^-)]\,ds
&\le\overline\lambda C_h\,
\mathbb E[T\{1+(w+H_D)^2\}]\\
&\le C_U(1+w^2).
\end{align*}
Thus the stopped compensated integral is square integrable and centered;
its isometry proves the bound for $U$ and justifies the stopped identity.

With $Y=G-f(a)T$ and
$C_0=\overline\lambda/\delta+P\overline\lambda$, we also have
\[
|Y|\le T(w+H_D+C_0)+PD.
\]
The independence above, $\mathbb E H_D^2=D(D+1)$, and
$\mathbb E T^2\le D(D+1)/\underline\lambda^2$ imply
$\mathbb E_w^aY^2\le C_Y(1+w^2)$. For example, one may take
\[
C_Y=\frac{6D(D+1)}{\underline\lambda^2}
\{1+D(D+1)+C_0^2\}+2P^2D^2.
\]
Polynomial growth is dominated by $\mathcal V$: for every fixed $q\ge1$,
\[
|g_a(w)|^q\le\frac{w^{2q}}{(2\delta)^q}
\le\frac{1}{(2\delta)^q}
\left(\frac{2q}{e\theta}\right)^{2q}\mathcal V(w),
\qquad
1+w^2\le\left\{1+\left(\frac2{e\theta}\right)^2\right\}\mathcal V(w).
\]
Taking $q=2$ verifies (S4), and the bound for $U$ verifies (S5).
The conditional identities hold given the actual learning history because
its inherited workload and chosen action determine the next cycle law.

\emph{Completed stages and durations.}
Write $M=|\mathcal U|$. Under the blocked schedule, by stage $\ell$ each
action has at most $m_{\ell,N}$ cycles and at most $\ell+1$ blocks.
Thus (S1)--(S2) hold with $r_{\ell,N}=Mm_{\ell,N}$,
$j_{\ell,N}\le\ell+1$, and a deterministic boundary cover of size
$\Lambda_{\ell,N}\le\Lambda_\ell<\infty$, as established in
Remark~\ref{rem:triangular-block-simplification}.
With $c=D/(2\overline\lambda)$, Remark~\ref{rem:poisson-duration} gives
\[
\eta_{\ell,N}(c)
\le M\exp\{-D(\log2-1/2)m_{\ell,N}\}.
\]
Each prescribed stage completes almost surely because it contains at
most $Mm_{\ell,N}$ cycles, each with an almost surely finite Gamma
duration.
Substituting these verified bounds and $q=2$ into
Lemma~\ref{lem:stagewise-blocked-saa} yields
\begin{align}
\mathbb P(e_{\ell,N}\ge\varepsilon)
&\le M e^{-D(\log2-1/2)m_{\ell,N}}
+\frac{4\kappa_U\bar v M}{\varepsilon^2c^2m_{\ell,N}}
\notag\\
&\quad+M\Lambda_\ell\kappa_g\bar v
\left(\frac{4(\ell+1)}{\varepsilon c m_{\ell,N}}\right)^2,
\label{eq:mm1-verified-stage-bound}
\end{align}
which tends to zero when $m_{\ell,N}\to\infty$.
\end{proof}

The same calculation also controls the original innovation
$M_t=Y_t-\mathbb E[Y_t\mid\mathcal H_{t-1}]$:
\[
\mathbb E[M_t^2\mid\mathcal H_{t-1}]
\le C_Y(1+X_{t-1}^2).
\]
This is distinct from the corrected innovation $U_t$. The weighted
bound alone does not establish (A3): the almost-sure results additionally
require their sampling, boundary-growth, and martingale-summability
hypotheses. The finite-budget Top-$K$ conclusion above uses the
increasing-budget stagewise limit. A hard clock-time cutoff, imperfect workload observations,
or a different service-switching convention requires the corresponding
additional analysis. The common drift constants can be conservative when
$\underline\mu-\overline\lambda$ is small and are not practical budget
prescriptions.

The functions $f$, $g_a$, and $\mathcal V$, and the moment constants,
are used for verification; they are not inputs required by the learner.

\section{Numerical Experiments}
\label{sec:numerical}

The numerical study is designed to characterize both the advantages and the
limits of direct Top-$K$ rather than to establish universal dominance. We first
validate the event-driven implementation against queueing models with known
steady-state formulas. We then ask four substantive questions. How does direct
Top-$K$ compare with LiQUAR in terminal identification and cumulative learning
cost? Does Top-$K$ help when controls are mixed discrete--continuous? Does the
basic behavior persist in a non-Markovian heavy-tailed service model? Finally,
how much can a specialized gradient learner gain in a smooth regime where
reliable queueing sensitivity information is available?

Each benchmark has a distinct evidentiary role. The first separates terminal
identification from cumulative learning cost in a continuous control problem;
the second tests mixed discrete--continuous decisions; the third probes
robustness beyond the light-tailed model verified analytically; and the fourth
tests the method against a specialized gradient learner in a regime favorable
to that learner. The comparisons are therefore interpreted benchmark by
benchmark, with their information inputs, resource matching, and primary
performance criteria stated explicitly rather than pooled into a single
algorithm ranking.

All four experiments use the same logistic demand function
\[
\lambda(p)=10\frac{e^{4.1-p}}{1+e^{4.1-p}}.
\]
Under a fixed price $p$, external arrivals are Poisson with rate
$\lambda(p)$. The feasible action domains, service mechanisms, staffing
costs, and learning-cycle conventions are specified for each benchmark
below.

\paragraph{Interpretation of the retention schedules and audits.}
The comparisons below evaluate the specified fixed retention schedules. They
do not evaluate a method that learns a sufficient $K$ without access to
reference-objective information. Exact or high-precision reference values are
used to assess refinement geometry and terminal performance. For the
four-level $M/M/1$ benchmark in Section~\ref{subsec:liquar-comparison}, the
reported margins have been independently reproduced by an exhaustive audit
of \eqref{eq:guardian-active-family} using the implementation's child maps.
The five-level comparison and the explicitly documented shared staffing
hierarchy also admit full-family audits with exact-formula objective
bounds. Their model-level sampling and moment conditions require separate
verification. The lognormal family is exhaustively checked against saved
reference estimates. A simultaneous Monte Carlo-error sensitivity analysis is
reported below, but missing transient-bias and historical run metadata still
prevent treating that check as a certificate for the true objective. These audits assess fixed retention
schedules and are not inputs to the online SAA rankings.

\paragraph{Computational record.}
The accompanying computational-provenance packages index every table and
figure below. It separates \emph{full} records (source, inputs,
randomization, raw output, and aggregation are available), \emph{partial}
records (saved output is traceable but some historical metadata are missing),
and \emph{archival} records (only a historical summary or figure was
recovered). Missing metadata are not reconstructed silently. In particular,
the historical light-tailed staffing summary has been replaced by a newly
specified canonical experiment with frozen configuration, independently seeded
randomization and simulation streams, replication-level raw output, automated
validation, and exact replay checks. This is a replacement experiment, not a
reproduction of the unrecovered historical run. The final GOLiQ comparison is
likewise replaced by a corrected run with independently seeded simulation
streams. Section~\ref{subsec:computational-record}
summarizes the record and the benchmark-specific paragraphs give the detailed
design information.

\subsection{Implementation validation}
\label{subsec:validation}

We validated the simulators independently of the learning algorithms. For the
single-server benchmark, long fixed-action runs were compared with the exact
arrival rate, mean workload, and cost rate. Each fixed-$D$ cycle was checked
to contain exactly $D$ designated arrivals, workload remained nonnegative, and
the workload integral was evaluated analytically between event epochs.
For fixed-price arrival cycles, the mean and variance of the duration are
$D/\lambda(p)$ and $D/\lambda(p)^2$. At price $p=3.78551$ with $D=50$,
the theoretical mean cycle duration is
$8.65081$ and the simulated value is $8.66238$; the corresponding variances are
$1.49673$ and $1.52048$. Table~\ref{tab:validation} reports representative
cost-rate checks. We also switched the parallel-server simulator back to
exponential service and verified it against exact Erlang-C calculations before
using lognormal service.
For each fixed-$D$ $M/M/1$ row, we used 20 independent replications of
2,000 cycles, with $D=50$, an empty start, and no discarded warm-up; the
reported Monte Carlo standard error is computed across replication-level
cost-rate estimates. The four displayed actions use seed ranges
$10000$--$10019$, $11000$--$11019$, $14000$--$14019$, and
$15000$--$15019$, respectively. The separate duration check uses seed
$7654321$ on one persistent 5,000-cycle path. The $M/M/5$ row is an archived
validation result: its original horizon, replication count, and seed rule
were not recovered.

\begin{table}[ht]
\centering
\small
\caption{Representative fixed-action simulator checks. The reported standard
error refers to the simulated cost rate.}
\label{tab:validation}
\begin{tabular}{llrrr}
\toprule
Model & Action & Exact cost & Simulated cost & MC SE \\
\midrule
$M/M/1$ & $(8.18393,3.78551)$ & -11.2915 & -11.2774 & 0.0131 \\
$M/M/1$ & $(7,4.1)$             & -11.0000 & -11.0183 & 0.0116 \\
$M/M/1$ & $(6.5,7)$             &   2.9365 &   2.9367 & 0.0023 \\
$M/M/1$ & $(6.5,3.7)$           &  -3.9839 &  -4.0333 & 0.2014 \\
$M/M/5$ & $(s,p)=(5,3.75)$      & -12.9702 & -12.9757 & 0.0087 \\
\bottomrule
\end{tabular}
\end{table}

The optimization layer was checked separately by replacing all stochastic SAA
values with exact or high-precision reference objective values. This
reproduces the reported master-grid optima and guardian-retention paths, so the
queue simulator and the hierarchical search can be audited independently.

For the three benchmarks with analytic steady-state objectives,
Table~\ref{tab:analytic-benchmarks} summarizes the full-domain optimum
$f^*$ and the minimum $f_L^*$ on the stated master grid. The optima were
recomputed from the objective formulas, using Erlang C for the staffing
model, and by enumeration of each master grid. The full-domain staffing
optimum retains integer $s\in\{4,\ldots,8\}$ and optimizes only price
continuously. Displayed values are rounded numerical evaluations of
the analytic objectives.

\begin{table}[ht]
\centering
\footnotesize
\setlength{\tabcolsep}{4pt}
\caption{Analytic benchmark optima and master-grid discretization gaps.
Coordinates are capacity and price for the single-server models, and
integer staffing and price for the multiserver model.}
\label{tab:analytic-benchmarks}
\begin{tabular}{@{}lccrr@{}}
\toprule
Benchmark & \shortstack{Full-domain\\optimizer}
& \shortstack{Master-grid\\optimizer} & $f^*$ & $f_L^*-f^*$ \\
\midrule
\shortstack[l]{Linear-cost\\$M/M/1$}
& $(8.183928,3.785511)$ & $(8.46875,3.71875)$
& $-11.29146766$ & $0.01565145$ \\
\addlinespace
\shortstack[l]{Integer-staffing\\$M/M/s$}
& $(5,3.734660)$ & $(5,3.71875)$
& $-12.97142146$ & $0.00128977$ \\
\addlinespace
\shortstack[l]{Quadratic-cost\\$M/M/1$}
& $(7.103113,4.023373)$ & $(7.03125,4.05625)$
& $-13.12609570$ & $0.00361515$ \\
\bottomrule
\end{tabular}
\end{table}

For a Top-$K$ recommendation on the master grid, the reported terminal
gap separates as
\[
f(\widehat a)-f^*
=\bigl[f(\widehat a)-f_L^*\bigr]+\bigl[f_L^*-f^*\bigr].
\]
The first term is the learning error relative to the grid; the second is
the deterministic discretization gap, which persists even after correct
master-grid identification. This decomposition does not impose the same
grid restriction on a continuously valued comparator.

\paragraph{Costs and information available to the learners.}
Table~\ref{tab:learner-observations} distinguishes the statistics used
to update each learner from the simulator's full state and evaluation
records. In the workload benchmark, $W(u)$ denotes unfinished work in
mean-one work units, rather than time to clear the system. In the
customer-count benchmarks, $Q(u)$ includes customers in service.
At a fixed stable action in the exponential single-server model,
\[
\mathbb E[W_\infty]=\mathbb E[Q_\infty]
=\frac{\lambda(p)}{\mu-\lambda(p)}.
\]
This equality of stationary means does not identify their realized
integrals or make their observation requirements interchangeable.

\begin{table}[ht]
\centering
\small
\setlength{\tabcolsep}{4pt}
\caption{Observation inputs used for learning in the reported comparisons.
Applied controls and elapsed time are known in every row.}
\label{tab:learner-observations}
\begin{tabular}{@{}>{\raggedright\arraybackslash}p{0.21\textwidth}
>{\raggedright\arraybackslash}p{0.48\textwidth}
>{\raggedright\arraybackslash}p{0.25\textwidth}@{}}
\toprule
Method & Observations used & Update statistic \\
\midrule
Top-$K$, workload benchmark
& Exact workload integral and $D$ observed arrivals per cycle.
& Cumulative realized cost divided by cumulative time. \\
\addlinespace
LiQUAR-oracleW
& Exact workload integral on the central time window; arrival count over the whole period.
& Two-point finite differences of trimmed cost-rate estimates. \\
\addlinespace
LiQUAR, delayed mode
& Workload statistic recoverable by period end, on the same central window; whole-period arrivals.
& Same finite-difference rule, with censored workload data. \\
\addlinespace
Staffing SAA methods
& Time integrals of total customer count and charged staffing $\max\{s,B(u)\}$; arrivals.
& Cumulative realized cost divided by cumulative time. \\
\addlinespace
Top-$K$, five-level comparison
& Customer-count integral and actual arrivals during $D$ service entries.
& Cumulative realized cost divided by cumulative time. \\
\addlinespace
GOLiQ
& Waiting times observed at service entry and busy-period ages recorded at arrival; known $\lambda(p)$, $\lambda'(p)$, and $c'(\mu)$.
& Queueing-sensitivity gradient from retained customer observations. \\
\bottomrule
\end{tabular}
\end{table}

Direct SAA uses the resulting cycle costs and durations; it does not
require latent service requirements or the analytic demand function
as ranking inputs. Nevertheless, supplying the exact workload integral
is an observation assumption: the workload experiment does not obtain
it from completion records alone. By contrast, the staffing reward
can be accumulated from customer counts, busy-server counts, arrivals,
and time without knowing unfinished service requirements.
The staffing row describes the documented shared model implemented by the
canonical light-tailed driver and the recovered lognormal implementation.
Evaluation formulas for $f$ and $f^*$ are not learning inputs for the
direct-SAA methods.

\paragraph{Learning-cost metrics and budget accounting.}
Let $H$ denote the completed observation horizon of one replication,
$C(H)$ its realized cumulative cost, $a(u)$ the action label in force at time
$u$, and $f^*$ the continuous-domain optimum for the benchmark. We distinguish
\begin{equation}
R^{\mathrm{real}}(H)=C(H)-f^*H,
\qquad
R^{\mathrm{ps}}(H)=\int_0^H\{f(a(u))-f^*\}\,du.
\label{eq:numerical-regret-definitions}
\end{equation}
The first is \emph{realized regret}: it includes inherited-state transients and
sample noise and can be negative on an individual path. The second is
\emph{stationary-policy pseudo-regret}: it replaces realized cost on each
operating interval by the stationary objective of its action label and excludes
those transient and sampling effects. The corresponding rates divide by $H$.
Neither quantity is the terminal gap $f(\widehat a)-f^*$, which evaluates only
the final recommendation. These evaluation quantities use $f$ and $f^*$
offline; they are not learner inputs.

Every completed Top-$K$ primitive cycle is charged in $C(H)$ and $H$, and the
reported runs finish their deterministic stage targets rather than stopping at
a hard clock cutoff. LiQUAR completes both full clock periods in every
iteration. Its beginning and ending portions omitted from the workload update
remain fully charged in realized cost and time. Thus Table~\ref{tab:liquar}
matches mean clock time only approximately: LiQUAR time is deterministic for a
given iteration schedule, while Top-$K$ time is random. For the GOLiQ
comparison, service entries are matched exactly within pairs; if the matched
event occurs inside a GOLiQ learning cycle, cost and time through that event are
charged but the incomplete cycle produces no gradient update. The three
LiQUAR sensitivity settings are reported separately; no best setting is
selected and no outer-loop selection cost is represented.

\subsection{LiQUAR benchmark: terminal identification versus learning cost}
\label{subsec:liquar-comparison}

This benchmark asks whether a direct ranking hierarchy can improve terminal
policy identification without obscuring the cumulative operating cost paid
during learning. We first consider the $M/M/1$ logistic benchmark of Chen,
Hong, and Liu
\cite{ChenHongLiu2026}. Here the staffing cost is linear, and
\[
f(\mu,p)=\frac{\lambda(p)}{\mu-\lambda(p)}+\mu-p\lambda(p),
\qquad
(\mu,p)\in\mathcal B:=[6.5,10]\times[3.5,7].
\]
On this rectangle,
\[
0.52153<\lambda(p)<6.45657<6.5\le\mu,
\]
so every candidate action satisfies the rate bounds of
Proposition~\ref{prop:mm1-complete-verification}. Under that proposition's
normalized-work, switching, observation, and initialization conventions,
its model-level estimation guarantee applies to the blocked Top-$K$
procedure. Guardian separation is verified below for its stated hierarchy. This
specialization concerns the stated workload model, not a verification
of the LiQUAR perturbation scheme.

For clarity, with $D=50$ external arrivals per cycle, the Top-$K$ reward is
\[
G_t=\int_{\text{cycle }t}W(u)\,du+\mu_tT_t-p_tD,
\]
where $W$ is normalized unfinished work with mean-one work requirements.
A cycle ends after the $D$th external arrival, including its work and
revenue. Inherited work is preserved and drains at the newly selected
rate $\mu_t$; subsequent arrivals use the newly selected price $p_t$.
The simulated workload integral is observed exactly here.

The exact continuous optimum is
\[
(\mu^*,p^*)\approx(8.18393,3.78551),
\qquad f^*\approx-11.29147.
\]
For Top-$K$, the hierarchy
\[
3\times3\to5\times5\to9\times9\to17\times17
\]
has master-grid optimum $(8.46875,3.71875)$ and deterministic grid gap
$0.01565$. We use $K=(1,2,4)$.
The archived specifications give cumulative per-action stage targets
$(30,80,110,135)$, $(94,251,345,423)$, and
$(236,630,867,1064)$ for the low, medium, and large budgets.
The saved Top-$K$ raw rows and summaries are available, but their original
experiment wrapper and row-level seed formula were not recovered; those rows
therefore have partial rather than full computational provenance. The
corrected LiQUAR side has complete source, seed lists, raw output, and probe
traces in the N3 package.

\paragraph{Child maps and full-family guardian audit.}
The archived four-level policy specifies equally spaced coordinate grids
with $n_\ell=2^{\ell+1}+1$ points and parent mesh
$h_\ell=3.5/(n_\ell-1)$. Its child rule is
\begin{equation}
\operatorname{Ch}_\ell(\mu,p)
=\left\{(\mu',p')\in\mathcal U_{\ell+1}:
|\mu'-\mu|\le h_\ell/2,\quad |p'-p|\le h_\ell/2\right\}.
\label{eq:mm1-child-map}
\end{equation}
The implementation adds a $10^{-12}$ comparison tolerance. All grid
coordinates and half-mesh boundaries here are binary rationals, and an
excluded coordinate exceeds its boundary by at least $0.21875$, so this
tolerance leaves the displayed child sets unchanged. Midpoint children may
have more than one parent; unions remove duplicate actions. Rankings break
ties lexicographically in $(\mu,p)$. Each level covers the next one, and
all 289 master-grid actions are reachable from the initial grid.

Using these maps, we identify the unique master-grid optimizer, compute
guardians backwards, and enumerate every guardian-preserving active set at
the pruning levels. Table~\ref{tab:mm1-guardian-audit} gives the results.
There are four guardians in each full candidate grid at these levels;
the active sets need not contain all four. Every enumerated case is
nonautomatic. Objective values and margin signs were checked with rational
interval bounds on the exponential in the logistic formula, including a
strict interval separation of the master-grid optimum. The displayed lower
bounds are conservative truncations of the certified bounds.

\begin{table}[ht]
\centering
\small
\caption{Full-family guardian audit for the four-level $M/M/1$ hierarchy.
The family is defined in \eqref{eq:guardian-active-family}.}
\label{tab:mm1-guardian-audit}
\begin{tabular}{crrrr}
\toprule
$\ell$ & $|\mathcal U_\ell|$ & $K_\ell$ & $|\mathfrak A_\ell^+|$
& Certified lower bound on $\underline\Delta_\ell$ \\
\midrule
0 & 9  & 1 & 1  & 5.23743432 \\
1 & 25 & 2 & 4  & 2.86034218 \\
2 & 81 & 4 & 32 & 0.05967732 \\
\bottomrule
\end{tabular}
\end{table}

The recomputed minima, approximately $5.23743433$, $2.86034218$, and
$0.05967732$, reproduce the saved audit values. They are full-family
minima, not margins from a single exact-objective path. In particular,
the exact Top-$K$ path has larger margins at levels 1 and 2,
approximately $2.92693723$ and $0.31096158$.
Together with Proposition~\ref{prop:mm1-complete-verification}, this verifies
the hypotheses of Theorem~\ref{thm:topk-guardian-retention} for the specified
workload implementation and completed stage targets. If every
$m_{\ell,N}\to\infty$, Corollary~\ref{cor:topk-master-grid-consistency}
therefore gives master-grid optimizer identification in probability.
The moment constants provide sufficient bounds, not calibrated predictions
of the finite budgets in the tables.

\paragraph{Structural evidence: exact Top-1 failure.}
The same hierarchy gives a concrete deterministic Top-1 failure.
Exact-objective Top-1 retains $(10,3.5)$ at level 0 and $(9.125,3.5)$ at
level 1. At level 2 its best action is again $(9.125,3.5)$, now a
nonguardian; the best guardian is $(8.6875,3.5)$, with
$\Delta_{2,1}\simeq-0.03699891$. Thus Top-1 loses the master-grid optimum
even with exact rankings. This geometric comparison does not imply that
Top-$K$ has smaller error at every equal finite operating-time budget.

LiQUAR operates two perturbed controls around its current center and forms a
finite-difference gradient. We use the published base schedules
\[
\eta_k=4k^{-1},\qquad
\delta_k=\min\{0.1,0.5k^{-1/3}\},\qquad
T_k=200k^{1/3},\qquad
\alpha=0.1,
\]
with nominal initial center $(10,5)$, projected as specified below.
Our primary comparator, labeled
\emph{LiQUAR-oracleW}, receives exact workload observations, as does
Top-$K$ in this benchmark. It still trims its holding-cost observations
to $[\alpha T_k,(1-\alpha)T_k]$. In the recovered implementation, its
cost-rate estimate for one period under $(\mu,p)$ is
\[
\widehat f^{\,G}
=-\frac{pJ}{T_k}
+\frac{1}{(1-2\alpha)T_k}
  \int_{\alpha T_k}^{(1-\alpha)T_k}W(u)\,du+\mu,
\]
where $J$ counts all arrivals during that period. Thus revenue uses
observed arrivals, not the analytic demand rate. Full-period costs are
recorded separately for performance evaluation. Exact workload removes
recoverability censoring but does not equate the methods' time windows,
cycle lengths, or update rules.

The additional delayed mode replaces $W(u)$ in this estimate by
$W(u)\mathbf1\{W(u)\le\mu(T_k-u)\}$, as in
\cite[Section~4.2]{ChenHongLiu2026}. The published FCFS observation model
reveals arrivals immediately and individual work on completion.
Our simulator evaluates the equivalent recoverability statistic from
its internal state; it does not implement a separate reconstruction
routine using customer records. Both modes retain the same warm-up
and overtime exclusions.
The corrected comparisons below use these observation modes, with
raw outputs, seed lists, and evaluation-control traces retained in
the accompanying reproducibility package.

\paragraph{Evaluation controls, feasibility, and the corrected projection.}
The recovered implementation uses a uniform coordinate choice
$i_k\in\{1,2\}$, $Z_k=2e_{i_k}$, and
\[
x_k^- =\bar x_k-\delta_k e_{i_k},\qquad
x_k^+ =\bar x_k+\delta_k e_{i_k},\qquad
H_k=\frac{2e_{i_k}}{\delta_k}
       \{\widehat f^{\,G}(x_k^+)-\widehat f^{\,G}(x_k^-)\}.
\]
It projects only the next center onto $\mathcal B$. This follows the
displayed update in \cite[Algorithm~1]{ChenHongLiu2026}, but does not
ensure that the two executed controls belong to $\mathcal B$.
For example, a capacity perturbation at the initial center $(10,5)$
uses $(10.1,5)$, outside the stated capacity range. At center
$(6.5,3.5)$, the possible lower probes have stability slacks
\[
6.4-\lambda(3.5)\simeq-0.0565631,\qquad
6.5-\lambda(3.4)\simeq-0.1818777.
\]
Thus stability of the center rectangle does not guarantee stability
of every fixed-action evaluation. A finite operating period can still
be simulated at such a control, but its stationary workload mean is
not finite; the resulting adaptive path is not thereby proved to
be unstable. The original numerical results must be distinguished
from results under a projection that enforces the stated domain.

For the corrected implementation, let $r=0.1$ and set
\[
\begin{gathered}
\mathcal B_r=[6.5+r,10-r]\times[3.5+r,7-r],\\
\bar x_1=\Pi_{\mathcal B_r}(10,5)=(9.9,5),\quad
\bar x_{k+1}=\Pi_{\mathcal B_r}(\bar x_k-\eta_kH_k).
\end{gathered}
\]
All other update and observation conventions are retained. The
archived tuning schedules are, more precisely,
\[
\eta_k(c)=4c/k,\qquad
\delta_k(c)=\min\{0.1,\,0.5c k^{-1/3}\},\qquad
c\in\{0.6,1,1.2\};
\]
the cap is not multiplied by $c$. Hence $\delta_k(c)\le r$ in every
case. Coordinatewise inclusion gives $x_k^\pm\in\mathcal B$ and
\[
\mu_k^\pm-\lambda(p_k^\pm)
\ge6.5-\lambda(3.5)
\simeq0.04343694>0.
\]
This construction preserves both the symmetry and the spacing of
the finite differences. It needs only the supplied domain and
perturbation cap; the demand function is used to verify the bound,
not to select controls. The benchmark optimizer
$(8.18393,3.78551)$ lies in $\mathcal B_r$, so the optimal value $f^*$
is unchanged. A fixed inset need not preserve the optimum of a
different problem, especially one with a boundary minimizer.
This is a feasibility correction, not a proof of LiQUAR convergence
or an application of the fixed-arrival-count Top-$K$ theorem.

We retain the displayed normalization of $H_k$ to isolate the
projection change. With exact objective values, its expectation
over the coordinate choice is $2\nabla f(\bar x_k)+O(\delta_k^2)$;
thus the unprojected update has twice the leading gradient term of
a convention normalized to $\nabla f$. No silent rescaling of the historical update is made.

The domain audit exactly reproduced all 190 saved baseline outputs
from an earlier independent LiQUAR recheck. Those trajectories
contained 7,874 out-of-domain evaluations, including 1,191
frozen-action-unstable evaluations. We therefore reran the comparison
with the corrected projection. Across 470 corrected replications
(190 oracle-workload, 190 delayed-workload, and 90 tuning runs), all
183,600 evaluations lie in $\mathcal B$; their smallest recorded
stability slack is $0.14344$. These checks supplement the pathwise
domain proof above.

Because Top-$K$ uses fixed-arrival-count cycles and LiQUAR uses fixed-clock
periods, we compare similar mean cumulative operating times rather than
iteration counts. Table~\ref{tab:liquar} combines the unchanged Top-$K$
samples with new LiQUAR-oracleW runs using the corrected projection.
LiQUAR uses 42, 100, and 200 iterations, respectively. Its mean times
differ from Top-$K$'s by at most about $2\%$; this is not a comparison
at identical pathwise stopping times. The reported Top-$K$ terminal
gaps are smaller at all three budgets. The original main-table
LiQUAR values and the independently reproduced rechecks are preserved
separately in the revision record.

\begin{table}[ht]
\centering
\small
\caption{Comparable operating-time budgets: corrected LiQUAR-oracleW
with inset-center projection. Top-$K$ samples are unchanged. The gap
is $f(\widehat x)-f(x^*)$.}
\label{tab:liquar}
\begin{tabular}{llrrrrr}
\toprule
Budget & Method & Rep. & Mean time & Mean gap & Median gap & $P\{\text{gap}\le0.25\}$ \\
\midrule
Low & Top-$K$ CTF       & 100 & 44,658  & 0.130 & 0.071 & 0.85 \\
    & LiQUAR-oracleW    & 100 & 44,383  & 5.063 & 1.511 & 0.13 \\
Medium & Top-$K$ CTF    & 60  & 137,661 & 0.084 & 0.056 & 0.93 \\
       & LiQUAR-oracleW & 60  & 140,066 & 1.425 & 0.358 & 0.38 \\
Large & Top-$K$ CTF     & 30  & 345,226 & 0.042 & 0.018 & 0.97 \\
      & LiQUAR-oracleW  & 30  & 351,941 & 0.154 & 0.111 & 0.83 \\
\bottomrule
\end{tabular}
\end{table}

With LiQUAR's delayed-workload statistic and the same corrected
projection, the median terminal gaps are approximately $0.861$,
$0.379$, and $0.153$ (100, 60, and 30 replications), above the
corresponding Top-$K$ values. Top-$K$ retains exact workload observations
in this comparison. These results are a sensitivity check for LiQUAR,
not a comparison in which both learners face delayed observations.
For a delayed-observation Top-$K$ procedure, the ranking rule must use
only costs available at its decision times, and the accumulated error
condition in Remark~\ref{rem:stagewise-observation-error} must be
verified under that actual policy. Neither this implementation nor
that verification is supplied here.

At the large budget, we also reran the three previously specified
settings $c\in\{0.6,1,1.2\}$ with shorter periods $T_k=40k^{1/3}$,
using 30 replications per setting and the same corrected projection.
Each run has 660 iterations and total operating time about $345{,}107$.
For $c=0.6,1,1.2$, respectively, the median terminal gaps are
$0.354$, $0.207$, and $0.236$. Table~\ref{tab:liquar-regret} supplies
the corresponding realized learning costs, replication counts, clock-time
variability, and pointwise distribution-free confidence intervals for the
population median. The intervals use symmetric order statistics and have
coverage at least $95\%$ under independent continuous replications; no
multiple-comparison adjustment is made.

\begin{table}[ht]
\centering
\scriptsize
\setlength{\tabcolsep}{4pt}
\caption{Realized cumulative learning cost in the corrected $M/M/1$
comparison. Time is displayed as mean (standard deviation). Regret is
$R^{\mathrm{real}}(H)$ from \eqref{eq:numerical-regret-definitions}, in
millions, displayed as sample median [pointwise distribution-free confidence
interval with coverage at least $95\%$].}
\label{tab:liquar-regret}
\begin{tabular}{@{}llrrl@{}}
\toprule
Budget & Method / setting & Rep. & Time: mean (SD) & Median regret [CI] \\
\midrule
Low & Top-$K$ CTF & 100 & $44{,}658\ (5{,}856)$ & $0.176\ [0.175,0.178]$ \\
    & LiQUAR-oracleW, baseline & 100 & $44{,}383\ (0)$ & $0.250\ [0.231,0.282]$ \\
\addlinespace
Medium & Top-$K$ CTF & 60 & $137{,}661\ (17{,}577)$ & $0.559\ [0.550,0.568]$ \\
       & LiQUAR-oracleW, baseline & 60 & $140{,}066\ (0)$ & $0.481\ [0.400,0.576]$ \\
\addlinespace
Large & Top-$K$ CTF & 30 & $345{,}226\ (40{,}540)$ & $1.425\ [1.397,1.452]$ \\
      & LiQUAR-oracleW, baseline & 30 & $351{,}941\ (0)$ & $0.689\ [0.610,0.826]$ \\
      & LiQUAR-oracleW, $c=0.6$, $T$ coefficient $40$ & 30
      & $345{,}107\ (0)$ & $0.536\ [0.456,0.589]$ \\
      & LiQUAR-oracleW, $c=1$, $T$ coefficient $40$ & 30
      & $345{,}107\ (0)$ & $0.478\ [0.425,0.583]$ \\
      & LiQUAR-oracleW, $c=1.2$, $T$ coefficient $40$ & 30
      & $345{,}107\ (0)$ & $0.449\ [0.426,0.515]$ \\
\bottomrule
\end{tabular}
\end{table}

The low-budget samples favor Top-$K$ on median realized regret, whereas the
medium- and large-budget samples favor corrected LiQUAR on that criterion.
At the large budget, every reported LiQUAR setting has a larger median
terminal gap but a smaller median realized regret than Top-$K$. The
shorter-period settings also reduce median realized regret relative to the
corrected LiQUAR baseline while increasing its median terminal gap. These are
unpaired comparisons at similar, not identical, time horizons. They support a
finite-sample tradeoff between terminal identification and learning cost; they
do not establish universal dominance or identify an optimally tuned LiQUAR
policy. The three settings are reported separately, without treating their
best observed value as an independently validated tuning result.

\subsection{Mixed discrete--continuous staffing: an $M/M/s$ benchmark}
\label{subsec:mms-benchmark}

This benchmark isolates the setting in which direct ranking is especially
natural: one decision coordinate is integer-valued and has no ordinary
gradient. Unlike the preceding continuous-capacity comparison, ordinary
derivatives with respect to staffing are unavailable. We consider
\[
s\in\{4,5,6,7,8\},\qquad p\in[3.5,7],\qquad \nu=1.75,
\]
where $s$ is the number of identical servers, each server works at rate $\nu$,
and the arrival rate remains logistic. The steady-state objective is
\[
f(s,p)=\mathbb E[Q_\infty(s,p)]+s-p\lambda(p).
\]
Here $Q_\infty$ counts all customers in the system, including those in service.
The shared nonpreemptive staffing model uses cycles of $D=50$ external
arrivals and cycle cost
\begin{equation}
G_t=\int_{\text{cycle }t}\bigl[Q(u)+\max\{s_t,B(u)\}\bigr]\,du-p_tD,
\label{eq:staffing-cycle-cost}
\end{equation}
where $Q(u)$ is the total number of customers and $B(u)$ the number of busy
servers. A staffing increase immediately starts waiting jobs when possible;
a decrease lets existing services finish and charges legacy busy servers
until completion. These immediate service starts are not counted as
external arrivals. The physical state retains the active residual service
requirements and the FCFS waiting work; these latent requirements need not
be observed by the learner, which uses accumulated cost and time. Future
arrivals use the newly selected price. These conventions define the canonical
staffing model implemented by the common Top-$K$/Top-1/flat driver below.
For exponential service, $\mathbb E[Q_\infty]$ is computed exactly from
Erlang C for evaluation only. The continuous-price optimum is approximately
\[
(s^*,p^*)=(5,3.73466),\qquad f^*=-12.97142.
\]

The hierarchy keeps the five integer staffing values unchanged while refining
only the price grid:
\[
5\times3\to5\times5\to5\times9\to5\times17.
\]
For the hierarchy audit and the canonical driver, a parent retains its
staffing coordinate and generates the next-grid
prices within half the parent price mesh. Thus, with
$h_\ell=3.5/(2^{\ell+1})$ and next-level price grid $P_{\ell+1}$,
\begin{equation}
\operatorname{Ch}_\ell(s,p)
=\{(s,p'):p'\in P_{\ell+1},\ |p'-p|\le h_\ell/2\}.
\label{eq:staffing-child-map}
\end{equation}
Unions deduplicate shared children, and cost ties are broken by $(s,p)$.
This rule covers every next-level grid. The master-grid optimum is
$(5,3.71875)$, with discretization gap about $0.00129$.
For $K=(2,2,2)$, exhaustive enumeration gives active-set family sizes
$1,27,47$. All cases are nonautomatic. Rational interval evaluation of
Erlang C certifies the unique optimizer and uniform margin lower bounds
$0.86630756$, $0.90998220$, and $0.69968727$, reproducing the reported
margins $0.8663$, $0.9100$, and $0.6997$.
This verifies the explicitly stated hierarchy. The stochastic implementation
is checked separately through fixed-action Erlang-C comparisons, transition
invariants, raw-schema and accounting validators, seed-isolation checks, and
exact aggregate replay from the retained replication rows.
In this example exact-objective Top-1 also
reaches the master-grid optimizer, so the main role of Top-$K$ is statistical
protection against an early noisy ranking rather than correction of a
deterministic coarse-grid failure.

Table~\ref{tab:mms} reports a new canonical comparison of Top-$K$, Top-1 CTF,
and flat direct SAA using 200 independent replications per method and budget.
All methods use the same simulator and terminal master grid. The frozen
per-action stage allocations for Top-$K$ are $(15,50,150,500)$,
$(45,150,450,1500)$, and $(120,400,1200,4000)$ at the low, medium, and large
budgets; the corresponding Top-1 allocations are $(22,72,215,717)$,
$(65,217,650,2167)$, and $(166,554,1663,5543)$. Flat SAA uses respectively
12, 36, and 95 cycles per master-grid action. These targets were calibrated
with pilot-only seeds and then frozen before the disjoint main-run seeds were
generated. Within each CTF stage, candidate blocks are randomly ordered by a
stream independent of the queue-simulation stream. Exact Erlang-C objectives
are used only for terminal evaluation, never for online ranking.

\begin{table}[ht]
\centering
\small
\caption{Canonical mixed discrete--continuous $M/M/s$ comparison, 200
replications per method and budget. Parentheses give pointwise normal 95\%
confidence intervals for the mean terminal gap.}
\label{tab:mms}
\resizebox{\textwidth}{!}{%
\begin{tabular}{llrrrrr}
\toprule
Budget & Method & Mean time & Mean events & Mean gap (95\% CI) & Median gap & Selected \\
\midrule
Low & Top-$K$ CTF & 31,486 & 270,270 & 0.1055 (0.0820, 0.1290) & 0 & 0.690 \\
    & Top-1 CTF    & 30,281 & 223,544 & 0.2615 (0.2209, 0.3020) & 0.3533 & 0.455 \\
    & Flat SAA     & 32,167 & 101,997 & 0.5052 (0.4020, 0.6083) & 0.3022 & 0.235 \\
Medium & Top-$K$ CTF & 94,284 & 809,995 & 0.0339 (0.0182, 0.0496) & 0 & 0.915 \\
       & Top-1 CTF    & 93,663 & 701,026 & 0.1898 (0.1596, 0.2199) & 0 & 0.550 \\
       & Flat SAA     & 96,526 & 305,997 & 0.2419 (0.2052, 0.2786) & 0.1934 & 0.410 \\
Large & Top-$K$ CTF & 249,881 & 2,133,995 & 0.0040 ($-0.0015$, 0.0096) & 0 & 0.990 \\
      & Top-1 CTF    & 248,249 & 1,870,332 & 0.1393 (0.1123, 0.1663) & 0 & 0.660 \\
      & Flat SAA     & 254,775 & 807,497 & 0.1505 (0.1254, 0.1756) & 0 & 0.530 \\
\bottomrule
\end{tabular}
}
\end{table}

The mean operating times differ from Top-$K$ by $-3.83\%$, $-0.66\%$, and
$-0.65\%$ for Top-1 and by $+2.16\%$, $+2.38\%$, and $+1.96\%$ for flat SAA
from low to large budget. Thus the comparison is approximately matched on
realized operating time, not pathwise or event matched; event totals differ
because the methods visit different prices and because CTF and flat SAA allocate
cycles differently. All six independent comparator-minus-Top-$K$ mean-gap
contrasts are positive: the Top-1 contrasts are $0.1560$ (95\% CI
$[0.1091,0.2028]$), $0.1559$ ($[0.1219,0.1898]$), and $0.1352$
($[0.1077,0.1628]$), while the flat-SAA contrasts are $0.3996$
($[0.2939,0.5054]$), $0.2080$ ($[0.1681,0.2479]$), and $0.1464$
($[0.1207,0.1721]$). The next-best master-grid action has gap $0.1934$,
so $P\{\text{gap}\le0.10\}$ equals the optimizer-selection frequency shown
in the final column. These are finite-sample results for this frozen design,
not a claim of universal or equal-resource dominance.

\begin{figure}[ht]
\centering
\includegraphics[width=0.72\textwidth]{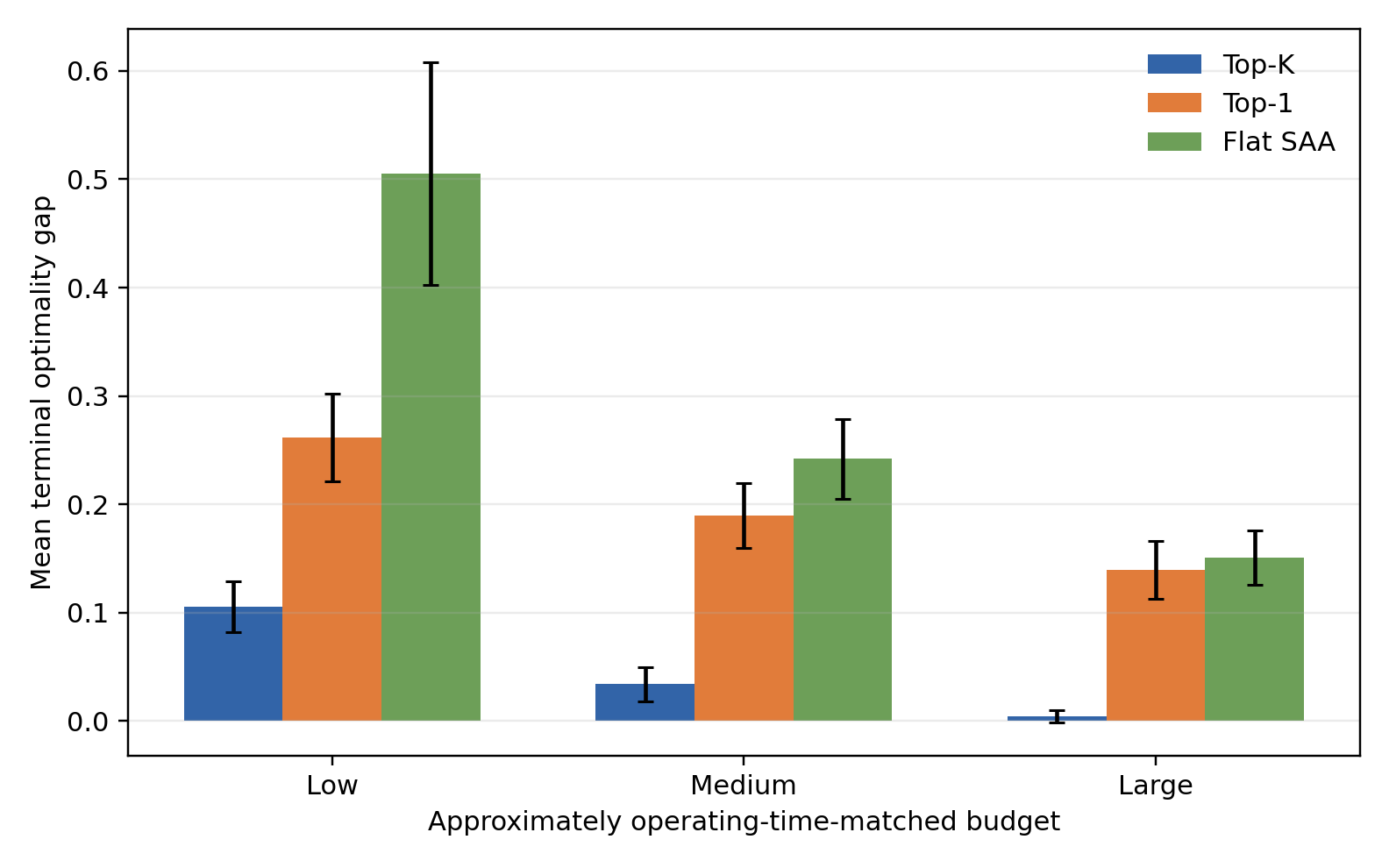}
\caption{Mean terminal optimality gap and pointwise normal 95\% confidence
interval in the canonical, approximately operating-time-matched $M/M/s$
comparison. Each point summarizes 200 independent replications.}
\label{fig:mms-gap}
\end{figure}

\FloatBarrier
\subsection{Heavy-tailed non-Markovian service: a nonpreemptive $M/G/s$ test}
\label{subsec:heavy-tail}

This experiment is a robustness test, not an additional model-verification
theorem. We next keep the same pricing and integer-staffing problem but replace
exponential service by lognormal service with
\[
\mathbb E[S]=1/1.75,
\qquad
\operatorname{SCV}(S)=4.
\]
The queue-length process alone is no longer Markovian; the complete persistent
state includes residual service requirements. Staffing changes are
nonpreemptive. When the target $s$ decreases, customers already in service
finish before the corresponding servers leave, and staffing cost is charged
for those legacy busy servers until completion.
Its fixed-arrival cycles and signed costs follow
\eqref{eq:staffing-cycle-cost}; the change is the service distribution.

Because lognormal service has no finite positive exponential moment, we do not
invoke Corollary~\ref{cor:common-lyapunov-regenerative} as an analytical
verification for this experiment: its common fixed-$D$ drift assumption is
geometric and need not hold in this heavy-tailed setting.  The experiment is
therefore presented as a numerical robustness test beyond that convenient
verification route.  The more general block-boundary conditions
\textnormal{(B2)}--\textnormal{(B3)} and the selected-cycle moment route
identify possible analytical avenues, but we do not claim here that those
conditions have been verified for this particular $M/G/s$ model.

There is no Erlang-C truth for this model. The archived reference table records
a mean, Monte Carlo standard error, and replication count for every point of
the $5\times17$ master grid, with additional replications for the leading
candidates. The fixed-action raw replications, seed rule, run length, and
burn-in were not recovered. The reference master-grid optimum is
\[
(s_L^*,p_L^*)=(5,3.9375),
\qquad f_L^*\approx-12.44433,
\]
while $(6,3.71875)$ is second-best with an estimated gap $0.08901$.
As a Monte Carlo-error sensitivity check, we form two-sided Student-$t$
intervals for all 85 archived means using Bonferroni allocation $0.05/85$.
The simultaneous interval for $f_L^*$ is
$[-12.46684,-12.42181]$, and its upper endpoint lies $0.05400$ below
the smallest lower endpoint among the other 84 actions. A separate pointwise
Welch interval for the runner-up gap is $[0.07623,0.10179]$.

The heavy-tailed model also changes the coarse-grid geometry. The recovered
nonpreemptive code uses \eqref{eq:staffing-child-map}. For $K=(3,3,2)$,
an exhaustive audit against its 85 saved reference means gives family sizes
$1,169,529$; all cases are nonautomatic. The uniform reference margins are
$0.54251352$, $0.86720211$, and $0.43318365$, reproducing the earlier
rounded values $0.543$, $0.867$, and $0.433$.
Propagating the simultaneous action intervals through the complete
guardian-preserving families gives conservative lower margins
$0.05753$, $0.73468$, and $0.36980$. Thus the reference ranking and all three
guardian margins remain positive after accounting for the reported Monte Carlo
standard errors under the usual replication-level $t$-interval assumptions.
The calculation does not control unknown transient bias or replace the missing
historical metadata, so it is not a true-objective guardian certificate and we
retain the term ``reference.'' Top-1 applied to these reference estimates
discards the reference-optimal refinement path and ends at the reference
second-best master-grid action. These statements concern estimated objective
values and are not included in the finite-sample theorem.

The first moderate-horizon runs were visibly noisier than in the light-tailed
benchmark. This is a descriptive comparison rather than an isolated service-tail
effect: the objective geometry and retention vector also change. To distinguish
a structural failure from a finite-horizon effect within this lognormal model,
we then increased all stage targets proportionally. Table~\ref{tab:heavy-long}
shows a clear concentration pattern. At the largest well-replicated horizon,
49 of 50 runs select the reference master-grid optimizer; the pointwise 95\%
Wilson interval for the corresponding replication frequency is
$[0.8950,0.9965]$, conditional on treating the reference target as fixed.
The first three rows use cumulative per-action stage targets
$(60,200,600,2000)$, $(120,400,1200,4000)$, and
$(240,800,2400,8000)$, with 30, 30, and 16 replications and seed ranges
$505000$--$505029$, $510000$--$510029$, and
$520000$--$520015$. The last two rows use targets
$(480,1600,4800,16000)$ and $(960,3200,9600,32000)$, 50 replications each,
and seeds $9400000$--$9400049$ and $9500000$--$9500049$. The saved
50-replication rows do not match the bundled driver's request for 100
replications or its seed formulas. The provenance package therefore supplies
a wrapper reconstructed from the saved rows and labels it as reconstructed,
not as the historical generating driver.
The uneven replication counts $30,30,16,50,50$ reflect successive experimental
extensions and the archived outputs that remain available; they were not chosen
to give equal precision or power across horizons.

\begin{table}[ht]
\centering
\scriptsize
\setlength{\tabcolsep}{3pt}
\caption{Long-horizon Top-$K$ identification under nonpreemptive lognormal
service with SCV $=4$. The mean gap is relative to the archived reference
master-grid minimum and is shown with a pointwise 95\% Student-$t$ interval
(clipped at the natural zero boundary for display). Counts are followed by
pointwise 95\% Wilson intervals. All intervals condition on the archived
reference surface.}
\label{tab:heavy-long}
\begin{tabular}{@{}rrlll@{}}
\toprule
Mean time & Rep. & Mean reference gap [CI] & Reference optimizer selected [CI] & Guardian survives [CI] \\
\midrule
$1.47\times10^5$ & 30 & $.0592\ [.0387,.0798]$ & $13/30\ [.274,.608]$ & $18/30\ [.423,.754]$ \\
$2.95\times10^5$ & 30 & $.0461\ [.0268,.0654]$ & $16/30\ [.361,.698]$ & $20/30\ [.488,.808]$ \\
$5.94\times10^5$ & 16 & $.0419\ [.0112,.0726]$ & $10/16\ [.386,.815]$ & $12/16\ [.505,.898]$ \\
$1.20\times10^6$ & 50 & $.0169\ [.00641,.0273]$ & $41/50\ [.692,.902]$ & $42/50\ [.715,.917]$ \\
$2.46\times10^6$ & 50 & $.00178\ [0,.00536]$ & $49/50\ [.895,.996]$ & $49/50\ [.895,.996]$ \\
\bottomrule
\end{tabular}
\end{table}

\begin{figure}[ht]
\centering
\includegraphics[width=0.72\textwidth]{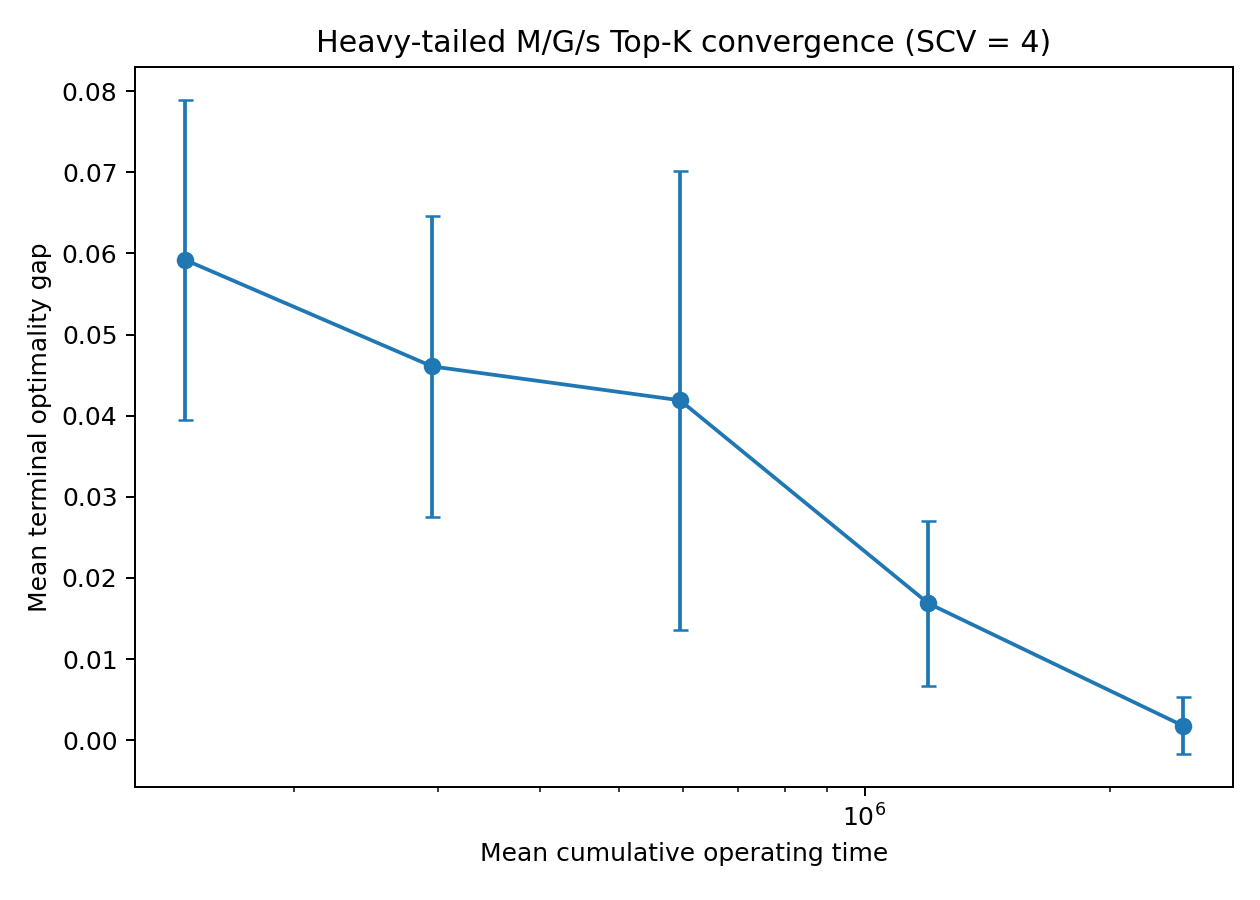}
\caption{Heavy-tailed $M/G/s$ Top-$K$ convergence as the operating horizon is
increased. The error bars are the pointwise normal approximations
$\overline\Delta\pm1.96\,\mathrm{SE}(\overline\Delta)$ computed across the
replications in Table~\ref{tab:heavy-long}; they are not simultaneous and
condition on the archived reference surface.}
\label{fig:heavy-convergence}
\end{figure}

A separate fixed-action diagnostic examines the queue-length tail at
$(s,p)=(5,3.9375)$.  The retained simulator routine uses the continuous-time
occupation estimator
\[
 \widehat p_{r,n}=\frac{1}{T_r}\int_0^{T_r}
 \mathbf 1\{Q_r(t)\ge n\}\,dt
\]
in replication $r$.  The archived final rows are labelled as 30,000 fixed-$D$
measurement cycles and contain 12 replications per service distribution, with
mean measurement duration about $2.775\times10^5$ time units per replication.
The historical wrapper, seed rule, and exact burn-in length were not recovered.

For the exponential row an exact benchmark is available.  At this action,
\[
 \rho=\frac{\lambda(3.9375)}{5(1.75)}=0.6177552449,
 \qquad C=0.2587030709,
\]
where $C$ is the Erlang-C delay probability, and for total number in system
\[
 \Pr\{Q\ge n\}=C\rho^{,n-5},\qquad n\ge5.
\]
Table~\ref{tab:heavy-tail-diagnostic} places these exact values beside the
archived means.  The brackets are pointwise 95\% Student-$t$ intervals across
the 12 replication estimates, treating the archived replications as
independent; lower endpoints are clipped at zero only for display.  These
intervals are not simultaneous, and the missing seed and burn-in record limits
their interpretation.
\begin{table}[ht]
\centering
\scriptsize
\setlength{\tabcolsep}{3.2pt}
\caption{Fixed-action queue-length tail diagnostic at
$(s,p)=(5,3.9375)$.  Archived entries are means over 12 replications with
30,000 measurement cycles each, followed by pointwise 95\% replication-level
Student-$t$ intervals.}
\label{tab:heavy-tail-diagnostic}
\begin{tabular}{@{}lccc@{}}
\toprule
Series & $P(Q\ge10)$ & $P(Q\ge20)$ & $P(Q\ge30)$ \\
\midrule
Exact exponential
& $0.0232747$ & $0.000188386$ & $1.52480\times10^{-6}$ \\
Archived exponential
& $.023309\ [.023009,.023609]$
& $.0001922\ [.0001726,.0002117]$
& $1.87\times10^{-7}\ [0,4.93\times10^{-7}]$ \\
Lognormal, SCV $=4$
& $.05446\ [.05395,.05497]$
& $.006490\ [.006195,.006786]$
& $.001246\ [.001102,.001391]$ \\
Lognormal, SCV $=9$
& $.08192\ [.08099,.08285]$
& $.02009\ [.01948,.02069]$
& $.007002\ [.006612,.007393]$ \\
\bottomrule
\end{tabular}
\end{table}
The archived exponential estimates agree well with the exact probabilities at
levels 10 and 20.  Level 30 is under-resolved: only 3 of the 12 exponential
replications record positive occupation above that level.  Their total is
$0.622$ time units over $3.329\times10^6$ measured time units, whereas the
stationary exact probability corresponds to about $5.08$ expected
occupation-time units over the same duration.  The archived mean is about
eight times smaller than the exact probability, so we do not use it as
validation at this threshold.

For the local ratio $P(Q\ge q+1)/P(Q\ge q)$ at $q=20$, the exact exponential
value is $\rho=0.617755$.  The archived replication means and pointwise 95\%
intervals are $0.5996\ [0.5753,0.6239]$, $0.8363\ [0.8303,0.8423]$, and
$0.8897\ [0.8867,0.8927]$ for exponential, lognormal-SCV-4, and
lognormal-SCV-9 service, respectively.  Thus the exponential ratio estimate
contains its exact value, while the two lognormal ratios remain much closer to
one and well separated under this replication-level calculation.  This
supports a qualitative slower-tail-decay interpretation at the fixed action,
not a causal attribution of the separate ranking experiment.

The general theory in Section~\ref{sec:theory} does not require the primitive
queueing state to be a CTMC; the selected-cycle route is formulated through
polynomial moment and boundary-growth conditions. We do not, however, claim
here that the strongest common-Lyapunov sufficient condition has been verified
for the lognormal $M/G/s$ model. We therefore interpret this experiment as a
numerical robustness test beyond the easiest analytically verifiable cases.

\FloatBarrier
\subsection{Comparison in a gradient-favorable regime}
\label{subsec:goliq-benchmark}

This final benchmark is intentionally unfavorable to a ``one method fits all''
interpretation. It asks how much is gained when a smooth model admits reliable,
purpose-built queueing sensitivities. We close the substantive comparisons
with the joint pricing--staffing $M/M/1$ example of Chen, Liu, and Hong
\cite{ChenLiuHong2023}. With the common demand function, holding-cost
coefficient $h_0=1$, and staffing cost $0.1\mu^2$, its objective is
\[
f(\mu,p)=\frac{\lambda(p)}{\mu-\lambda(p)}
+0.1\mu^2-p\lambda(p).
\]
On the comparison rectangle
$[6.7,12]\times[3.7,7.5]$, the exact optimum is approximately
\[
(\mu^*,p^*)=(7.10311,4.02337),
\qquad f^*=-13.12610.
\]
GOLiQ starts at $(12,7.5)$ and uses the published schedules
$\eta_k=1/k$ and $D_k=\lceil10+10\log k\rceil$, rounding the service-entry
target up to an integer as in the archived source. For Top-$K$, we use the hierarchy
\[
3\times3\to5\times5\to9\times9\to17\times17\to33\times33
\]
with $D=50$ and $K=(1,2,4,4)$. The recovered five-level source uses closed
half-parent-mesh boxes, as in \eqref{eq:mm1-child-map}, but with the
coordinate-specific meshes $5.3/(n_\ell-1)$ and $3.8/(n_\ell-1)$.
The low, medium, and large comparisons use cumulative per-action cycle
targets $(4,8,14,20,26)$, $(10,20,35,50,65)$, and
$(20,40,70,100,130)$, respectively.
An independent exhaustive audit reproduces its pruning-family sizes
$1,4,26,6667$ and positive uniform margins approximately
$4.87323137$, $3.29508491$, $1.15551954$, and $0.23064149$.
One level-1 active set preserves a guardian automatically; the finite
minimum there is taken over the other three sets.
Rational interval bounds certify their positivity and the unique master-grid
optimizer $(7.03125,4.05625)$. The implementation's comparison tolerance
preserves the intended closed-box child sets.
This grid optimum has cost $f_L^*\approx-13.12248055$, giving the
discretization gap $f_L^*-f^*\approx0.00361515$ in
Table~\ref{tab:analytic-benchmarks}. It is the smallest terminal gap
relative to the full-domain optimum attainable by a recommendation on
this master grid.

This is a hierarchy certificate. In the recovered source, the direct-SAA
cycle counts $D$ service entries and has cost
\[
G_t=\int_{\text{cycle }t}Q(u)\,du+0.1\mu_t^2T_t-p_tJ_t,
\]
where $J_t$ is the random number of external arrivals during the cycle,
which need not equal $D$. A service entry occurs at an arrival to an empty
system or at a departure followed by immediate refill. The cycle ends
after the $D$th such entry; an inherited in-service customer and a capacity
change are not counted as new entries. This continuous-time Markov chain
uses total customer count $Q$, rather than the workload integral of the
first benchmark. It is distinct from the fixed-arrival workload construction of
Proposition~\ref{prop:mm1-complete-verification}; the sampling and moment
hypotheses for this implementation have not been verified by that proposition.

The archived GOLiQ comparator retains an already scheduled arrival and its
stored price across a policy update, following
\cite[Section~3.3]{ChenLiuHong2023}; its updated price applies after that
arrival. Direct Top-$K$ instead applies its new price immediately to
subsequent arrivals. The two methods share a fixed-action objective
but use different transient price-switching conventions. The numerical
effect of that difference has not been isolated.

GOLiQ's update uses completed waiting times observed at service entry
and busy-period ages recorded at the corresponding arrivals, together
with the known demand function and its derivative.
The recovered code discards the first $\lfloor0.25D_k\rfloor$ customer
observations from its gradient average. Customer-count and staffing
integrals and realized revenue are recorded for evaluation, not supplied
to that gradient. In particular, evaluation revenue is the sum of prices
actually charged to arrivals, including the stored price of a pending
arrival; all elapsed time and costs are counted despite gradient trimming.
For pseudo-regret, the integrand in \eqref{eq:numerical-regret-definitions}
uses the stationary objective at the current action label. It therefore does
not reproduce the transient revenue effect of a pending arrival carrying its
stored earlier price; that effect is included in realized cost.

Table~\ref{tab:goliq} uses service-entry budgets matched exactly within every
paired replication. Elapsed operating times remain random and are not
matched. An audit of the archived driver found that its Python-level seed did
not seed the Numba-compiled Top-$K$ simulator, so Top-$K$ inherited Numba
state from the preceding GOLiQ run. The corrected driver seeds the Top-$K$
action order, Top-$K$ simulator, and GOLiQ explicitly and independently, then
reruns all 300 pairs. In
this gradient-favorable regime, GOLiQ uses smooth local sensitivity information
very efficiently and achieves much smaller terminal gaps as well as lower
mean pseudo-regret and realized-regret rates. The comparison therefore serves as a boundary-of-scope
check: when reliable model-specific sensitivity structure is available, it can
be substantially more efficient to exploit that structure than to replace it
with direct ranking.

\begin{table}[ht]
\centering
\small
\caption{Corrected independently seeded Top-$K$ and GOLiQ comparison on the
smooth joint pricing--staffing benchmark. Each row is based on 100 paired
replications with exactly matched service-entry counts.}
\label{tab:goliq}
\begin{tabular}{lrrrr}
\toprule
Method / budget & Mean service entries & Mean gap & Median gap & $P\{\text{gap}\le0.10\}$ \\
\midrule
GOLiQ, low      & 57,755  & 0.0452 & 0.0298 & 0.87 \\
Top-$K$, low    & 57,755  & 0.1952 & 0.1109 & 0.42 \\
GOLiQ, medium   & 145,690 & 0.0145 & 0.00854& 1.00 \\
Top-$K$, medium & 145,690 & 0.1389 & 0.0896 & 0.53 \\
GOLiQ, large    & 289,495 & 0.00827& 0.00722& 1.00 \\
Top-$K$, large  & 289,495 & 0.1375 & 0.1024 & 0.49 \\
\bottomrule
\end{tabular}
\end{table}

Table~\ref{tab:goliq-regret} reports both resource measures and the two
learning-cost rates. The common service-entry count varies across pairs because
the Top-$K$ final active set is random, but is identical for the two methods
within each pair. Top-$K$ takes more elapsed time at the same entry count in
these samples, so raw cumulative regret would combine policy loss with a
different clock horizon. We therefore compare $R^{\mathrm{ps}}(H)/H$ and
$R^{\mathrm{real}}(H)/H$, while displaying clock time explicitly. Parentheses
for entries and time contain sample standard deviations; parentheses for the
two rates contain Monte Carlo standard errors.

\begin{table}[ht]
\centering
\scriptsize
\setlength{\tabcolsep}{3pt}
\caption{Resource and learning-cost accounting for the corrected
Top-$K$--GOLiQ experiment (100 paired replications per budget).}
\label{tab:goliq-regret}
\begin{tabular}{@{}llrrrr@{}}
\toprule
Budget & Method & Entries: mean (SD) & Time: mean (SD)
& $R^{\mathrm{ps}}/H$: mean (SE) & $R^{\mathrm{real}}/H$: mean (SE) \\
\midrule
Low & Top-$K$ & $57{,}755\ (5{,}188)$ & $13{,}154\ (1{,}122)$
& $3.8197\ (0.0261)$ & $3.7963\ (0.0264)$ \\
    & GOLiQ & $57{,}755\ (5{,}188)$ & $11{,}361\ (1{,}175)$
& $0.5674\ (0.0596)$ & $0.5551\ (0.0605)$ \\
\addlinespace
Medium & Top-$K$ & $145{,}690\ (11{,}954)$ & $33{,}151\ (2{,}540)$
& $3.8139\ (0.0240)$ & $3.8061\ (0.0246)$ \\
       & GOLiQ & $145{,}690\ (11{,}954)$ & $28{,}176\ (2{,}348)$
& $0.2264\ (0.0171)$ & $0.2251\ (0.0178)$ \\
\addlinespace
Large & Top-$K$ & $289{,}495\ (15{,}602)$ & $66{,}022\ (3{,}340)$
& $3.7809\ (0.0165)$ & $3.7755\ (0.0173)$ \\
      & GOLiQ & $289{,}495\ (15{,}602)$ & $55{,}726\ (3{,}273)$
& $0.1127\ (0.0086)$ & $0.1091\ (0.0097)$ \\
\bottomrule
\end{tabular}
\end{table}

As a descriptive uncertainty check, the mean paired Top-$K$ minus GOLiQ
pseudo-regret-rate differences are $3.252$, $3.588$, and $3.668$ at the low,
medium, and large budgets. Their pointwise $95\%$ $t$ intervals are
$[3.130,3.374]$, $[3.534,3.642]$, and $[3.630,3.706]$, respectively.
Thus the claim of lower GOLiQ pseudo-regret rate is supported by the retained
samples. It remains specific to the event-matched design and does not assert
equal-clock-time efficiency or a general algorithmic ordering.

\FloatBarrier
\subsection{Computational specification and provenance}
\label{subsec:computational-record}

Table~\ref{tab:computational-record} consolidates the experimental record.
``Full'' means that source or deterministic calculations, inputs, randomization,
raw output, and aggregation are sufficient to reproduce or independently
check the displayed object. ``Partial'' means that saved raw output is
traceable but some historical driver, seed, or burn-in information is absent.
``Archival'' means that the recovered record is limited to a historical
summary or figure and validation material. The accompanying machine-readable
manifest gives filenames, exact seed ranges, stage targets, switching and tie
rules, and the known gaps for each object; its validation script reaggregates
the retained raw data and checks the manuscript figure hashes.

\begin{table}[ht]
\centering
\footnotesize
\setlength{\tabcolsep}{4pt}
\caption{Computational specification and provenance of the numerical study.}
\label{tab:computational-record}
\begin{tabular}{@{}>{\raggedright\arraybackslash}p{0.19\textwidth}
>{\raggedright\arraybackslash}p{0.48\textwidth}
>{\raggedright\arraybackslash}p{0.24\textwidth}@{}}
\toprule
Object & Replication and budget specification & Record \\
\midrule
$M/M/1$ validation
& 20 replications, 2,000 cycles each, $D=50$; separate 5,000-cycle duration path.
& Full. \\
\addlinespace
$M/M/1$ Top-$K$/LiQUAR
& Top-$K$: 100/60/30 replications and recovered four-stage targets. Corrected LiQUAR: 100/60/30 baseline replications plus three 30-run sensitivity settings.
& Top-$K$ partial; corrected LiQUAR full in N3 package. \\
\addlinespace
Light-tailed $M/M/s$
& 200 replications per method and budget; frozen allocations; independent
ordering and simulation seeds; 1,800 retained raw rows.
& Full canonical replacement: configuration, manifest, raw rows, aggregation,
validators, and exact replay retained. \\
\addlinespace
Lognormal $M/G/s$
& 30/30/16/50/50 replications with all target vectors and seeds recorded; 12-replication tail diagnostic.
& Main runs full/partial as specified; tail diagnostic partial, with exact exponential baseline and replication-level intervals added. \\
\addlinespace
Top-$K$/GOLiQ
& 100 paired replications per budget; service entries matched exactly; three independent explicit seed streams.
& Full corrected run; replaces coupled-stream archive. \\
\bottomrule
\end{tabular}
\end{table}

The classifications concern computational traceability, not theoretical
verification. For example, the full canonical light-tailed record makes its
reported summaries reproducible but does not by itself prove a general
algorithmic ordering. Likewise, a full simulation record would not remove the
reference-value uncertainty in the heavy-tailed model.

\FloatBarrier
\subsection{Summary of the numerical evidence}

The experiments support a deliberately qualified conclusion. When the goal is
accurate terminal identification from direct queue operation, Top-$K$ can be
stronger than the corrected LiQUAR comparator in the reported
exact-workload samples, although LiQUAR has lower median cumulative
learning cost at the large budget. In the canonical mixed-integer staffing
experiment, Top-$K$ has lower mean terminal gaps and higher optimizer-selection
frequencies than Top-1 and flat SAA at each approximately clock-matched budget;
all six reported independent mean-gap contrasts exclude zero. In the lognormal benchmark, rankings
stabilize only at longer horizons and the identification effect remains visible.
Because objective geometry and $K$ also change, this experiment does not isolate
a causal service-tail effect. The final GOLiQ comparison also marks an important
boundary: when smooth, reliable model-specific sensitivity information is
available, a specialized gradient method can be much more efficient.

\section{Discussion}
\label{sec:discussion}

\paragraph{Where direct Top-$K$ is most useful.}
The numerical evidence argues against viewing direct SAA as a replacement for
all queue-learning methods. Direct Top-$K$ is most attractive when objective
observations are easy to accumulate but sensitivity formulas are inconvenient,
when controls are discrete or mixed, or when the principal objective is
high-quality terminal policy identification rather than minimum cumulative
learning loss. The final GOLiQ comparison complements these cases by showing
the value of exploiting a reliable model-specific gradient on a smooth
problem.

\paragraph{Two roles of Top-$K$.}
The experiments reveal two distinct reasons to retain more than one parent.
First, coarse-grid objective ranking can distort descendant quality: a guardian
can be outranked even under exact objective values, as in the original
$M/M/1$ hierarchy, and under the estimated reference values in the
heavy-tailed $M/G/s$ example. Second, even when
exact-objective Top-1 would follow the correct path, finite-sample SAA noise can
eliminate it; the mixed-staffing $M/M/s$ benchmark illustrates this statistical
protection role. Proposition~\ref{prop:topk-margin-comparison} distinguishes
these mechanisms at a common active set, and the guardian theorem propagates
retention through stages under its full-family separation condition. This
does not establish dominance of Top-$K$ over Top-1 at a fixed total budget.

\paragraph{Transient effects and heavy tails.}
The direct-SAA decomposition isolates accumulated conditional transient bias
from martingale fluctuation. Contiguous action blocks are helpful because
bias-function terms telescope internally, leaving only block boundaries. The
lognormal experiment suggests that heavy-tailed service can make these finite
horizons practically much more demanding: rare long services create
long-lived congestion and slow stabilization of rankings. The long-horizon
experiment shows that this need not prevent accurate identification, but it
raises an important theoretical question about verifying weaker polynomial
stability conditions without a common geometric-drift assumption. This
interpretation is supported at the fixed diagnostic action by the exact-baseline
comparison around level 20, although the archived exponential estimate at
level 30 is too sparse for validation.  The cross-benchmark interpretation
remains suggestive because the objective
geometry and retention vector are not held fixed.

\paragraph{Terminal identification versus cumulative regret.}
The finite-sample and asymptotic theorems target the terminal recommendation, which is close
to a simple-regret objective. Cumulative regret charges the exploration blocks
themselves. In the corrected large-budget LiQUAR sensitivity study,
shorter periods reduce median cumulative learning cost but increase
the median terminal gap relative to the corrected baseline. This
illustrates why the two objectives need to be reported separately. A cumulative
regret analysis for blocked CTF is therefore a useful extension rather than an
implicit consequence of the present identification theory.

\paragraph{Scalability and limitations.}
A finite hierarchy makes integer and categorical controls natural and avoids a
projection or nonlinear solver at each update, but the number of candidate
actions can still grow rapidly with dimension. Top-$K$ reduces the number of
actions sampled deeply, yet high-dimensional problems may require sparse
coordinate refinement, problem-specific hierarchies, or local borrowing across
nearby actions. Guardian separation is also a genuine structural condition:
if every parent leading to a master-grid optimizer is ranked below
$K_\ell$ competing regions even under the true objective, additional sampling
cannot repair deterministic mis-refinement.

\paragraph{Tuning and verification still needed.}
A practical data-driven retention rule would need to balance the risk of
discarding an optimal region against the cost of sampling its descendants.
The present fixed-$K$ analysis identifies this tradeoff but does not solve it.
Full guardian-family audits are now available for the exact-formula
hierarchies documented above. The light-tailed staffing comparison now has a
canonical shared driver and a fully retained simulation record; that
computational validation remains distinct from a theorem establishing
equal-resource or universal superiority. The exhaustive lognormal reference
audit now includes simultaneous Monte Carlo-error sensitivity bounds, but
missing transient-bias and historical run metadata still preclude a certificate
for true-objective margins.
Proposition~\ref{prop:mm1-complete-verification} completes the common
Lyapunov and weighted cycle-reward verification for the specified
$M/M/1$ workload model. Extending that verification to other queueing
models remains substantive work; the duration bound alone is not a
stability argument. These distinctions separate the conditional
theorems, the current numerical evidence, and the remaining validation tasks.

\paragraph{Further directions.}
Several extensions are natural. The master grid may be refined with the total
budget if estimation errors vanish relative to shrinking guardian margins.
Quantitative guarantees under delayed or imperfect cost observation
require control of the accumulated observation error at the actual
ranking times; Remark~\ref{rem:stagewise-observation-error} gives the
corresponding transfer bound. Heavy-tailed
models motivate direct verification of selected-cycle polynomial moment
conditions. Finally, in large service systems it would be useful to compare
direct learned integer staffing with square-root staffing prescriptions in the
Halfin--Whitt regime, thereby quantifying when online learning adds value over
classical asymptotic rules.

\section{Conclusion}
\label{sec:conclusion}

This paper develops a direct sample-average route to queueing control on a
single persistent operating trajectory. Fixed-event-count cycles provide the
cost-rate estimator; the transient-bias/martingale decomposition explains its
error; and regenerative boundary representations show why contiguous action
blocks can control inherited-state effects without resetting the queue.

Top-$K$ coarse-to-fine search turns that estimator into a structured decision
procedure. Guardian retention connects finite-stage ranking accuracy to
master-grid identification and distinguishes two reasons to retain multiple
parents: deterministic coarse-grid distortion and finite-sample mis-ranking.
The exact-objective Top-1 failure in the four-level $M/M/1$ hierarchy gives a
concrete structural example of the first mechanism. For the specified
persistent-workload $M/M/1$ implementation, the paper verifies the common
drift, bias representation, cycle-reward moments, duration control, and
completed-stage sampling assumptions needed by the theory. Other queueing
models require their own verification.

The numerical results separate terminal identification from the operating cost
of learning. Top-$K$ has a terminal-gap advantage in the reported exact-workload
LiQUAR comparison, whereas corrected LiQUAR has lower median cumulative
learning cost at the large budget. The canonical mixed-integer staffing
experiment shows lower mean terminal gaps and higher optimizer-selection
frequencies for Top-$K$ than for Top-1 and flat SAA across three approximately
clock-matched budgets, illustrating the value of protecting several refinement
paths in this design. The lognormal-service experiment
shows slower finite-horizon stabilization and is interpreted as a robustness
test rather than a completed heavy-tail verification. The GOLiQ comparison
marks the opposite boundary: a specialized gradient method can be much more
efficient when reliable model-specific sensitivities are available.

Accordingly, the contribution is not a universal replacement for gradient or
bandit methods. It is a complementary framework for problems in which realized
objective costs are observable, the queue persists across control changes, and
the operational decision space makes direct ranking and structured refinement
more credible than a model-specific gradient construction.

\medskip
\noindent\chg{\textbf{Funding:} This research was supported in part by the Natural Sciences and Engineering Research Council of Canada [Grant RGPIN-2022-04915]. M. Nozohour Yazdi also acknowledges support from Carleton University.}

\end{document}